\documentclass[11pt]{article}

\usepackage{hyperref}
\usepackage{booktabs}
\usepackage{multirow}
\usepackage{subfigure}
\usepackage{graphicx}           
\usepackage{caption}
\usepackage{tikz}
\usepackage{tikz-qtree} 
\usetikzlibrary{matrix}
\usepackage{amsfonts}
\usepackage{mathdots}
\usepackage{amsmath}
\usepackage{amsthm}  
\usepackage{amssymb}
\usepackage{glossaries} 
\usepackage{mathtools}
\usepackage{enumerate} 
\usepackage{stmaryrd} 
\usepackage{xcolor} 
\usepackage{algorithmic}
\usepackage{algorithm}
\usepackage[margin=1in]{geometry} 
\usepackage{eqparbox}

\newcommand{\cdf}[1]{{#1}}

\DeclarePairedDelimiter\abs{\lvert}{\rvert}%
\DeclarePairedDelimiter\norm{\lVert}{\rVert}%

\newcommand{\ZZZ}{\mathcal{Z}}
\newcommand{\SSS}{\mathcal{S}}

\newcommand{\KSS}{K_{S_1 S_2}}

\newcommand{\iii}{\mathcal{I}}
\newcommand{\nys}{Nystr{\"o}m }
\newcommand{\alphahat}{\hat{\alpha}}
\newcommand{\betahat}{\hat{\beta}}

\newcommand{\dist}[1]{\text{dist}( #1 )}

\makeatletter
\let\oldabs\abs
\def\abs{\@ifstar{\oldabs}{\oldabs*}}
\let\oldnorm\norm
\def\norm{\@ifstar{\oldnorm}{\oldnorm*}}
\makeatother

\newtheorem{definition}{Definition}[section]
\newtheorem{theorem}{Theorem}[section]
\newtheorem{lemma}{Lemma}[section]
\newtheorem{proposition}{Proposition}[section]  
\newtheorem{corollary}{Corollary}[section]

\title{Data-Driven Linear Complexity Low-Rank Approximation of General Kernel Matrices: A Geometric Approach} 
\author{Difeng Cai\thanks{Department of Mathematics, Emory University, Atlanta, GA 30322}
\and Edmond Chow \thanks{School of Computational Science and Engineering, Georgia Institute of Technology, Atlanta, GA 30332}
\and Yuanzhe Xi$^{\ast}$}
\date{}

\begin{document}
\maketitle

\abstract{A general, {\em rectangular} kernel matrix may be defined as $K_{ij} = \kappa(x_i,y_j)$ where $\kappa(x,y)$ is a kernel function and where
$X=\{x_i\}_{i=1}^m$ and $Y=\{y_i\}_{i=1}^n$ are two sets of points.
In this paper, we seek a low-rank approximation to a
kernel matrix where the sets of points $X$ and $Y$ are large and are \cdf{arbitrarily distributed, such as away from each other, ``intermingled'', identical, etc.}
Such rectangular kernel matrices may arise, for example, in Gaussian
process regression where $X$ corresponds to the training data and
$Y$ corresponds to the test data.  In this case, the points are often
high-dimensional.  
Since the point sets are large, we must exploit the fact that the matrix
arises from a kernel function, and avoid forming the matrix, and thus
ruling out most algebraic techniques.  In particular, we seek methods
that can scale \cdf{linearly or nearly linearly with respect to the size of data for a fixed approximation rank}.  The main idea in this paper is to
{\em geometrically} select appropriate subsets of points to construct
a low rank approximation.  An analysis in this paper guides how this
selection should be performed.
}



\section{Introduction} 
\label{sec:intro}

Given a function $\kappa(x,y)$ and two sets of points
$X=\{x_i\}_{i=1}^m$ and $Y=\{y_i\}_{i=1}^n$, the $m$-by-$n$
matrix with entries
\begin{equation}
\label{eq:K}
   K_{ij} = \kappa(x_i,y_j),\quad x_i\in X, \;\; y_j\in Y
\end{equation}
and denoted by $K_{XY}$ is called a {\em kernel matrix} 
and $\kappa(x,y)$ is called a kernel function. 
Kernel matrices associated with various kernel functions arise in diverse computations such as those involving integral equations \cite{kress2013linear,hsiaowendlandbook,atkinson1967eig,rokhlinpotential,eigCMAM},
$N$-body simulations \cite{BarnesHut,GREENGARD1997280}, 
Gaussian processes
\cite{bishopbook,edmondgaussian2014},
and others \cite{vapnikbook,hack2015book,huan2018,huan2019,huan2020fast}.

One frequently encounters the problem of finding a low-rank factorization, exactly or approximately, of a kernel matrix. 
We first note that algebraic techniques
such as the singular value decomposition and some pseudoskeleton \cite{Tyrtyshnikov1996,maxvol2001,skeleton2011,kressner2020}
and CUR decompositions \cite{Mahoney697,guCUR2015} do not take advantage of the fact that a matrix is a \emph{kernel} matrix.
We further note that when the kernel function $\kappa(x,y)$ is smooth (but possibly singular at $x=y$) and the datasets $X,Y$ are well-separated, then
the corresponding kernel matrix $K_{XY}$ generally has low numerical rank and there exists a variety of efficient methods 
for finding the low-rank approximation (e.g., degenerate approximations of the kernel function \cite{rokhlinpotential,GREENGARD1997280,hack1989,anderson92,xiaobaiFMM,hackintroh2app,smash}
and proxy point methods \cite{Gillman2012,xing2020interpolative}).

In this paper, we seek a low-rank approximation to a kernel matrix where the sets of points $X$ and $Y$ are large and 
are \cdf{arbitrarily distributed, such as away from each other, ``intermingled'', identical, etc.}
Since the point sets are large, we must exploit the fact that the matrix arises from a kernel function, and avoid forming the matrix, and thus ruling out most algebraic techniques.
In particular, we seek methods \cdf{that can scale linearly or nearly linearly for a fixed rank}.
Such kernel matrices arise, for example, in Gaussian process regression where $X$ corresponds to the training data and $Y$ corresponds to the test data.
In this case, the points are often high-dimensional, which also rules out the use of any existing methods (e.g., degenerate approximations and proxy point methods) that are limited by the curse of dimensionality.

An existing method called adaptive cross approximation (ACA) \cite{bebendorf2000,bebendorf2003ACA}
is often suitable for our case.
ACA scales linearly with the number of points.
ACA corresponds to a pivoted partial LU factorization and only needs to compute matrix elements used in the
partial factorization. However, ACA may fail in some circumstances since it does not perform full pivoting \cite{HCA2005,darve2019}. We will numerically compare
our proposed method to ACA later in this paper.

The main idea in this paper is to {\em geometrically} select a subset of points
$S_1$ in $X$ and/or a subset of points $S_2$ in $Y$ to construct a low
rank approximation.  An analysis in this paper guides how this selection
should be performed. 

We analyze the use of these subsets of points to construct two forms of low-rank factorizations. The first is a two-sided form:
\begin{equation}
 K_{XY}\approx K_{XS_2}\KSS^+K_{S_1 Y},\quad S_1\subseteq X,\;\; S_2\subseteq Y,
 \label{eq:2sided}
\end{equation}
where $\KSS^+$ denotes the pseudoinverse of $K_{S_1 S_2}$.
This form is a CUR decomposition, except that we will treat $K_{XY}$ as a kernel matrix. 
Note that this form is similar to that of a
\nys factorization, except that a \nys factorization \cite{nys2001} expects the kernel matrix to be symmetric, with $Y=X$, since eigenvalues of the kernel matrix are implicitly being approximated in the \nys factorization.
The matrix $K_{XY}$ in \eqref{eq:2sided} is rectangular in general.

The second form of low-rank factorization that we study is the one-sided form of the interpolative decomposition \cite{ID2005} :
\begin{equation}
K_{XY}\approx UK_{\iii Y},\quad U=P \begin{bmatrix} I \\ G \end{bmatrix},
\label{eq:1sided}
\end{equation}
where $\iii\subseteq X$, $P$ is a permutation matrix, $I$ is an identity matrix and $G$ is a general dense matrix.
This form can be computed algebraically using the strong rank-revealing QR factorization \cite{rrqr96} \cdf{with the property that the $\max$-norm of $G$ is bounded by a prescribed constant larger than 1. However, this algebraic factorization requires the entire matrix $K_{XY}$ to be formed explicitly.}

Instead, it is common to algebraically compute the interpolative decomposition of the smaller matrix
\begin{equation}
K_{X S_2} \approx U K_{\iii S_2},
\quad U=P \begin{bmatrix} I \\ G \end{bmatrix},
\label{eq:1sided-intermediate}
\end{equation}
where $S_2 \subseteq Y$ or $S_2$ is an entirely different set of points altogether,
and then use $U$ and $\iii$ computed in (\ref{eq:1sided-intermediate}) for the approximation (\ref{eq:1sided}).
Examples of this approach can be found in \cite{HCA2005,Gillman2012,darve2019}.
In these approaches, the choice of $S_2$ is made analytically (e.g., Chebyshev points \cite{HCA2005,darve2019} or proxy surface points \cite{Gillman2012}) or algebraically (e.g., ACA) \cite{HCA2005}.
In this paper, for the one-sided approximation (\ref{eq:1sided}), we will analyze a {\em geometric} choice of the subset $S_2$.
After $S_2$ is chosen, the subset $\iii$ is selected by the algebraic interpolative decomposition via strong rank-revealing QR factorization.

\cdf{Low-rank methods based on subset selection are useful in improving the scalability of Gaussian process, often under the name of ``sparse Gaussian process''(cf. \cite{smola2000SGP,seeger2002SGP,snelson2005SGP}), where ``sparse'' refers to the fact that the selected subsets, for example, $S_1$, $S_2$ in \eqref{eq:2sided}, are much smaller than (thus sparsely distributed in) the original data sets.
Thus one application of the paper is the design of scalable Gaussian process.
}

This paper will show that the low-rank approximation
error in the maximum norm depends on the quantities
$\delta_{X,S_1}$ and/or
$\delta_{Y,S_2}$, where
\[
\delta_{Z,S}:=\max\limits_{x\in Z}\dist{x,S}
\]
measures the closeness between $Z$ and $S$.  In order for $\delta_{X,S_1}$
(or $\delta_{Y,S_2}$) to be small, points in $S_1$ (or $S_2$) should be
close to as many points in $X$ (or $Y$) as possible.  This implies that
selecting sample points that are evenly distributed over the entire
dataset can yield better approximations than, for example, choosing
clustered points in small regions that fail to capture the geometry of
the entire dataset.
A similar geometric selection can be used in a
version of skeletonized interpolation \cite{darveEndo} but has only
been studied in the case of well-separated sets of points.

Several known methods can be used to select $O(1)$ sample points that are evenly distributed over a dataset with a complexity that scales \emph{linearly} with the size of the dataset.
For example, {\em farthest point sampling} (FPS) \cite{FPS97} constructs a subset $S$ of $X$ by first initializing $S$ with one point and then sequentially adding the point in $X$ not in $S$ that is farthest from the current points in $S$.  
The complexity for selecting $r$ samples from $n$ points in $\mathbb{R}^d$ is $O(dr^2 n)$.  
FPS produces highly evenly distributed samples and is often used in mesh generation \cite{FPS06}, computer graphics \cite{FPS11}, etc., but primarily where the data are at most three dimensional.  
It has not previously been used for the low-rank compression of matrices or applied to high dimensional datasets.
Computationally, for high dimensional datasets, FPS can be potentially slow in practice due to its sequential nature.
One can combine FPS with uniform random sampling for faster speed, for example, by generating approximately 20\% of samples using FPS and 80\% using uniform random sampling.
As will be shown in Section \ref{sub:high}, the resulting \emph{mixed method} tends to yield an approximation that is less accurate than FPS and more accurate than random sampling.

Another method for selecting evenly distributed sample points is the {\em anchor net method} \cite{anchornet}.
This method was proposed for the efficient generation of landmark points for \nys approximations such that the resulting approximation is accurate and numerically stable.  
It leverages discrepancy theory to generate evenly-spaced samples and was shown in \cite{anchornet} to achieve better accuracy and robustness than uniform random sampling and $k$-means clustering for low-rank approximations.
The anchor net method has the optimal complexity $O(drn)$ for selecting $r$ points from $n$ points in $\mathbb{R}^d$ and is efficient for a wide range of problems from low to high dimensions.
However, the anchor net method has only been used for approximating symmetric kernel matrices and its performance for approximating general rectangular kernel matrices is as yet unknown.

Figure \ref{fig:pony} shows the $100$ samples obtained from FPS
and the anchor net method for a highly irregular dataset in two dimensions.
Results for uniform random sampling is also shown, which
does not generally produce a uniform distribution of points over the data.

\begin{figure}
\centering
\subfigure[Pony dataset]{
    \includegraphics[scale=0.23]{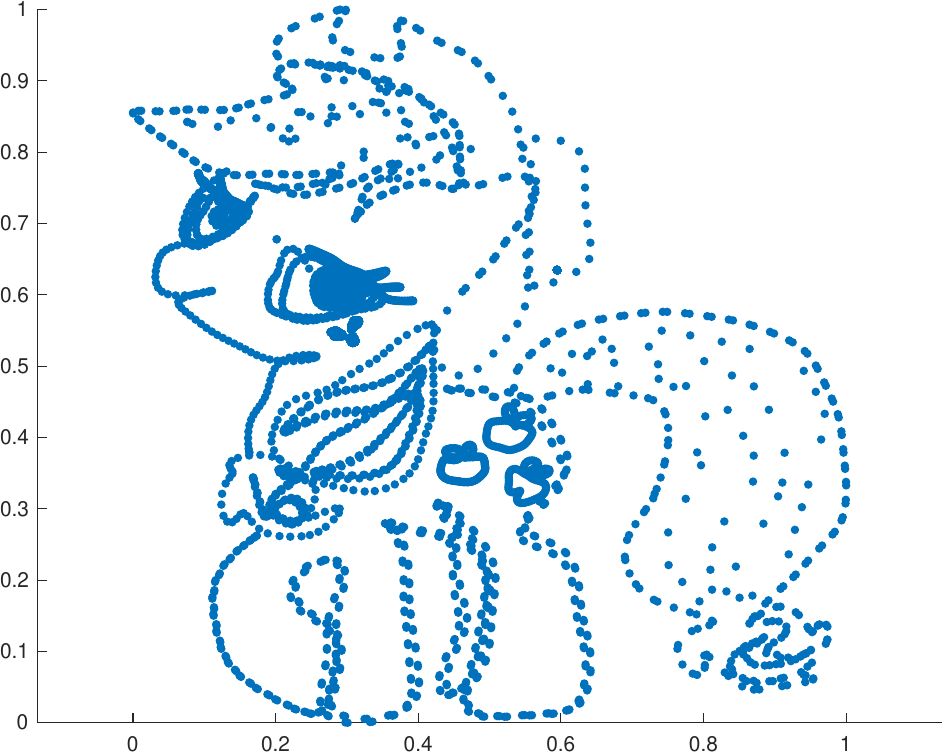}
}
\subfigure[Uniform random sampling]{
    \includegraphics[scale=0.23]{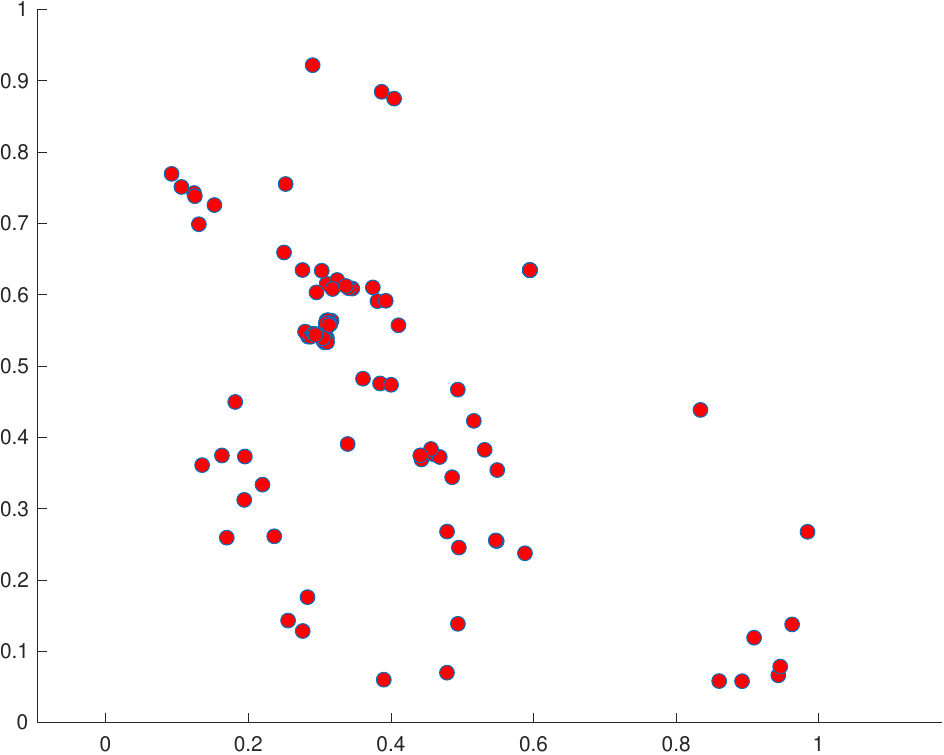}
}
\subfigure[Farthest point sampling]{
    \includegraphics[scale=0.23]{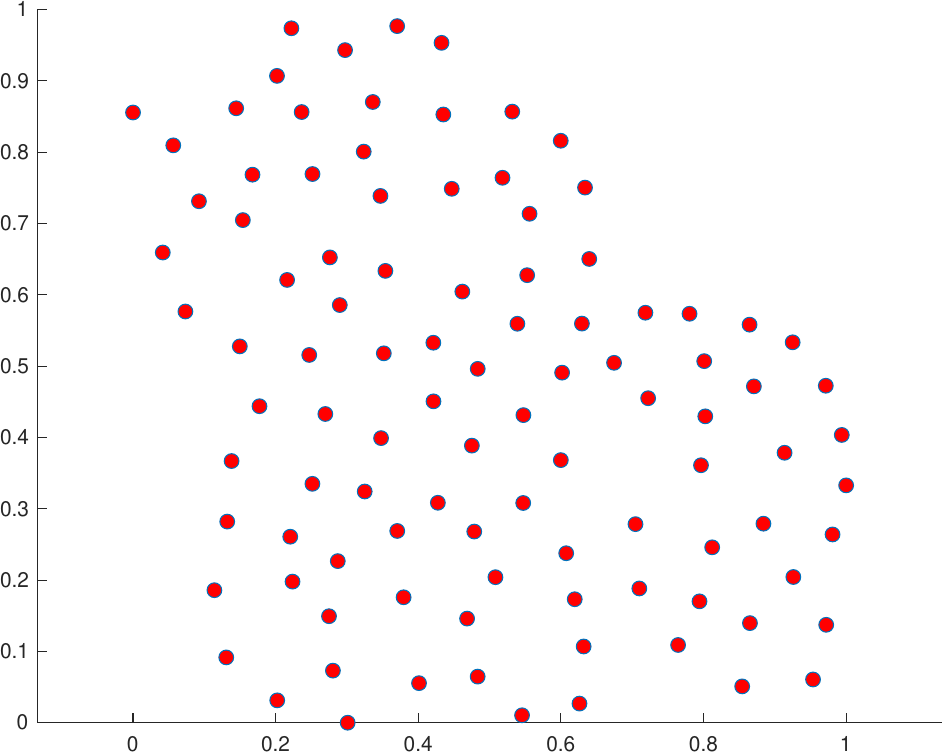}
}
\subfigure[Anchor net method]{
    \includegraphics[scale=0.23]{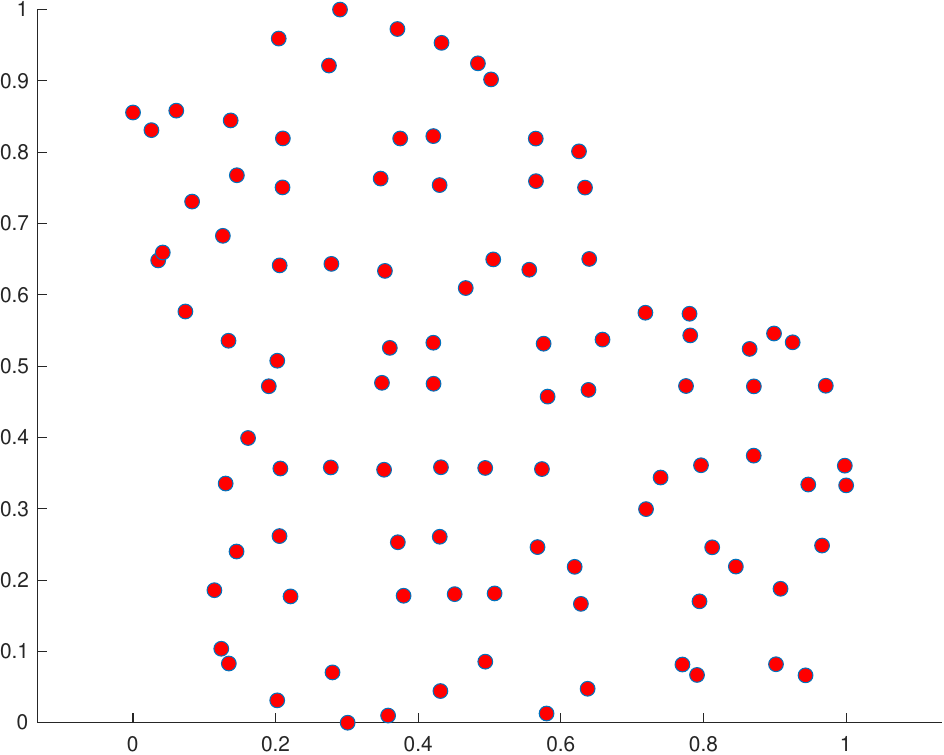}
}
\caption{Different geometric selection schemes for the Pony dataset.}
\label{fig:pony}
\end{figure}

In summary, we seek a linear \cdf{or nearly linear} complexity low-rank factorization approach for kernel matrices where the points may be intermingled and the points may be high-dimensional.
Some low-rank approximation techniques are matrix-based (e.g., ACA) and don't rely on knowing the specific kernel function or sets of points, except for assuming that the kernel function is smooth and gives rise to a kernel matrix $K_{XY}$ that is low rank.  Other techniques only need knowledge of the kernel function and bounding boxes for the sets of points, and do not depend on the points themselves when selecting the set $S_2$ in (\ref{eq:1sided-intermediate}), for example.  The method we propose is based on the sets $X$ and $Y$ and is independent of the kernel function.  
We thus call our method a {\em data-driven} method.
By choosing $S_2$ to be existing points rather than a new set of points that sample possibly high-dimensional space, the data-driven method is not limited by the curse of dimensionality.

Our proposed method relies on the geometric selection of the subsets $S_1 \subseteq X$ and/or $S_2 \subseteq Y$. 
We address the following questions:
(1) how does the data selection affect the low-rank approximation error?
(2) given two subsets with equal numbers of points, how can one tell which
one leads to a more accurate low-rank approximation?
(3) how can one perform the desired data selection efficiently?

The rest of the paper is organized as follows.
Section \ref{sec:two-sided} proposes the data-driven approach for efficiently computing the \emph{two-sided} factorization \eqref{eq:2sided}.  
Section \ref{sec:one-sided} similarly considers the \emph{one-sided} factorization \eqref{eq:1sided}.  
We will show that the one-sided factorization is more stable than the two-sided factorization.
The two-sided factorization, however, is slightly cheaper to compute than the one-sided factorization.
The results of numerical experiments are presented in Section \ref{sec:test}, and a conclusion given in Section \ref{sec:conclusion}.
Unless otherwise stated, all norms used in this paper are the 2-norm, denoted by $\norm{\cdot}$.
The Euclidean distance between $x,y\in \mathbb{R}^d$ is denoted by $|x-y|$.

\section{Two-sided low-rank kernel matrix approximation}
\label{sec:two-sided}
This section analyzes the data-driven geometric approach
for the two-sided low-rank approximation \eqref{eq:2sided}.
 
\subsection{Algorithm}
\label{sub:two-sided}
The two-sided factorization \eqref{eq:2sided} can be computed immediately once the subsets $S_1\subseteq X$ and $S_2\subseteq Y$ are determined.
The subsets can be computed in linear time with suitable geometric selection schemes. The full algorithm is given in Algorithm \ref{alg:fac2}. 
Depending on the specific geometric selection scheme, the total complexity of Algorithm \ref{alg:fac2} is $O(dr(m+n))$ for uniform random sampling and the anchor net method, or $O(dr^2(m+n))$ for farthest point sampling, where $r=\max(r_1,r_2)$ denotes the maximum sample size.
The choice of subsets has a strong impact on the low-rank approximation accuracy, robustness of the algorithm, as well as numerical stability, and thus the subset has to be chosen judiciously.
Theoretical guidance on geometric selection is provided in Section \ref{sub:errorfac2} via analyzing the approximation error of the two-sided factorization.
Experiments in Section \ref{sub:high} show that different geometric selections can yield dramatically different results for approximating the kernel matrix, with FPS and the anchor net method yielding the best results, which is consistent with our analysis.


\begin{algorithm}
    \caption{\it Data-driven two-sided compression of $K_{XY}$ with two sets of points $X, Y$}
    \label{alg:fac2}
    \emph{Input:} Datasets $X=\{x_{1},\dots,x_{m}\}$, $Y=\{y_1,\dots,y_n\}\subset \mathbb{R}^d$, kernel function $\kappa$, numbers of sample points $r_1,r_2$ for $X,Y$, respectively

    \emph{Output:} Approximation $K_{XY}\approx K_{XS_2}\KSS^+K_{S_1 Y}$ with $\operatorname{card}(S_1) = r_1$, $\operatorname{card}(S_2) = r_2$

        \begin{algorithmic}        
        \STATE Apply a linear complexity geometric selection algorithm to $X$ and $Y$ to generate $r_1$ and $r_2$ samples $S_1\subseteq X$ and $S_2\subseteq Y$, respectively
        \STATE Return $K_{XS_2},\KSS,K_{S_1 Y}$
    \end{algorithmic}
\end{algorithm}

\subsection{Error analysis for two-sided approximation}
\label{sub:errorfac2}
The goal of this section is to derive an error estimate of the two-sided approximation that is able to provide a straightforward geometric understanding of how the subsets $S_1, S_2$ affect the approximation accuracy. 
This analysis is independent of how the subsets $S_1, S_2$ are selected in Algorithm \ref{alg:fac2}.
To prepare for the derivation of the geometric estimates, in Section \ref{ssub:preparation}, we derive error bounds involving only submatrices of $K_{XY}$.
The geometric estimates are presented in Section \ref{ssub:Geometric}.

\subsubsection{Algebraic preparation} 
\label{ssub:preparation}
In order to estimate for the approximation error of \eqref{eq:2sided} for arbitrary subsets $S_1\subseteq X$ and $S_2\subseteq Y$, we first review one lemma from \cite[Lemma 3.1]{anchornet}, which is stated below.
\begin{lemma}
\label{lm:lm1}
    Assume $A$ is an $m$-by-$n$ matrix, $\alpha,\alphahat$ are $m$-by-1 vectors and $\beta,\betahat$ are $n$-by-1 vectors.
    Define $\epsilon_1(u):=\norm{\alphahat-\alpha}$ and $\epsilon_2:=\norm{\betahat-\beta}.$
    Then 
    \begin{equation}
    \label{eq:lm1}
        \abs{\alphahat^T A \betahat - \alpha^T A \beta} 
        \leq \norm{\alpha^T A}\cdot \epsilon_2 + \norm{A\beta}\cdot \epsilon_1(u)
        + \norm{A}\cdot \epsilon_1(u)\epsilon_2.
    \end{equation}
\end{lemma}

In the next theorem, we  derive the estimate for the entrywise approximation error of \eqref{eq:2sided}  at an arbitrary pair of points $(x,y)$. 
This can be viewed as an error estimate for the ``\emph{algebraic}'' separable approximation to the kernel function $\kappa(x,y)$.
\begin{theorem}
\label{thm:kxyerror}
Consider finite sets $X, Y\subset \mathbb{R}^d$ and a kernel function $\kappa(x,y)$ defined on $X\times Y$.
For any non-empty subsets 
$S_1 \subseteq X$ and $S_2 \subseteq Y$,
the entrywise error of the approximation in \eqref{eq:2sided}  satisfies
\begin{equation}
\label{eq:kxyerror}
\abs{\kappa(x,y)-K_{xS_2}\KSS^+K_{S_1 y}}
    \leq 
\min_{\substack{u\in S_1\\ v\in S_2}}\left(\abs{\kappa(x,y) - \kappa(u,v)} + \epsilon_1(u)+\epsilon_2(v)+\norm{\KSS^{+}} \epsilon_1(u)\epsilon_2(v)\right),
\end{equation}
where  $\epsilon_1(u) = \norm{K_{x S_2}-K_{u S_2}}$ and $\epsilon_2(v) = \norm{K_{S_1 y}-K_{S_1 v}}$.
\end{theorem}
\begin{proof}
Since
    $\KSS \KSS^+ \KSS=\KSS$, we have $\forall u\in S_1, v\in S_2$
\begin{equation}
\label{eq:kppuv}
    \kappa(u,v)=K_{u S_2}\KSS^+ K_{S_1 v}.
\end{equation}
For any $x\in X,\, y\in Y,\, u\in S_1,\, v\in S_2$, define the column vectors
\begin{equation*}
    \alpha := K_{u S_2}^T,\quad 
    \alphahat:= K_{x S_2}^T,\quad
    \beta :=K_{S_1 v},\quad
    \betahat:=K_{S_1 y}.
\end{equation*}
Then it is easy to see that $\epsilon_1(u) = \Vert\alphahat-\alpha\Vert$ and $\epsilon_2(v) = \Vert\betahat-\beta\Vert$.
With \eqref{eq:kppuv}, we obtain 
\begin{equation}
\label{eq:diffrewrite}
   \kappa(x,y)-K_{xS_2}\KSS^+K_{S_1 y}  = \kappa(x,y) - \alphahat^T \KSS^{+} \betahat 
    = (\kappa(x,y) - \kappa(u,v)) +( \alpha^T \KSS^{+} \beta - \alphahat^T \KSS^{+} \betahat),
\end{equation}
for any $u\in S_1, v\in S_2$.
According to Lemma \ref{lm:lm1}, we get
\begin{equation}
\label{eq:term2}
\begin{aligned}
        \abs{\alphahat^T \KSS^{+} \betahat - \alpha^T \KSS^{+} \beta} 
        &\leq \norm{\alpha^T \KSS^{+}} \epsilon_2(v) + \norm{\KSS^{+}\beta} \epsilon_1(u)
        + \norm{\KSS^{+}} \epsilon_1(u)\epsilon_2(v)\\
        &\leq \epsilon_2(v)+\epsilon_1(u)+\norm{\KSS^{+}} \epsilon_1(u)\epsilon_2(v).
\end{aligned}
\end{equation}
The last inequality in \eqref{eq:term2} follows from the fact that
$\alpha^T \KSS^{+}=K_{u S_2}\KSS^{+}$ is a row in $\KSS\KSS^+$ and $\KSS^{+}\beta=\KSS^{+}K_{S_1 v}$ is a column in $\KSS^+ \KSS$, and meanwhile
\[
    \norm{\KSS\KSS^{+}} = \norm{\KSS^{+}\KSS} = 1.
\]
We see from \eqref{eq:diffrewrite} and \eqref{eq:term2} that
\begin{equation}
\label{eq:syuvBound}
    \abs{\kappa(x,y) - \alphahat^T \KSS^{+} \betahat} 
    \leq \abs{\kappa(x,y) - \kappa(u,v)} + \epsilon_1(u)+\epsilon_2(v)+\norm{\KSS^{+}} \epsilon_1(u)\epsilon_2(v),\quad \forall u\in S_1, v\in S_2.
\end{equation}
Minimizing the upper bound in \eqref{eq:syuvBound} over all $u\in S_1, v\in S_2$ yields \eqref{eq:kxyerror}, which completes the proof.
\end{proof}

The entrywise estimate in Theorem \ref{thm:kxyerror} immediately leads to  a matrix max norm estimate, which is proved in the next theorem. 
\begin{theorem}
\label{thm:KXYerror}
Consider finite sets $X, Y\subset \mathbb{R}^d$ and kernel function $\kappa(x,y)$ defined on $X\times Y$.
For any non-empty subsets
$S_1 \subseteq X$ and $S_2 \subseteq Y$,
denote by $\mathcal{X} = X\times Y,\ \mathcal{S}=S_1\times S_2$. 
Then the approximation in \eqref{eq:2sided} satisfies the following estimate
\begin{equation}
\label{eq:KXYerror}
        \norm{ K_{XY} - K_{XS_2}\KSS^+K_{S_1 Y} }_{\max}
        \leq  
    \max_{\substack{x\in X\\ y\in Y}} 
\min_{\substack{u\in S_1\\ v\in S_2}}\left(\abs{\kappa(x,y) - \kappa(u,v)} + \epsilon_1(u)+\epsilon_2(v)+\norm{\KSS^{+}} \epsilon_1(u)\epsilon_2(v)\right),
\end{equation}
where  $\epsilon_1(u) = \norm{K_{x S_2}-K_{u S_2}}$ and $\epsilon_2(v) = \norm{K_{S_1 y}-K_{S_1 v}}$.
\end{theorem}
\begin{proof}
Taking maximum of both sides of \eqref{eq:kxyerror} over $x\in X, y\in Y$ yields \eqref{eq:KXYerror}.
\end{proof}

\cdf{Assuming $\kappa$ is Lipschitz continuous,} Theorems \ref{thm:kxyerror} and \ref{thm:KXYerror} imply that the bounds will be small if for any point $x\in X$ there is a point $u\in S_1$ nearby and for any point $y\in Y$ there is a point $v\in S_2$ nearby.
As a result, Theorems \ref{thm:kxyerror} and \ref{thm:KXYerror} indicate that $S_1$ and $S_2$ should be evenly distributed inside $X$ and $Y$ in order to achieve a small approximation error. 
This can be more easily identified when the special case $S_1=X$ or $S_2=Y$ is considered.
\begin{corollary}
\label{cor:KXYerror}
Let $X, Y\subset \mathbb{R}^d$ be finite sets and $\kappa(x,y)$ be defined on $X\times Y$. For any non-empty subsets $S_1  \subseteq X$ and $S_2 \subseteq Y$,
the following estimates hold
\begin{equation}
\label{eq:KXYerrorCor}
\begin{aligned}
        \norm{ K_{XY} - Q_{XS_2}^|K_{XY} }_{\max} &\leq  
\max_{\substack{x\in X\\ y\in Y}} \min_{v\in S_2} \left( \abs{\kappa(x,y)-\kappa(x,v)}+\norm{K_{Xy}-K_{Xv}} \right),\\
        \norm{ K_{XY} - K_{XY} Q_{S_1Y}^- }_{\max} &\leq  
\max_{\substack{x\in X\\ y\in Y}} \min_{u\in S_1} \left( \abs{\kappa(x,y)-\kappa(u,y)}+\norm{K_{xY}-K_{uY}} \right),
\end{aligned}
\end{equation}
where $Q_{XS_2}^|:=K_{XS_2}K_{XS_2}^+$, $Q_{S_1Y}^-:=K_{S_1Y}^+K_{S_1Y}$.
\end{corollary}
\begin{proof}
We only show the first inequality in \eqref{eq:KXYerrorCor} and the second one can be proved in a similar fashion.
Note that the first inequality in \eqref{eq:KXYerrorCor} is a special case of \eqref{eq:KXYerror} with $S_1=X$. 
In this case, in the upper bound of \eqref{eq:KXYerror},
the minimum over $u\in X$ is no greater than the value achieved by choosing $u=x$.
Hence we see that if $S_1=X$ and $u=x$, then $\epsilon_1=\norm{K_{xS_2}-K_{xS_2}}=0$
and \eqref{eq:KXYerror} becomes
\begin{equation*}
        \norm{ K_{XY} - Q_{XS_2}^|K_{X Y} }_{\max}
        \leq  
    \max_{\substack{x\in X\\ y\in Y}} 
\min_{v\in S_2}\left(\abs{\kappa(x,y) - \kappa(x,v)} + \norm{K_{Xy}-K_{Xv}}\right),
\end{equation*}
which is the first inequality in \eqref{eq:KXYerror}.
\end{proof}

\cdf{Assuming $\kappa$ is Lipschitz continuous,}
Corollary \ref{cor:KXYerror} further reveals the interconnection between the approximation accuracy and the geometry of sample points. Algebraically, $\norm{ K_{XY} - Q_{XS_2}^|K_{XY} }_{\max}$ and $\norm{ K_{XY} - K_{XY}Q_{S_1Y}^-  }_{\max}$ measure how well $K_{XS_2}$ and $K_{S_1Y}$ capture the column and row spaces of $K_{XY}$, respectively. Geometrically, the bound on the right-hand side of \eqref{eq:KXYerrorCor} will be small if $S_1$ and $S_2$ are able to capture the global geometry of $X$ and $Y$, respectively.

\subsubsection{Geometric estimates} 
\label{ssub:Geometric}
In the following, we reveal the geometric implication of the error bounds in Theorem \ref{thm:KXYerror} and Corollary \ref{cor:KXYerror} with the help of the so-called \emph{discrete Lipschitz constant} as defined below.
It is used to derive new error bounds that give a more straightforward interpretation of how the sets of landmark points $S_1$ and $S_2$ affect the accuracy of the approximation $K_{XY}\approx K_{X S_1} \KSS^+ K_{S_1 Y}$.

\begin{definition}[Discrete Lipschitz constant]
\label{def:LXY}
   Let $\kappa(x,y)$ be a function defined on $X\times Y$. Denote $\ZZZ=Z_1\times Z_2$, $\SSS=S_1\times S_2$, $W_1\times W_2$ as three non-empty subsets of $X\times Y$. 
    The discrete Lipschitz constants of $\kappa$ associated with these three subsets are defined by
    \begin{equation}
        \label{eq:LXY}
    \begin{aligned}
    L(\ZZZ,\SSS) &:= \min \{ C: |\kappa(x,y)-\kappa(u,v)|^2\leq C^2 (|x-u|^2+|y-v|^2) \;\; \forall (x,y)\in \ZZZ, (u,v)\in \SSS \},\\
    L(Z_2,S_2)_{W_1} &:= \min \{ C: |\kappa(x,y)-\kappa(x,v)|^2\leq C^2 |y-v|^2 \;\; \forall x\in W_1, y\in Z_2, v\in S_2\},\\
    L(Z_1,S_1)_{W_2} &:= \min \{ C: |\kappa(x,y)-\kappa(u,y)|^2\leq C^2 |x-u|^2 \;\; \forall  y\in W_2, x\in Z_1, u\in S_1\}.
    \end{aligned}
    \end{equation}
\end{definition}

Since $X, Y$ are finite sets, each minimum in \eqref{eq:LXY} exists. 
Note that in general $L(\ZZZ,\SSS)$ is \emph{not} the Lipschitz constant of $\kappa$ since we do \emph{not} assume $\kappa$ to be Lipschitz continuous or even defined outside {$X\times Y$}.
\cdf{
If $\kappa(x,y)$ is Lipschitz continuous in a region containing $X\times Y$ with Lipschitz constant $L$, then it is easy to see that $L(\ZZZ,\SSS)\leq L$, as stated in Proposition \ref{prop:L}.
\begin{proposition}
\label{prop:L}
   Let $\kappa(x,y)$ be a Lipschitz continuous function on a domain $D_1\times D_2$ with Lipschitz constant $L$. 
For any discrete subset $X\times Y\subset D_1\times D_2$, the discrete Lipschitz constants defined in \eqref{eq:LXY} are all smaller than or equal to $L$.
\end{proposition}
}

\cdf{
The discrete Lipschitz constants are introduced to make the result derived in this section applicable to general kernel functions with as few constraints as possible.
In many applications, the kernel functions are actually not only Lipschitz continuous but also smooth in the domain of interest.
Hence it is sufficient to use the Lipschitz constant.
}
For example, in machine learning and statistics, the Gaussian kernel $\exp(-\frac{|x-y|^2}{\sigma^2})$ is smooth; 
radial basis functions like $\sqrt{1+|x-y|^2}$ and $(1+|x-y|^2)^{-1/2}$ are smooth;
in potential theory, kernels like $\frac{1}{|x-y|}$ are smooth in $D_1\times D_2$ with well-separated domains $D_1$ and $D_2$, which is a key assumption in the fast multipole method \cite{rokhlinpotential,GREENGARD1997280,xiaobaiFMM} and hierarchical matrices in general \cite{h2lec,bebendorf2000,hack2015book,smashIPDPS}.

Using the discrete Lipschitz constant, 
we can show in the following that the low-rank approximation error bound depends on the geometric quantity
\begin{equation}
\label{eq:delta}
    \delta_{Z,S}:=\max_{x\in Z}\dist{x,S} \quad S\subseteq Z.
\end{equation}
The quantity $\delta_{Z,S}$ measures the closeness between $Z$ and $S$.
The smaller $\delta_{Z,S}$ is, the ``closer'' $S$ is to $Z$.
In fact, if $\delta_{Z,S}$ is small, then for any $x\in Z$, there exists a point in $S$ that is close to $x$.

%

We can now derive an error bound for \eqref{eq:2sided} in terms of the geometric quantities $\delta_{X,S_1}$, $\delta_{Y,S_2}$ for subsets $S_1\subseteq X$, $S_2\subseteq Y$, respectively.
The result is stated in Theorem \ref{thm:deltaKXYerror}.
\begin{theorem}
\label{thm:deltaKXYerror}
Let $X, Y\subset \mathbb{R}^d$ be finite sets and $\kappa(x,y)$ be a function defined on $X\times Y$.
For any non-empty subsets
$S_1 \subseteq X$ and $S_2 \subseteq Y$,
define $\mathcal{X} = X\times Y$, $\mathcal{S}=S_1\times S_2$. 
Then the following estimate holds
\begin{equation*}
        \norm{ K_{XY} - K_{XS_2}\KSS^+K_{S_1 Y} }_{\max}
\leq C_1\delta_{X,S_1} + C_2\delta_{Y,S_2} + C_3\delta_{X,S_1}\delta_{Y,S_2},
\end{equation*}
where 
\begin{equation}
\label{eq:C123}
\begin{aligned}
C_1&=L(\mathcal{X},\mathcal{S})+\sqrt{r_2} L(X,S_1)_{S_2},\\
C_2&=L(\mathcal{X},\mathcal{S})+\sqrt{r_1} L(Y,S_2)_{S_1},\\
C_3&=\norm{\KSS^+}\sqrt{r_1 r_2} L(X,S_1)_{S_2}L(Y,S_2)_{S_1},
\end{aligned}
\end{equation}
with $r_i=\text{card}(S_i)$.
\cdf{
Furthermore, if $\kappa(x,y)$ is Lipschitz continuous over $D_1\times D_2$ containing $X\times Y$ with Lipschitz constant $L$, then 
\begin{equation*}
        \norm{ K_{XY} - K_{XS_2}\KSS^+K_{S_1 Y} }_{\max}
\leq (1+\sqrt{r_2})L\delta_{X,S_1} + (1+\sqrt{r_1})L\delta_{Y,S_2} + \norm{\KSS^+}\sqrt{r_1 r_2}L^2\delta_{X,S_1}\delta_{Y,S_2}.
\end{equation*}
}
\end{theorem}
\begin{proof}
The result can be proved using Theorem \ref{thm:KXYerror} and the definition in \eqref{eq:LXY}.
First we estimate the terms in the upper bound in Theorem \ref{thm:KXYerror}.
The definition of Lipschitz constants in \eqref{eq:LXY} implies that 
\begin{equation}
\label{eq:kLip}
 \begin{aligned}
    \abs{\kappa(x,y)-\kappa(u,v)}&\leq L(\mathcal{X},\mathcal{S}) \left( |x-u|^2+|y-v|^2 \right)^{1/2}
    \leq L(\mathcal{X},\mathcal{S}) \left( |x-u|+|y-v| \right),\\
\epsilon_1(u)=\norm{K_{xS_2}-K_{uS_2}}&\leq \left( \sum_{v\in S_2} L(X,S_1)_{S_2}^2 |x-u|^2  \right) ^{1/2}
\leq \sqrt{r_2} L(X,S_1)_{S_2} |x-u|,\\
\epsilon_2(v)=\norm{K_{S_1y}-K_{S_1v}}&\leq \left( \sum_{u\in S_1} L(Y,S_2)_{S_1}^2 |y-v|^2  \right) ^{1/2}
\leq \sqrt{r_1} L(Y,S_2)_{S_1} |y-v|.
 \end{aligned}
\end{equation}
Define $C_1,C_2,C_3$ as in \eqref{eq:C123}.
The estimates in \eqref{eq:kLip}, which separate $(x,u)$ and $(y,v)$ into different terms, allow us to organize the upper bound in \eqref{eq:KXYerror} in terms of $|x-u|$ and $|y-v|$ and deduce that
\begin{equation*}
 \begin{aligned}
&\max_{\substack{x\in X\\ y\in Y}} 
\min_{\substack{u\in S_1\\ v\in S_2}}\left(\abs{\kappa(x,y) - \kappa(u,v)} + \epsilon_1(u)+\epsilon_2(v)+\norm{\KSS^{+}} \epsilon_1(u)\epsilon_2(v)\right)\\
&\leq 
\max_{\substack{x\in X\\ y\in Y}} 
\min_{\substack{u\in S_1\\ v\in S_2}}
\left( C_1|x-u| + C_2|y-v| + C_3|x-u||y-v| \right)\\
&=
\max_{\substack{x\in X\\ y\in Y}} 
\left( C_1\dist{x,S_1} + C_2\dist{y,S_2} + C_3\dist{x,S_1}\dist{y,S_2} \right)\\
&=
C_1\max_{x\in X}\dist{x,S_1} + C_2\max_{y\in Y}\dist{y,S_2} + C_3\max_{x\in X}\dist{x,S_1}\max_{y\in Y}\dist{y,S_2} \\
&=
C_1\delta_{X,S_1} + C_2\delta_{Y,S_2} + C_3\delta_{X,S_1}\delta_{Y,S_2}.
\end{aligned}
\end{equation*}
This, together Theorem \ref{thm:KXYerror}, completes the proof of the first inequality.
\cdf{The special case where $\kappa(x,y)$ is Lipschitz continuous follows immediately from the first inequality and Proposition \ref{prop:L}.}
\end{proof}

The estimate in Theorem \ref{thm:deltaKXYerror} implies that in order to obtain a good approximation, $S_1, S_2$ should be chosen such that $\delta_{X,S_1}, \delta_{Y,S_2}$ are small.
Geometrically, according to the definition of $\delta$ in \eqref{eq:delta}, this means that $S_1$ and $S_2$ should represent the geometry of $X$ and $Y$ as much as possible. 
In the context of integral equations, a recent analytical study \cite{bebendorf2020} also discussed the relationship between the approximation error of adaptive cross approximation (ACA) and the selected subsets measured by the geometric concept of \emph{fill distance} (cf. \cite{filldistance}), which reflects how well the subset spans the computational domain.
Both fill distance and $\delta$ in \eqref{eq:delta} provide similar geometric interpretations of the quality of the selected subsets, \cdf{where the fill distance in \cite[Section 1.3]{filldistance}: $d(\Omega,X)=\sup\limits_{y\in\Omega}\inf\limits_{x\in X}|y-x|$ is defined for continuous regions $\Omega$ while $\delta$ focuses on finite sets of points. 
As a result, $\delta$ is always computable but fill distance is not in general.}


The error estimates derived in this section apply to \emph{any} subsets $S_1\subset X$, $S_2\subset Y$, regardless of the algorithm used to generate $S_1$, $S_2$.
Thus when $S_1, S_2$ are poorly chosen (i.e. corresponding to poor low-rank approximation), we expect the bounds to reflect the fact that the matrix approximation error is large.
This motivates the use of the estimates as indicators to distinguish ``good'' subsets and ``bad'' subsets, which will be investigated next in Section \ref{sub:indicator}.

\cdf{
\textit{Remark.}
The estimate in Theorem \ref{thm:deltaKXYerror} (as well as the one in Theorem \ref{thm:KXYdelta}) is derived to offer guidance to the fast and general algorithm based on subset selection, and it is not necessarily ``tight''.
Since the goal is to design an algorithm with linear (or nearly linear) complexity in time and space for computing accurate low-rank kernel matrix approximations by subset selection, it is desirable to obtain a straightforward characterization of ``good'' subsets via analyzing the approximation error, in order to inspire the algorithm design.
The geometric quantity $\delta$ serves the purpose.
In fact, for $O(1)$ subsets $S_1$, $S_2$, the quantities $\delta_{X,S_1}$, $\delta_{Y,S_2}$ are not only easy to compute (with linear complexity in the size of $X$, $Y$), but also consistent with the practical result when distinguishing ``good'' and ``bad'' choices of subsets for low-rank approximation as illustrated in the following section.
Hence we see that the geometric quantity $\delta$ from the theoretical result in Theorem \ref{thm:deltaKXYerror} (or Theorem \ref{thm:KXYdelta}) leads to error \emph{indicators} for subset selection.
}

\subsection{Subset quality indicators}
\label{sub:indicator}

The error bounds in Theorems \ref{thm:KXYerror} and
\ref{thm:deltaKXYerror} are fully computable and can be used to relate
the choice of subset to the low-rank approximation error.
Error bounds of this kind often arise in a posteriori error estimates for the numerical solution of partial differential equations using adaptive mesh refinement (AMR).
In AMR, an error indicator, usually a computable term in the a posteriori error estimate, is used to indicate the quality of the numerical solution and determine whether further refinement is needed without knowing the exact solution (cf. \cite{ZZ1987,verf1994,verf2005confusion,braess2008,localL2,cai2020equi,ainsworth2011confusion}).
Inspired by this philosophy, in low-rank compression methods based on geometric selection, we can use the error estimates to construct \emph{subset quality indicators} for inferring the quality of the selected subsets.
For any choice of subset $S_1\times S_2\subseteq X\times Y$, 
we consider the following five subset quality indicators:
\begin{equation}
\label{eq:indicators}
\begin{aligned}
\text{
indicator 1 = $\max_{\substack{x\in X\\ y\in Y}} \min\limits_{\substack{u\in S_1\\ v\in S_2}} |\kappa(x,y)-\kappa(u,v)|$,\quad
indicator 2 = $\max_{x\in X} \min_{u\in S_1} \norm{K_{x S_2}-K_{u S_2}}$},\\ 
\text{
indicator 3 = $\delta_{X,S_1}$,\quad 
indicator 4 = $\delta_{Y,S_2}$,\quad
indicator 5 = $\norm{K_{S_1 S_2}^+}$.
}
\end{aligned}
\end{equation}
The first two indicators are related to the upper bound derived in Theorems \ref{thm:KXYerror},
while the last three indicators are from the estimate in Theorem \ref{thm:deltaKXYerror}.
The costs for computing the indicators are \emph{not} the same.
In fact, assume $K_{XY}$ is $m$-by-$n$ and there are $O(1)$ points in $S_1$ and $S_2$.
The computational complexities for the five indicators in \eqref{eq:indicators} are:
$O(mn)$, $O(m)$, $O(m)$, $O(n)$, $O(1)$, respectively.
Hence in practice, it is more convenient to use the latter four indicators.


Given different choices of subsets, we present numerical experiments below to demonstrate 
how to use the subset quality indicators to predict which choice is more likely to yield a better approximation without computing the exact matrix approximation error.
The results also underscore the impact of the geometry of the selected subset on the low-rank approximation accuracy.
We perform two experiments, one with a rectangular kernel matrix associated with \emph{two} sets of points and  the other with a symmetric positive definite kernel matrix associated with \emph{one} set of points.

\textbf{Experiment 1.}
We consider the approximation 
$K_{XY}\approx K_{XS_2}K_{S_1 S_2}^+K_{S_1 Y}$.
The kernel function is chosen as $\kappa(x,y)=\log|x-y|$ and the rectangular kernel matrix $K_{XY}$ is associated with $X$, $Y$ (illustrated in Figure \ref{fig:ratioTestSet1}), where $X$ contains $50$ points and $Y$ contains $100$ points.
We considers two choices for $S_1\times S_2\subset X\times Y$.
Choice 1 generates points $S_1$, $S_2$ via random sampling from $X$, $Y$.
Choice 2 chooses evenly distributes points to form $S_1$, $S_2$ using FPS. These subsets are shown in Figure \ref{fig:ratioTestSet1rand} and Figure \ref{fig:ratioTestSet1FPS}. 
To determine which choice yields the better approximation, we take the ratio of the respective indicators and compare it to the ratio of the exact matrix approximation errors from the two choices.
Namely, we compute
\[
    \text{ratio-indicator } k = \frac{\text{indicator } k\; \text{of Choice 2}}{\text{indicator } k\; \text{of Choice 1}},\quad 
    \text{ratio-error}  = \frac{\text{matrix error of Choice 2}}{\text{matrix error of Choice 1}},
\]
where the matrix approximation error is measured in max norm.
If the ratio-indicator is larger than 1, then the prediction is that Choice 1 is better.
Otherwise, the prediction is that Choice 2 is better.
We then compare the indicator ratios to the ground truth: the ratio of matrix approximation errors between Choice 2 and Choice 1.
If the indicator ratio is consistent with the error ratio, i.e. both larger than 1 or both smaller than 1, then the prediction based on the indicator is correct.
The result is shown in Figure \ref{fig:ratioSet1x}.
It is easily seen that, for different approximation ranks, the indicator ratios and the error ratio always stay below the horizontal line $y=1$.
Hence the indicators correctly predict the fact that Choice 2 of subsets yields a better low-rank approximation than Choice 1.
Furthermore, note that unlike Choice 1, the points in Choice 2 are evenly distributed over the dataset and thus are expected to yield a better approximation according to the theoretical results in Section \ref{sub:errorfac2}.

\begin{figure}[htbp]
    \centering
    \subfigure[Datasets $X$, $Y$]{\label{fig:ratioTestSet1}\includegraphics[width=60mm]{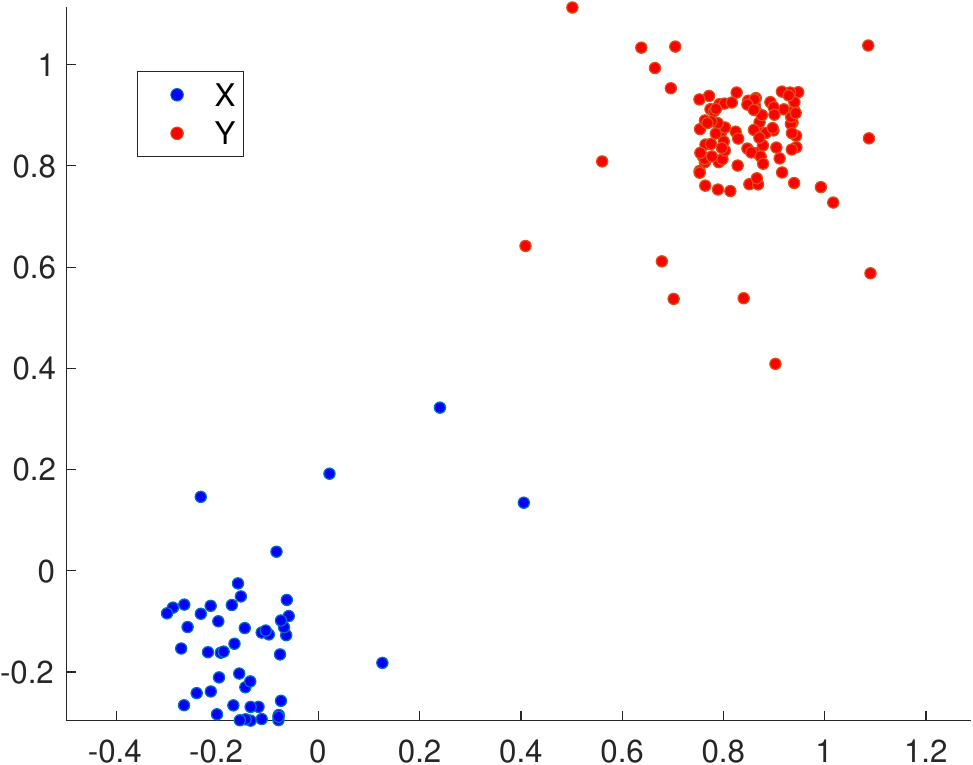}}
     \hspace{.1mm}
    \subfigure[Comparison of ratios]{\label{fig:ratioSet1x}\includegraphics[width=60mm]{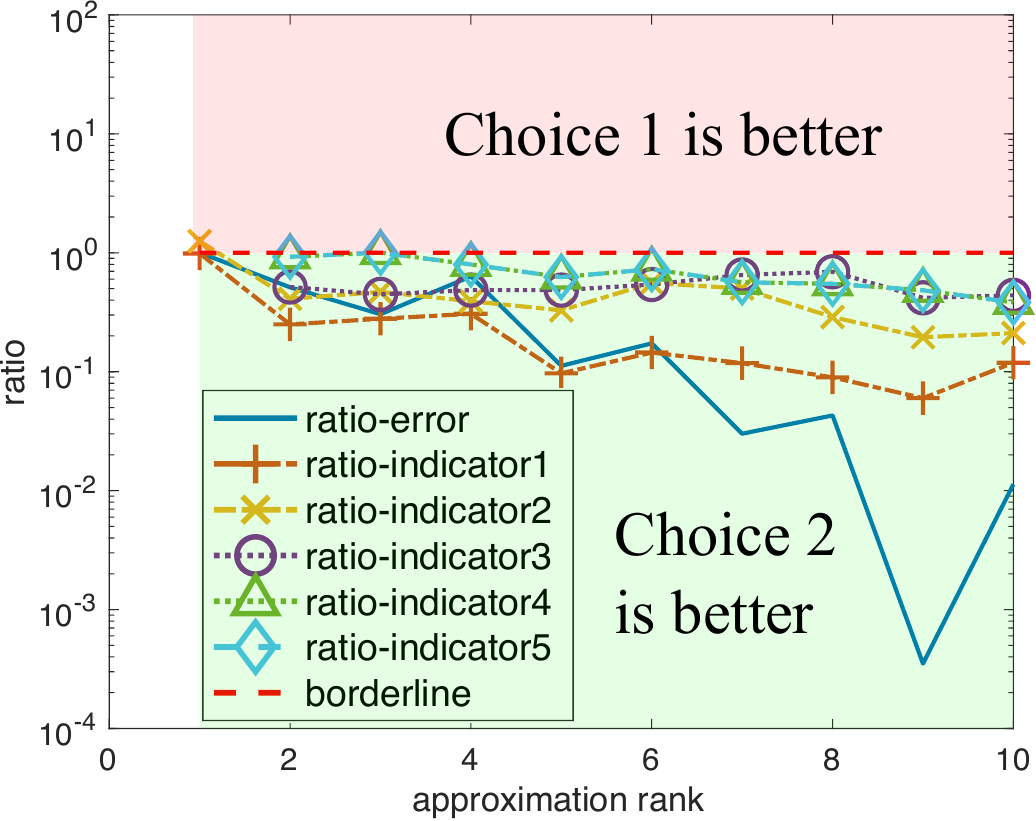}}
     \hspace{.1mm}
    \subfigure[Choice 1]{\label{fig:ratioTestSet1rand}\includegraphics[width=60mm]{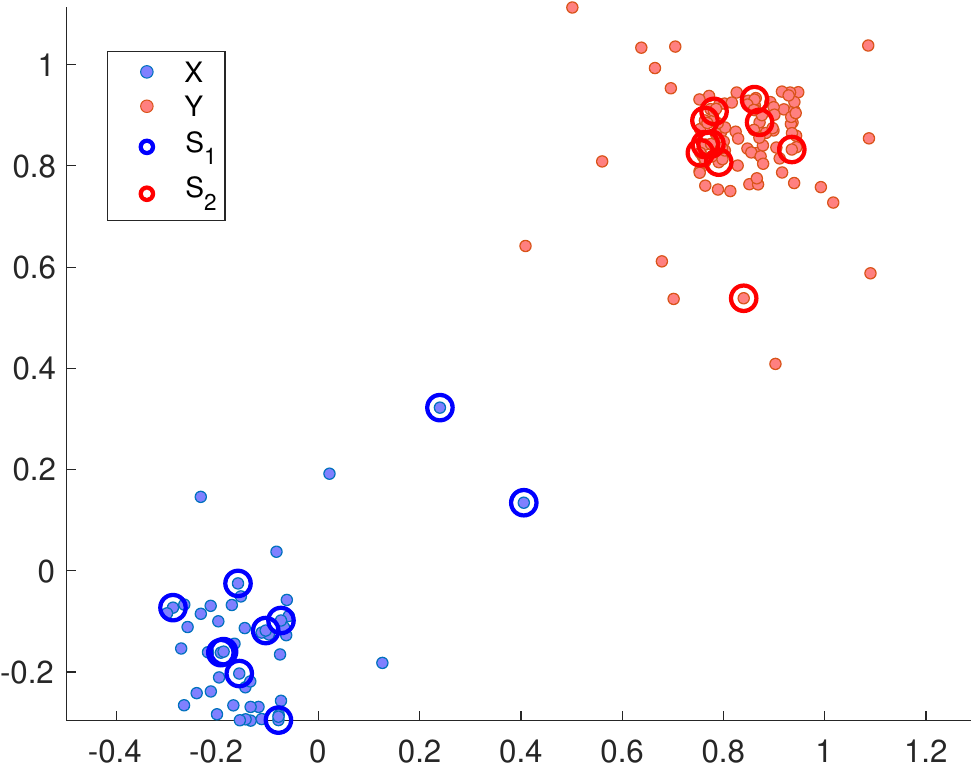}}
     \hspace{.1mm}
    \subfigure[Choice 2]{\label{fig:ratioTestSet1FPS}\includegraphics[width=60mm]{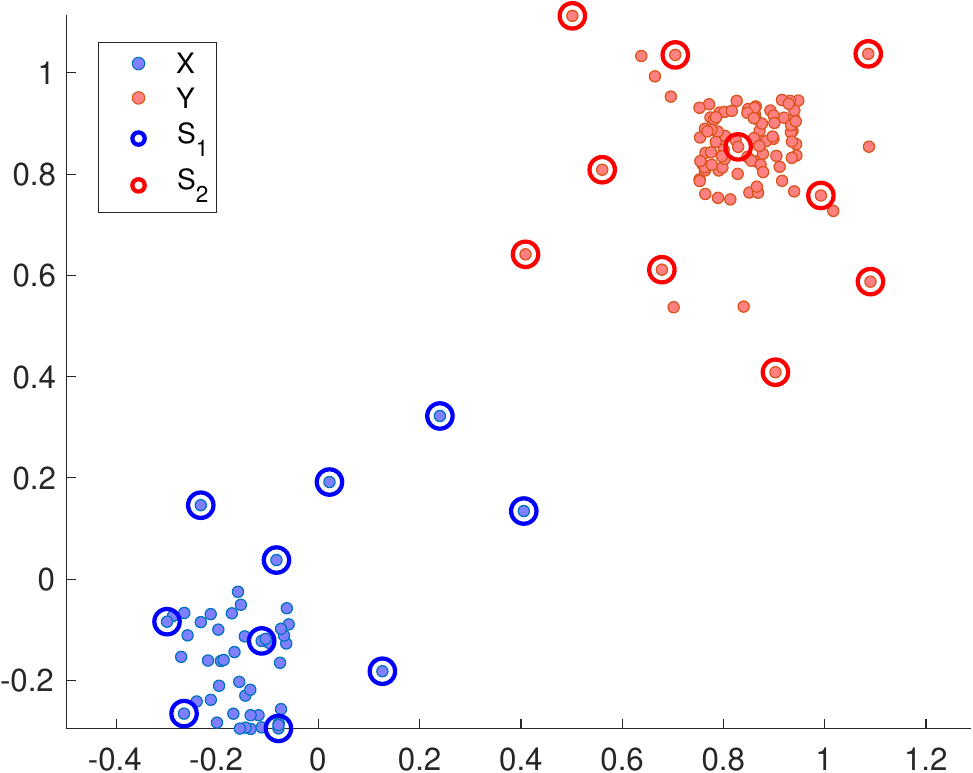}}
    \caption{Experiment 1: Predicting the better choice of subsets $S_1,S_2$ using subset indicators in \eqref{eq:indicators}.}
    \label{fig:ratioSet1}
\end{figure}

\textbf{Experiment 2.}
We consider the Gaussian kernel
$\kappa(x,y)=\exp(-|x-y|^2/0.09)$
and the symmetric approximation 
$K_{XX}\approx K_{XS}K_{SS}^+K_{SX}$,
where the dataset $X$ \cdf{contains 100 points as shown} in Figure \ref{fig:ratioTestSet2}.
We follow the same choices of subset as in Experiment 1, i.e., Choice 1 selects random samples while Choice 2 selects evenly distributed points. These two choices of subset $S$ are shown in Figure \ref{fig:ratioTestSet2rand} and Figure \ref{fig:ratioTestSet2FPS}.
We compute the same indicators as in \eqref{eq:indicators}, where in this case $Y=X$ and $S_2=S_1=S$.
The result is shown in Figure \ref{fig:ratioSet1x}.
We see that when the approximation rank is larger than 5, all indicator ratios and the error ratio stay below the horizontal line $y=1$ simultaneously.
This implies that Choice 2 yields a better approximation and the indicators give the correct prediction.
Again, we see that evenly distributed points  yield a better approximation, as discussed in Section \ref{sub:errorfac2}.

\begin{figure}[htbp]
    \centering
    \subfigure[Dataset $X$]{\label{fig:ratioTestSet2}\includegraphics[width=60mm]{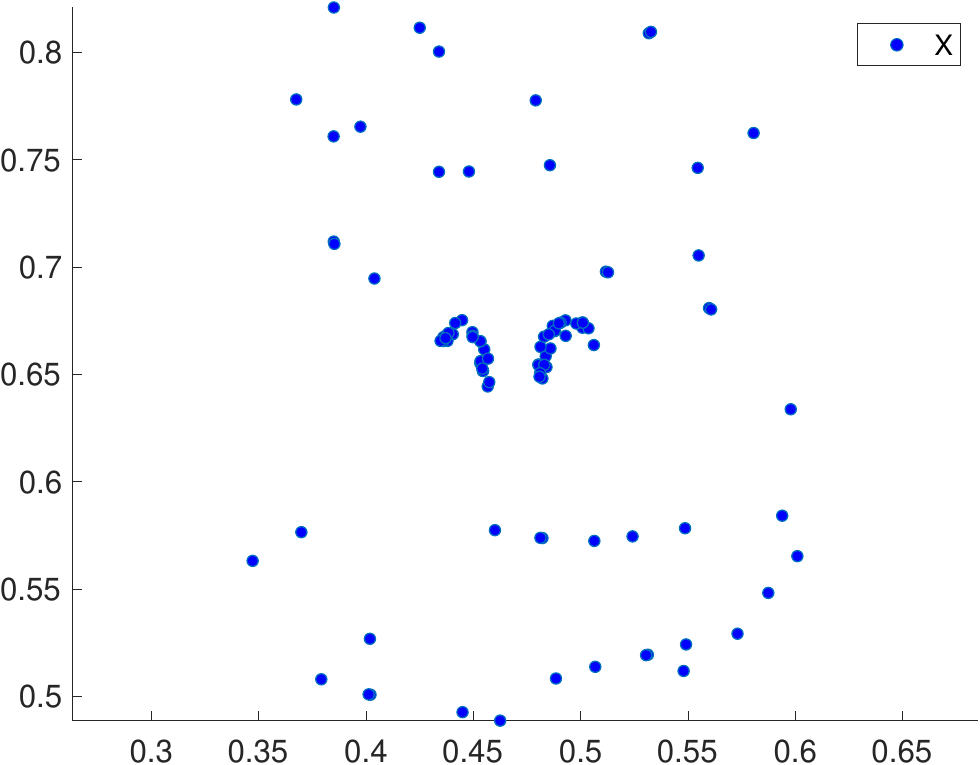}}
     \hspace{.1mm}
    \subfigure[Comparison of ratios]{\label{fig:ratioSet2x}\includegraphics[width=60mm]{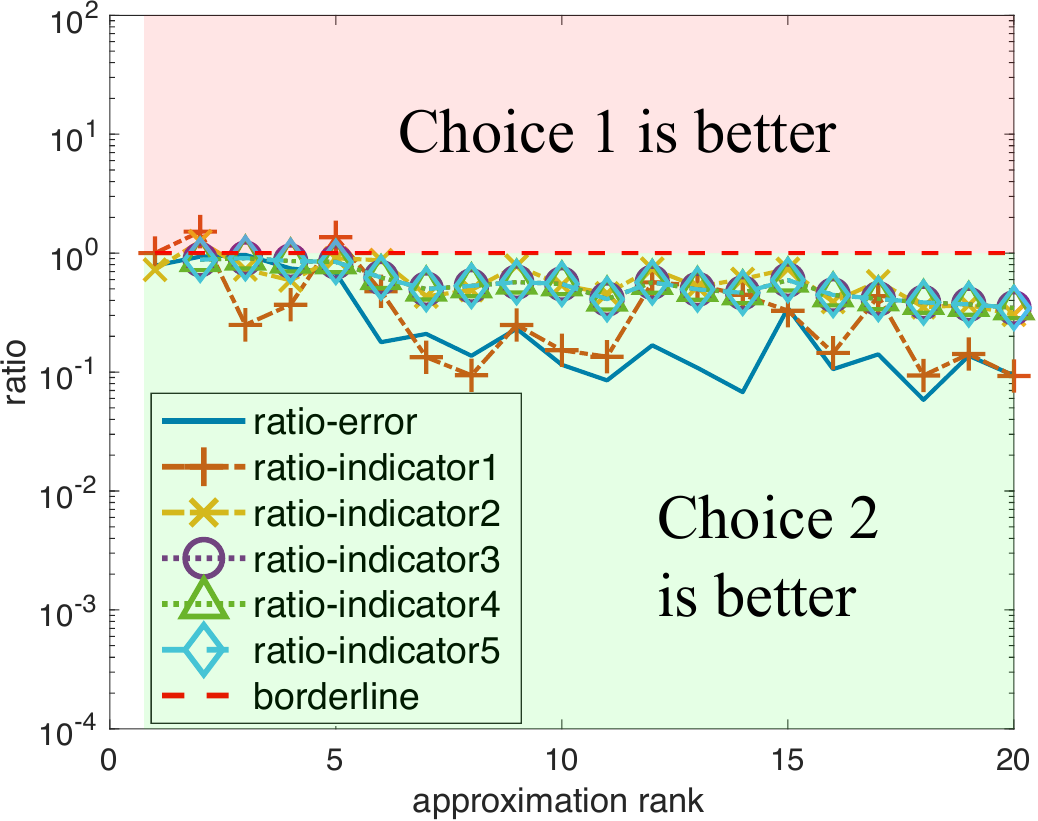}}
     \hspace{.1mm}
    \subfigure[Choice 1]{\label{fig:ratioTestSet2rand}\includegraphics[width=60mm]{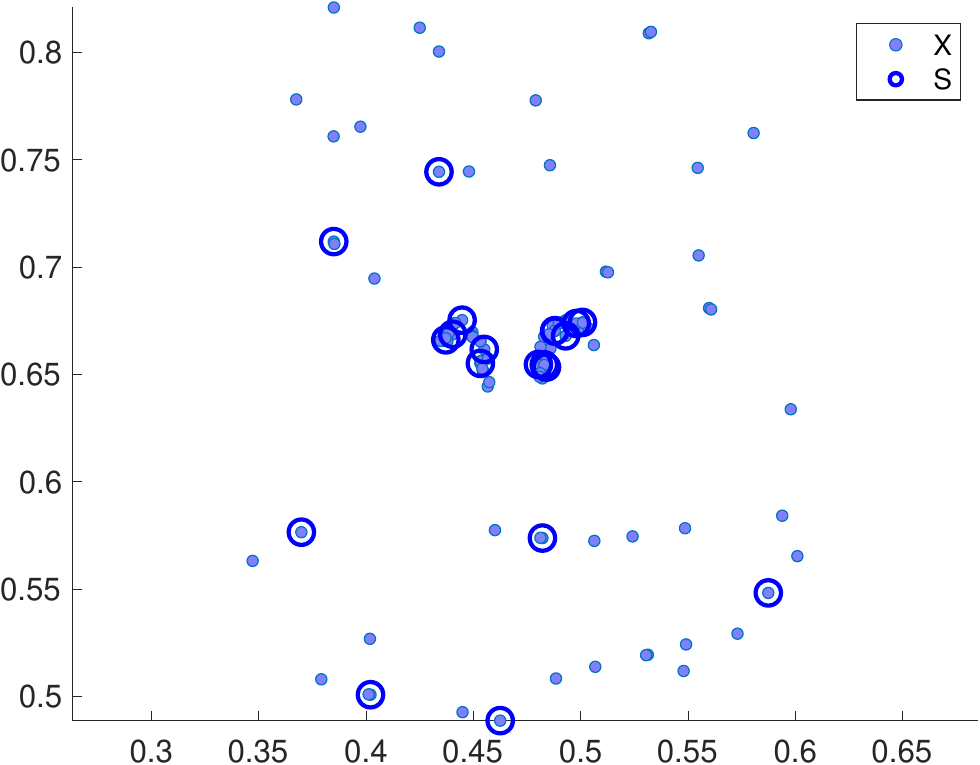}}
     \hspace{.1mm}
    \subfigure[Choice 2]{\label{fig:ratioTestSet2FPS}\includegraphics[width=60mm]{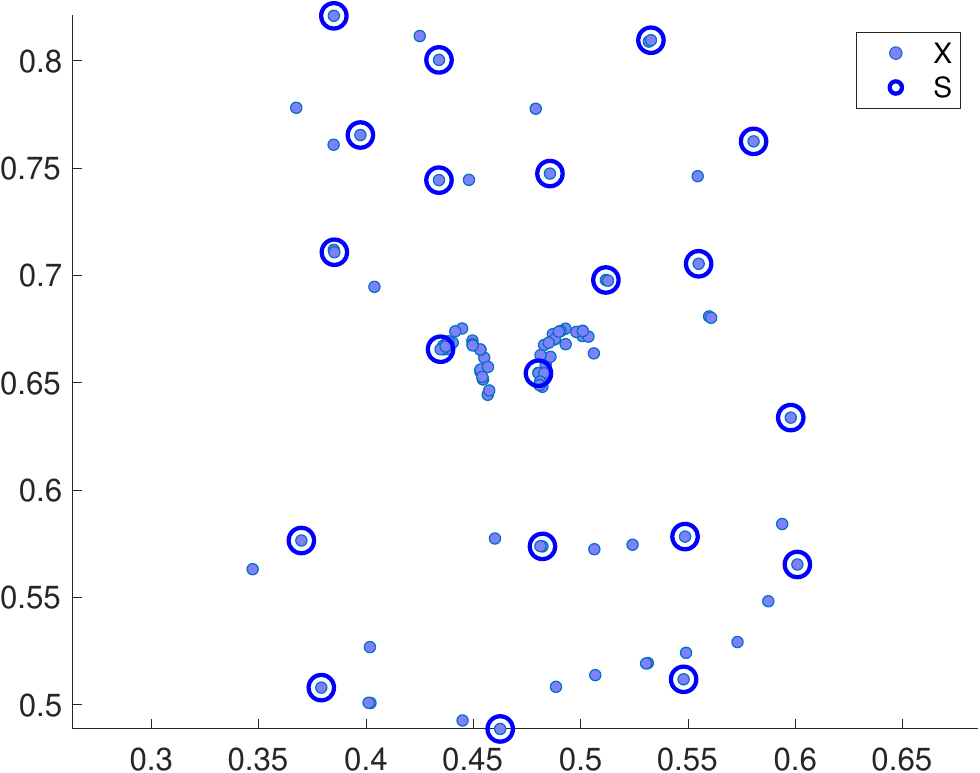}}
    \caption{Experiment 2: Predicting the better choice of subset $S$ using subset indicators in \eqref{eq:indicators}.}
    \label{fig:ratioSet2}
\end{figure}

\section{One-sided low-rank kernel matrix approximation}
\label{sec:one-sided}
This section analyzes the data-driven geometric approach for the one-sided low-rank approximation \eqref{eq:1sided}.  
Compared to the two-sided case, the algorithm in the one-sided case only samples points from one set of points and applies an algebraic factorization to postprocess the selected submatrix. 
Numerical experiments in Section \ref{sec:test} show that the one-sided algorithm is slightly more accurate.

\cdf{
The key algebraic technique we use is the strong rank-reveal QR factorization (SRRQR\cite{rrqr96}).
The original setting in \cite{rrqr96} considers ``tall'' matrices and direct application of the result yields pessimistic computational complexity in the current setting where ``short'' matrices are of interest.
To resolve this issue, we discuss the new setting and derive a nearly optimal complexity estimate in Section \ref{sub:mySRRQR}.
}
We present the algorithm for the one-sided low-rank approximation in Section \ref{sub:one-sided} and provide the error analysis in Section
\ref{sub:errorfac1}.  The special case of a symmetric kernel matrix
$K_{XX}$ with a symmetric kernel $\kappa(x,y)=\kappa(y,x)$ is discussed
in Section \ref{sub:Special case}.

\subsection{Strong rank-revealing QR for ``short'' matrices} 
\label{sub:mySRRQR}
\cdf{
The classical result on SRRQR is cited in Proposition \ref{prop:SRRQR}. 
Algorithms for computing SRRQR are proposed in \cite{rrqr96} and we use \cite[Algorithm 4]{rrqr96} in our approach to postprocess the $r\times n$ matrix $K_{X S_2}^T$, where $S_2\subseteq Y$ is the selected subset with $r$ points.
We point out that the original result on SRRQR (cited in Proposition \ref{prop:SRRQR}) only considers ``tall'' matrices of size $m\times n$ with $m\geq n$ and the complexity contains a term of $O(n^3)$. 
Such a complexity will be too pessimistic for the ``short'' matrix $K_{X S_2}^T$ of size $r\times n$ with $r\ll n$ in general.
To obtain a complexity estimate that is nearly optimal in $n$, we present in this section a rigorous analysis of SRRQR for ``short'' matrices.
The result in Proposition \ref{prop:mySRRQR} shows that the complexity of SRRQR is in between $O(r^2 n)$ and $O(r^2 n \log_s n)$ for $r\times n$ matrices with rank $r$.
That is, the complexity is linear or nearly linear in $n$. 
}

\cdf{
\begin{proposition}[Strong Rank-revealing QR Factorization\cite{rrqr96}]
    \label{prop:SRRQR}
    Let $M$ be an $m\times n$ matrix with $m\geq n$.
    The SRRQR of $M$ yields
   $M = Q \begin{bmatrix}
      A_k & B_k \\ 
        & C_k
   \end{bmatrix} \Pi$,
where $Q$ is $m\times m$ orthogonal, $\Pi$ is a permutation matrix, $A_k$ is a well-conditioned $k\times k$ upper triangular matrix with the $i$th ($1\leq i\leq k$) singular value $\sigma_i(A_k)\geq \sigma_i(M)/\sqrt{1+s^2k(n-k)}$, $C_k$ satisfies $\sigma_i(C_k)\leq \sigma_{k+j}(M)\sqrt{1+s^2k(n-k)}$ with $1\leq j\leq n-k$,
$\norm{A_k^{-1}B_k}_{\max}\leq s$.
Here $s>1$ is a user-specified constant.
The complexity for SRRQR is $O(mn^2+n^3\log_s n)$.
\end{proposition}
}

\cdf{
We are interested in applying SRRQR to ``short'' $r\times n$ matrices as described below and the complexity in Proposition \ref{prop:SRRQR} does not reflect the efficiency in the new setting, where $M$ is $r\times n$ with rank $r$. \\
\textbf{Algorithm\cite[Algorithm 4]{rrqr96} for computing the SRRQR in Proposition \ref{prop:mySRRQR}:}
\begin{enumerate}
    \item Compute $R=[A_r, B_r]:=\mathcal{QR}(M)$ and define $\Pi=I$, where $\mathcal{QR}$ denotes the QR factorization;
    \item \texttt{while} $||A_r^{-1} B_r||_{\max} > s$ \texttt{do}
    \item Find $i$, $j$ such that $|(A_r^{-1} B_r)_{i,j}| > s$;
    \item Compute $R=[A_r, B_r]:=\mathcal{QR}(R\Pi_{i,j+r})$ and $\Pi := \Pi\; \Pi_{i,j+k}$, where $\Pi_{i,j+k}$ denotes the permutation that interchanges the $i$th and $j+k$th columns;
    \item \texttt{endwhile}
\end{enumerate}
We analyze SRRQR for ``short'' matrices and prove the nearly optimal complexity of the algorithm above.
The result is summarized in Proposition \ref{prop:mySRRQR}.
A straightforward corollary of Proposition \ref{prop:mySRRQR} gives a stable interpolative decomposition for ``tall'' matrices (Corollary \ref{cor:ID}) that will be used in the one-sided low-rank approximation in Algorithm \ref{alg:fac1}.
}

\cdf{
\begin{proposition}[SRRQR for ``short'' matrices]
    \label{prop:mySRRQR}
    Let $M$ be an $r\times n$ matrix with rank $r$ (thus $r\leq n$).
    The SRRQR of $M$ yields
   $M = Q \begin{bmatrix}
      A_r & B_r
   \end{bmatrix} \Pi$,
where $Q, P, \Pi, A_r, B_r$ are as in Proposition \ref{prop:SRRQR}, with $||A_r^{-1}B_r||_{\max}\leq s$.
Here $s>1$ is a user-specified constant.
The complexity for computing such a factorization is 
$O(n_{\text{iter}} r^2 n)$, where $n_{\text{iter}}$ denotes the number of \texttt{while} loops in Line 2 of the SRRQR algorithm above, and
$n_{\text{iter}}$ is between $O(1)$ and $O(\log_s n)$.
That is, the complexity of SRRQR is between $O(r^2 n)$ and $O(r^2 n \log_s n)$.
\end{proposition}

\begin{proof}
The QR factorization in Line 1: $[A_r, B_r]:=\mathcal{QR}(M)$ for has complexity $O(r^2 n)$.
The number of \texttt{while} loops $n_{\text{iter}}$ is at most $O(\log_s n)$ according to the estimate in [34, Section 4.4].
In fact, the $O(\log_s n)$ estimate is calculated for the \emph{two} nested while loops in \cite[Algorithm 5]{rrqr96} in which $n_{\text{iter}}$ corresponds to the inner while loop.
As a result, the complexity of $n_{\text{iter}}$ must not exceed $O(\log_s n)$.
In the \texttt{while} loop, computing each $||A_r^{-1} B_r||_{\max}$ requires $O(r^2 n)$ complexity, since $A_r$ is triangular and $B_r$ is $r\times (n-r)$.

Next we analyze the complexity for each QR factorization to $\tilde{R}:=R\Pi_{i,j+r}$ in the \texttt{while} loop.
Without loss of generality, we assume $i=1$.
This is because in this case, $\tilde{R}=R\Pi_{1,j+r}$ has the following sparsity pattern (blank entries denote zeros) and QR will be applied to the whole matrix, which results in the highest complexity.
If $i>1$, then the first $i-1$ columns already form an upper triangular matrix and QR is applied to the non-triangular submatrix in the lower right part of $\tilde{R}$ whose size is strictly smaller than $r$-by-$n$.
\begin{equation*}
\tilde{R} :=
\begin{bmatrix}
   * & * & \dots & * & * & \dots & \dots & \dots  \\
   * & * & \dots & \vdots & * & \dots & \dots & \dots  \\
   * &  & \ddots& \vdots & * & \dots & \dots & \dots  \\
   * &  &  & * & * & \dots & \dots & \dots  \\
   * &  &  &    & * & \dots & \dots & \dots 
\end{bmatrix}
\end{equation*}
To compute the QR factorization of $\tilde{R}$ efficiently, 
we apply Householder reflection or Givens rotation to submatrices of row size two in a \emph{bottom-up} fashion, which will reduce the matrix into an upper Hessenberg form.
Then we apply Householder reflection or Givens rotation to the upper Hessenberg form in a \emph{top-down} fashion to obtain an upper triangular matrix, which completes the QR factorization.

In the bottom-up reduction, we first apply Householder reflection or Givens rotation to the last two rows of $\tilde{R}$ to zero out the entry in the bottom left corner (see \eqref{eq:do1}), i.e. entry  $(r,1)$ in $\tilde{R}$.
Note that this will introduce a nonzero entry at $(r,r-1)$, denoted by `$\bullet$' in \eqref{eq:do1}.

\begin{equation}
\label{eq:do1}
\begin{bmatrix}
   * & & & &  & * & * & \dots & \dots & \dots  \\
   * & & & &  &    & * & \dots & \dots & \dots 
\end{bmatrix}
\longrightarrow
\begin{bmatrix}
   * &  &  & * & * & \dots & \dots & \dots  \\
     &  &  & \bullet & * & \dots & \dots & \dots 
\end{bmatrix}
\end{equation}

Applying the same process recursively to the two-row submatrices (rows $k-1, k$) with $k=r-1, r-2, \dots, 3$, 
we obtain an upper Hessenberg form:
\begin{equation}
\label{eq:Hess}
\begin{bmatrix}
   * & * & \dots & * & * & \dots & \dots & \dots  \\
   * & * & \dots & \vdots & * & \dots & \dots & \dots  \\
    & \bullet & \ddots& \vdots & * & \dots & \dots & \dots  \\
    &  & \ddots & * & * & \dots & \dots & \dots  \\
    &  &  &  \bullet  & * & \dots & \dots & \dots 
\end{bmatrix}
\end{equation}
The total complexity of this bottom-up procedure is 
$$O\left((n-r+3)+(n-r+4)+\dots+(n-r+r)\right)=O(rn),$$ 
where the number in each inner parenthesis denotes the number of nonzero columns in each two-row matrix.

Then we reduce the upper Hessenberg form in \eqref{eq:Hess} into an upper triangular matrix by applying Householder reflection or Givens rotation sequentially (in a top-down fashion) to the two-row submatrix (rows $k, k+1$) with $k=1,2,\dots,r-1$, in order to zero out the subdiagonal entries.
Similar to the bottom-up procedure, it is easy to see that the total complexity of the top-down procedure is also $O(rn)$.

Therefore, we see that the each execution in the \texttt{while} loop is donimated by the cost of computing $||A_r^{-1} B_r||_{\max}$, with $O(r^2 n)$ complexity.
The total complexity of the entire algorithm is then 
$$O(r^2 n + n_{\text{iter}} r^2 n ) = O(n_{\text{iter}} r^2 n).$$
Given the fact that $n_{\text{iter}}$ is between $O(1)$ and $O(\log_s n)$,
the complexity of SRRQR is between $O(r^2 n)$ and $O(r^2 n \log_s n).$
\end{proof}

\begin{corollary}
    \label{cor:ID}
    Let $M$ be an $n\times r$ matrix with rank $r$.
    Then $M$ can be factorized via SRRQR as
   $M = P  \begin{bmatrix}
      I \\ 
     G  
   \end{bmatrix} M_1$,
where $P$ is a permutation matrix, $I$ is an identity matrix, $M_1$ is the matrix that consists of the first $r$ rows of $P^T M$, and 
$\norm{G}_{\max}\leq s$.
Here $s>1$ is a user-specified constant.
The computational complexity is at most $O(r^2 n \log_s n)$.
\end{corollary}
\begin{proof}
    Applying SRRQR to the $r\times n$ matrix $M^T$ yields 
\begin{equation}
\label{eq:MT}
M^T = Q
\begin{bmatrix}
    A_r & B_r
\end{bmatrix} \Pi,
\end{equation}
where $Q, \Pi, A_r, B_r$ are matrices in Proposition \ref{prop:mySRRQR}. 
In particular, 
$||A_r^{-1}B_r||_{\max}\leq s$.
Meanwhile, the complexity is at most $O(r^2 n \log_s n)$ according to Proposition \ref{prop:mySRRQR}.

Since $\Pi$ is a permutation matrix, $QA_r$ is a submatrix of $M^T$ containing the first $r$ columns of $M^T \Pi^T$. 
Define $M_1^T = QA_r$.
We see that $M_1$ contains the first $r$ rows of $\Pi M$.
We can then rewrite \eqref{eq:MT} as 
\begin{equation*}
    M^T = QA_r  \begin{bmatrix}
    I & A_r^{-1} B_r
\end{bmatrix} \Pi 
= M_1^T \begin{bmatrix}
    I & A_r^{-1} B_r
\end{bmatrix} \Pi.
\end{equation*}
Transposing both sides yields the desired factorization 
with $P:=\Pi^T$, $G:=A_r^{-1}B_r$, $||G||_{\max}\leq s$.
The complexity is at most $O(r^2 n \log_s n)$ thanks to Proposition \ref{prop:mySRRQR}.
\end{proof}
}

\cdf{
\textit{Remark.}
Note that in the original SRRQR\cite{rrqr96}, the permutation is performed for the smaller dimension, i.e. $n$ columns for a ``tall'' $m\times n$ matrix with $m\geq n$.
In Proposition \ref{prop:mySRRQR} and Corollary \ref{cor:ID}, the permutation is performed over the larger dimension.
This calls for the new complexity analysis in the proof of Proposition \ref{prop:mySRRQR} different from the original estimate in \cite{rrqr96}.
}


\subsection{Algorithm}
\label{sub:one-sided}

The one-sided approximation method consists of two stages.  In the first stage, a subset $S_2\subseteq Y$ is selected using a linear complexity geometric selection scheme (Section \ref{sec:intro}).
In the second stage,
\cdf{
we compute the interpolative decomposition in \eqref{eq:1sided} via applying SRRQR\cite{rrqr96} to guarantee the maximum norm of the column basis matrix is bounded by a prescribed number $s>1$.
More precisely, we apply SRRQR to the ``short'' matrix $K_{X S_2}^T$ and then transpose the output to obtain $K_{X S_2} = P \begin{bmatrix}
    I\\
    G
\end{bmatrix} K_{\mathcal{I}_1 S_2}$, with $\mathcal{I}_1\subseteq X$, $P$ a permutation matrix and $||G||_{\max}\leq s$.
See Corollary \ref{cor:ID} for a more detailed discussion.
}

The full one-sided compression algorithm is summarized in Algorithm \ref{alg:fac1}.
Notice that in Step 2 of Algorithm \ref{alg:fac1}, SRRQR is only used to obtain a stable factorization of $K_{XS_2}$ and thus \emph{no} approximation error is introduced.

\begin{algorithm}
    \caption{\it Data-driven one-sided compression of $K_{XY}$ with two sets of points $X, Y$}
    \label{alg:fac1}
    \emph{Input:} Datasets $X=\{x_{1},\dots,x_{m}\}$, $Y=\{y_1,\dots,y_n\}\subset \mathbb{R}^d$, kernel function $\kappa$, number of sample points $r$ for $Y$\\
    \emph{Output:} Low-rank approximation $K_{XY}\approx UK_{\iii_1 Y}$ in \eqref{eq:1sided}
        \begin{algorithmic}
        \STATE Apply a linear complexity geometric selection algorithm to $Y$ to generate $r$ sample points $S_2\subseteq Y$
        \STATE Apply \cdf{SRRQR-based ID} to the $m$-by-$r$ kernel matrix $K_{XS_2}$: 
    $K_{XS_2}
    = P \begin{bmatrix}
        I\\ 
        G
    \end{bmatrix}
    K_{\iii_1 S_2}$,
where $I$ is an identity matrix, $\iii_1\subseteq X$, $P$ is a permutation matrix \cdf{that maps the row indices of $I$ to the indices for $\iii_1$ in $X$,} and $\norm{G}_{\max}\leq 2$.
        \STATE Define
    $U = P \begin{bmatrix}
        I\\ 
        G
    \end{bmatrix}$.
        \STATE Return $U, K_{\iii_1 Y}$
    \end{algorithmic}
\end{algorithm}

Compared to purely algebraic methods such as LU, QR, rank-revealing QR, and SVD decompositions, Algorithm \ref{alg:fac1} does not access the full kernel matrix and scales linearly or nearly linearly with respect to the data size.
Compared to the proxy point methods \cite{Gillman2012,xing2020interpolative}, hybrid cross approximation \cite{HCA2005}, Algorithm \ref{alg:fac1} does not require the evaluation of the kernel function outside the given dataset (where the function may not necessarily be defined) and is able to scale to high dimensions.
In terms of numerical stability, Algorithm \ref{alg:fac1} leverages the robustness of algebraic methods to obtain a stable factorization compared to the two-sided approximation.
As we shall see in Section \ref{sec:test}, despite being more stable and more general, the error-time trade-off of the proposed method can be noticeably better than that of existing methods.

In addition to the factorization $K\approx UK_{\iii_1 Y}$, a similar one-sided factorization $K\approx K_{X\iii_2}V^{\ast}$ can be computed by applying Algorithm \ref{alg:fac1} to $K^{\ast}_{XY}$. 
That is, we first select a subset from $X$ and then apply ID to obtain a subset $\iii_2\subseteq Y$. 
\cdf{Both options first apply geometric selection to $X$ or $Y$ to obtain a small submatrix and then apply algebraic factorization to it.
They differ in which data set the geometric selection is applied to, i.e. $X$ or $Y$.
If one set contains significantly more points than the other one (for example, $m\gg n$), for efficiency, it is better to perform geometric selection on the larger set to reduce its size to $O(1)$, so that the following algebraic factorization, which is more expensive than geometric selection, is applied to a submatrix with a smaller dimension $n$-by-$O(1)$ instead of $m$-by-$O(1)$.}

\subsection{Complexity and error analysis}
\label{sub:errorfac1}
In this section, we analyze the complexity and the approximation error of Algorithm \ref{alg:fac1}. 
First, we show that Algorithm \ref{alg:fac1} scales as $O(r^2(m+n))$ for obtaining a rank-$r$ approximation to an $m\times n$ kernel matrix.
\begin{theorem}
Given $X=\{x_i\}_{i=1}^m$, $Y=\{y_i\}_{i=1}^n$ in $\mathbb{R}^d$ and kernel function $\kappa$,
the complexity of Algorithm \ref{alg:fac1} to compute a rank-$r$ approximation $K_{XY}\approx UK_{\iii_1 Y}$ is $O(dr^2(m+n))$.
\end{theorem}
\begin{proof}
    Compressing a set of $n$ points into $r$ points with any scheme in Section \ref{sec:intro} has a complexity at most $O(dr^2 n)$.
    The cost of applying ID on a $m$-by-$r$ matrix $K_{XS_2}$ is $O(r^2 m)$.
    Therefore, the overall complexity of Algorithm \ref{alg:fac1} is $O(dr^2(m+n))$.
\end{proof}

Next we analyze the approximation error for $K_{XY}\approx UK_{\iii_1 Y}$ computed by Algorithm \ref{alg:fac1}. 
We will see that, different from the two-sided factorization, the error bound for $K_{XY}\approx UK_{\iii_1 Y}$ does not involve the norm of the pseudoinverse of the matrix.

\begin{theorem}
\label{thm:KXYUK1}
    Let $X$ and $Y$ be finite sets in $\mathbb{R}^d$
    and $\kappa(x,y)$ be defined on $X\times Y$. For any non-empty subset $S\subseteq Y$,
    let $K_{XS}$ be decomposed by \cdf{SRRQR-based ID} as
    $K_{XS} = U K_{\iii S} = P \begin{bmatrix} I\\G \end{bmatrix} K_{\iii S}$ with $\Vert G\Vert_{\max}\leq 2$.
    Then 
    \begin{equation}
    \label{eq:KXYUK1}
    \begin{aligned}
        \norm{ K_{XY} - U K_{\iii Y} }_{\max}
        \leq &  
\max_{\substack{x\in X\\ y\in Y}} \min_{v\in S} \left( \abs{\kappa(x,y)-\kappa(x,v)}+\norm{K_{Xy}-K_{Xv}} \right)\\
&+
2r\max_{\substack{x\in \mathcal{I}\\ y\in Y}} \min_{v\in S} \left( \abs{\kappa(x,y)-\kappa(x,v)}+\norm{K_{Xy}-K_{Xv}} \right),
    \end{aligned}
    \end{equation}
where $r=\text{card}(\iii)$.
\end{theorem}
\begin{proof}
    We decompose $K_{XY}$ as  
\begin{equation}
\label{eq:KXYdecom}
\begin{aligned}
    K_{XY} &= K_{XS}K_{XS}^+K_{XY} + E_1\quad\;\; \text{with}\quad E_1= K_{XY} - K_{XS}K_{XS}^+K_{XY}\\
    &= P \begin{bmatrix} I\\G \end{bmatrix} K_{\iii S}K_{XS}^+K_{XY} + E_1\\
    &= P \begin{bmatrix} I\\G \end{bmatrix} (K_{\iii Y}+E_2) +  E_1\quad\text{with}\quad E_2=K_{\iii S}K_{XS}^+K_{XY}-K_{\iii Y} \\
    &= P \begin{bmatrix} I\\G \end{bmatrix} K_{\iii Y}+P \begin{bmatrix} I\\G \end{bmatrix}E_2 +  E_1.
\end{aligned}
\end{equation}
According to Corollary \ref{cor:KXYerror},
\begin{equation}
\label{eq:E1UK1}
   \norm{E_1}_{\max} \leq 
\max_{\substack{x\in X\\ y\in Y}} \min_{v\in S} \left( \abs{\kappa(x,y)-\kappa(x,v)}+\norm{K_{Xy}-K_{Xv}} \right).
\end{equation}
Similarly, for $E_2$, we have
\begin{equation}
\label{eq:E2UK1}
\norm{E_2}_{\max}\leq 
\max_{\substack{x\in \mathcal{I}\\ y\in Y}} \min_{v\in S} \left( \abs{\kappa(x,y)-\kappa(x,v)}+\norm{K_{Xy}-K_{Xv}} \right).
\end{equation}
Since $\norm{G}_{\max}\leq 2$  and the row size of $E_2$ is equal to $r=\text{card}(\iii)$, it follows that
\[
    \norm{GE_2}_{\max}\leq 2r\norm{E_2}_{\max}.
\]
Together with \eqref{eq:E2UK1}, \eqref{eq:E1UK1}, and \eqref{eq:KXYdecom}, we deduce the inequality in \eqref{eq:KXYUK1}.
\end{proof}

The estimate in Theorem \ref{thm:KXYUK1} relates the approximation error to the subset $S$. 
\cdf{We remark that in the estimate, $r$ is fixed and $S$ is viewed as a variable since we aim to study how the choice of $S$ affects the low-rank approximation accuracy. This is different from estimates that study how error decays with $r$.}
We show a more geometric characterization of the error bound of Theorem \ref{thm:KXYUK1} in the following theorem. 
This theorem implies that the approximation error depends on the ability of $S$ to capture $Y$, which is similar to the two-sided approximation case described in Theorem \ref{thm:deltaKXYerror}.

\begin{theorem}
\label{thm:KXYdelta}
    Let $X,Y,\kappa,S,\iii$ be given in Theorem \ref{thm:KXYUK1} and let $K_{XY}\approx UK_{\iii Y}$ be the approximation in Theorem \ref{thm:KXYUK1}. Then
\begin{equation}
\label{eq:KXYdelta}
\begin{split}
    \norm{K_{XY}-UK_{\iii Y}}_{\max} 
\leq &
L(X\times Y, X\times S) \delta_{Y,S}+(1+2r)\sqrt{m}L(Y,S)_X \delta_{Y,S}\\ 
&+2r L(\mathcal{I}\times Y, \mathcal{I}\times S) \delta_{Y,S},
\end{split}
\end{equation}
where $m=\text{card}(X), r=\text{card}(\iii)$.
\cdf{
Furthermore, if $\kappa(x,y)$ is Lipschitz continuous over $D_1\times D_2$ containing $X\times Y$ with Lipschitz constant $L$, then 
\begin{equation*}
    \norm{K_{XY}-UK_{\iii Y}}_{\max} 
\leq 
L \delta_{Y,S} + (1+2r)\sqrt{m}L \delta_{Y,S} + 2r L \delta_{Y,S}.
\end{equation*}
}
\end{theorem}

\begin{proof}
The proof is analogous to that of Theorem \ref{thm:deltaKXYerror}.
According to \eqref{eq:kLip}, we deduce that 
\begin{equation*}
\begin{aligned}
\max_{\substack{x\in X\\ y\in Y}} \min_{v\in S} \left( \abs{\kappa(x,y)-\kappa(x,v)}+\norm{K_{Xy}-K_{Xv}} \right)   
&\leq  \max_{\substack{x\in X\\ y\in Y}} \min_{v\in S}
\left( L(X\times Y, X\times S) |y-v|+\sqrt{m}L(Y,S)_X |y-v|  \right) \\
&=\max_{\substack{x\in X\\ y\in Y}} \left( L(X\times Y, X\times S) \dist{y,S}+\sqrt{m}L(Y,S)_X \dist{y,S}  \right)\\
&=L(X\times Y, X\times S) \delta_{Y,S}+\sqrt{m}L(Y,S)_X \delta_{Y,S}.
\end{aligned}
\end{equation*}
Similarly, it can be deduced that 
\begin{equation*}
\max_{\substack{x\in \mathcal{I}\\ y\in Y}} \min_{v\in S} \left( \abs{\kappa(x,y)-\kappa(x,v)}+\norm{K_{Xy}-K_{Xv}} \right)   
\leq L(\mathcal{I}\times Y, \mathcal{I}\times S) \delta_{Y,S}+\sqrt{m}L(Y,S)_X \delta_{Y,S}.
\end{equation*}
Inserting the above two inequalities into \eqref{eq:KXYUK1} completes the proof of \eqref{eq:KXYdelta}.
\cdf{The special case of $\kappa(x,y)$ being Lipschitz follows immediately from \eqref{eq:KXYdelta} and Proposition \ref{prop:L}.}
\end{proof}

From Theorem \ref{thm:KXYdelta}, it is easy to see that smaller $\delta_{Y,S}$ contributes to better approximation and the approximation error is zero if $S=Y$.
Also, we see that the smoother the kernel function is (small Lipschitz constant), the more accurate the low-rank approximation will be. 
This is consistent with the fact that smooth kernel functions yield kernel matrices with rapidly decaying singular values.

Compared to the error estimates in Theorem \ref{thm:KXYerror} and Theorem \ref{thm:deltaKXYerror} for the \emph{two-sided} factorization, the estimates for the \emph{one-sided} factorization in Theorem \ref{thm:KXYUK1} and Theorem \ref{thm:KXYdelta} appear to be better since they do not contain the norm of any matrix, for example, the possibly large factor $\norm{K_{S_1S_2}^+}$ in Theorem \ref{thm:KXYerror} and Theorem \ref{thm:deltaKXYerror}.
This factor disappears when only \emph{one} geometric selection is performed (for either rows or columns), as shown in Corollary \ref{cor:KXYerror} and Theorem \ref{thm:KXYUK1}.

\subsection{The symmetric case} 
\label{sub:Special case}
In this section, we consider a variant of the approximation \eqref{eq:1sided} when the kernel matrix $K_{XX}=[\kappa(x,y)]_{x,y\in X}$ is associated with \emph{one} set of points $X$ and a symmetric kernel $\kappa(x,y)$.  
This type of kernel matrix arises frequently as covariance or correlation matrices in statistics and machine learning.
In order to preserve the symmetry of $K_{XX}$, we compute a symmetric factorization of the form
\begin{equation}
\label{eq:symfac}
K_{XX}\approx UK_{\iii\iii}U^{T} \quad\text{with}\quad \iii\subseteq X
\end{equation}
whose structure-preserving properties are shown in the next proposition. 
This is the symmetric version of the ``double-sided ID'' \cite{martinsson_tropp_2020}.

\begin{proposition}
    \label{prop:symmetric}
    If $K_{XX}$ is symmetric, then the low-rank approximation $UK_{\iii\iii}U^{T}$ in \eqref{eq:symfac} is also symmetric.
    If $K_{XX}$ is assumed to be positive semi-definite, then $UK_{\iii\iii}U^{T}$ is also positive semi-definite.
\end{proposition}
\begin{proof}
    Since $\iii\subseteq X$, $K_{\iii\iii}$ is a principal submatrix of $K_{XX}$.
    If $K_{XX}$ is symmetric, $K_{\iii\iii}$ is also symmetric, which implies that $UK_{\iii\iii}U^T$ is symmetric. 
    
    If $K_{XX}$ is positive semi-definite,
    then $K_{\iii\iii}$ is also positive semi-definite since it is a principal submatrix of $K_{XX}$.
    As a result, $UK_{\iii\iii}U^T$ is symmetric positive semi-definite.
\end{proof}

The symmetric factorization in \eqref{eq:symfac} is a straightforward extension of the one-sided factorization and the algorithm is summarized in Algorithm \ref{alg:fac0}. 
\begin{algorithm}
    \caption{\it Data-driven compression of $K_{XX}$ with one set of points $X$}
    \label{alg:fac0}
    \emph{Input:} Dataset $X=\{x_{1},\dots,x_{n}\}\subset \mathbb{R}^d$, kernel function $\kappa$, number of sample points $r$\\
    \emph{Output:} Low-rank approximation $K_{XX}\approx UK_{\iii\iii}U^T$ 
        \begin{algorithmic}
        \STATE Apply a linear complexity geometric selection algorithm to $X$ to generate $r$ sample points $S$
        \STATE Apply \cdf{SRRQR-based ID} to the $n$-by-$r$ kernel matrix $K_{XS}$:
\begin{equation}
    \label{eq:QR1}
    K_{XS} = [\kappa(x,y)]_{\substack{x\in X\\y\in S}} 
    = P \begin{bmatrix}
        I\\ 
        G
    \end{bmatrix}
    K_{\iii S},
\end{equation}
where $\iii\subseteq X$, $P$ is a permutation matrix \cdf{that maps the row indices in $I$ to the indices for $\iii$ in $X$,} and $\norm{G}_{\max}\leq 2$    
        \STATE Define 
        $U = P \begin{bmatrix}
        I\\ 
        G
    \end{bmatrix}$
        \STATE Return $U, K_{\iii\iii}$
    \end{algorithmic}
\end{algorithm}


\section{Numerical experiments} 
\label{sec:test}

In this section, we illustrate the data-driven geometric approach using
both low- and high-dimensional data.  All experiments were conducted in
MATLAB R2021a on a MacBook Pro with Apple M1 chip and 8GB of RAM.

\subsection{Data on a manifold in three dimensions}
\label{sub:low}
For data in low dimensional ambient space, e.g., $d=3$, there exist
several effective methods for compressing kernel matrices.  However,
their efficiency may decrease when the separation between the sets
$X$ and $Y$ decreases and when the data lies on a manifold rather
than be distributed relatively uniformly in the ambient space.
To illustrate the advantages of the geometric approach in these
cases, we use a sequence of three datasets as illustrated in 
Figure \ref{fig:3Dsets}.  In each dataset, $X$ and $Y$ each contain 1400 points, with 400 on each small cube and 600 on the hemisphere in Figure \ref{fig:3Dsets}.
From Dataset 1 to 3, $Y$ is a vertical shift of $X$ by 2.7,
2, and 0.5, respectively.  The minimum distance between points in $X$
and points in $Y$ from Datasets 1 to 3 is 1, 0.43, 0.12, respectively.  The smallest
bounding boxes for $X$ and $Y$ are well-separated in Dataset 1, adjacent
in Dataset 2, and overlapping in Dataset 3.
With these data, kernel matrices were constructed using the
kernel function $\kappa(x,y)={1}/{|x-y|}$.

\begin{figure}[htbp] 
    \centering 
     \subfigure[Dataset 1]{\label{fig:3Dset1}\includegraphics[scale=0.30]{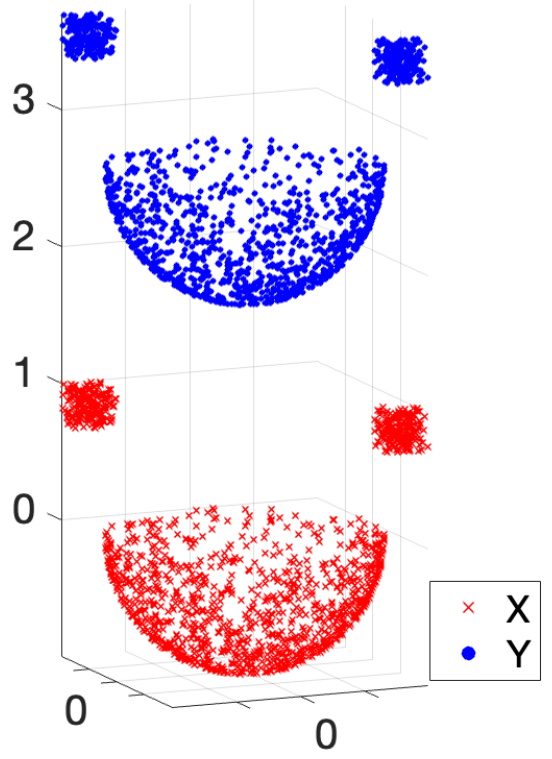}}
    \hspace{0.2in}
     \subfigure[Dataset 2]{\label{fig:3Dset2}\includegraphics[scale=0.30]{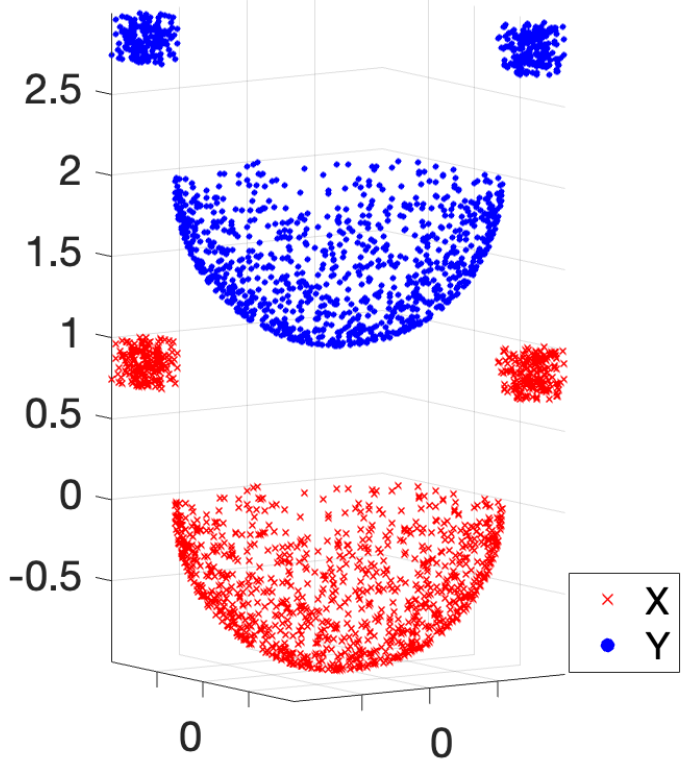}}
    \hspace{0.2in}
     \subfigure[Dataset 3]{\label{fig:3Dset3}\includegraphics[scale=0.30]{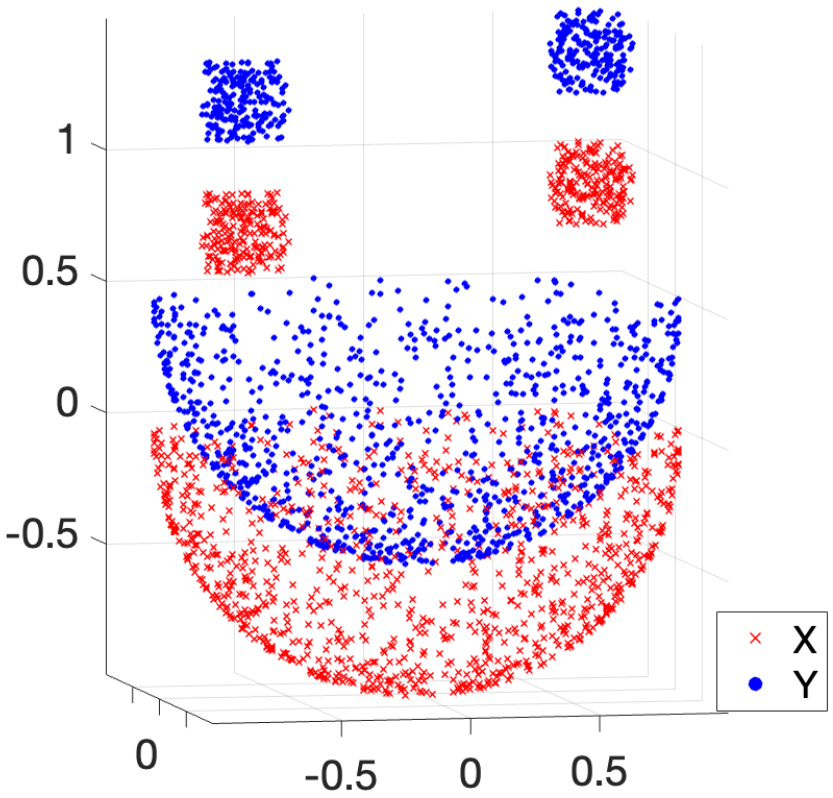}}
    \caption{Sequence of three datasets in three dimensions. From Datasets 1 to 3, $Y$ is a vertical shift of $X$ by 2.7,
2, and 0.5, respectively.  The minimum distance between points in $X$
and points in $Y$ from Dataset 1 to 3 is 1, 0.43, 0.12, respectively. }
    \label{fig:3Dsets} 
\end{figure}

\textbf{Test 1. Robustness with respect to data geometries.}
For above the settings,
we compare the approximation error of the data-driven geometric approach
with that of an algebraic method, ACA \cite{bebendorf2003ACA},
and proxy point method (`proxy') \cite{xing2020interpolative}.
For the data-driven method (`DD'), we construct
a one-sided factorization (Algorithm \ref{alg:fac1}) using farthest point sampling with sample size at most $2r$ for a rank-$r$ approximation. 
\cdf{Namely, $2r$ points are chosen for $S_2$ and SRRQR is applied to $K_{X S_2}$.}
For the proxy point method, the sample size is 2000 for $\Omega_X$ (the smallest bounding box containing $X$) and 10000 for $\Omega_Y$, independent of the approximation rank.
Figure \ref{fig:3DsetsError} shows, for the various
methods, the relative matrix approximation error as a function of the rank of the approximation.  
The relative error is defined as ${||K-\tilde{K}||}/{||K||}$, where $\tilde{K}$ denotes the low-rank approximation to $K$ and $||\cdot||$ denotes the 2-norm. 
The optimal relative approximation error as computed by the SVD is also shown.

\begin{figure}[htbp] 
    \centering 
    \subfigure[Dataset 1]{\label{fig:3Dset1Error}\includegraphics[scale=.26]{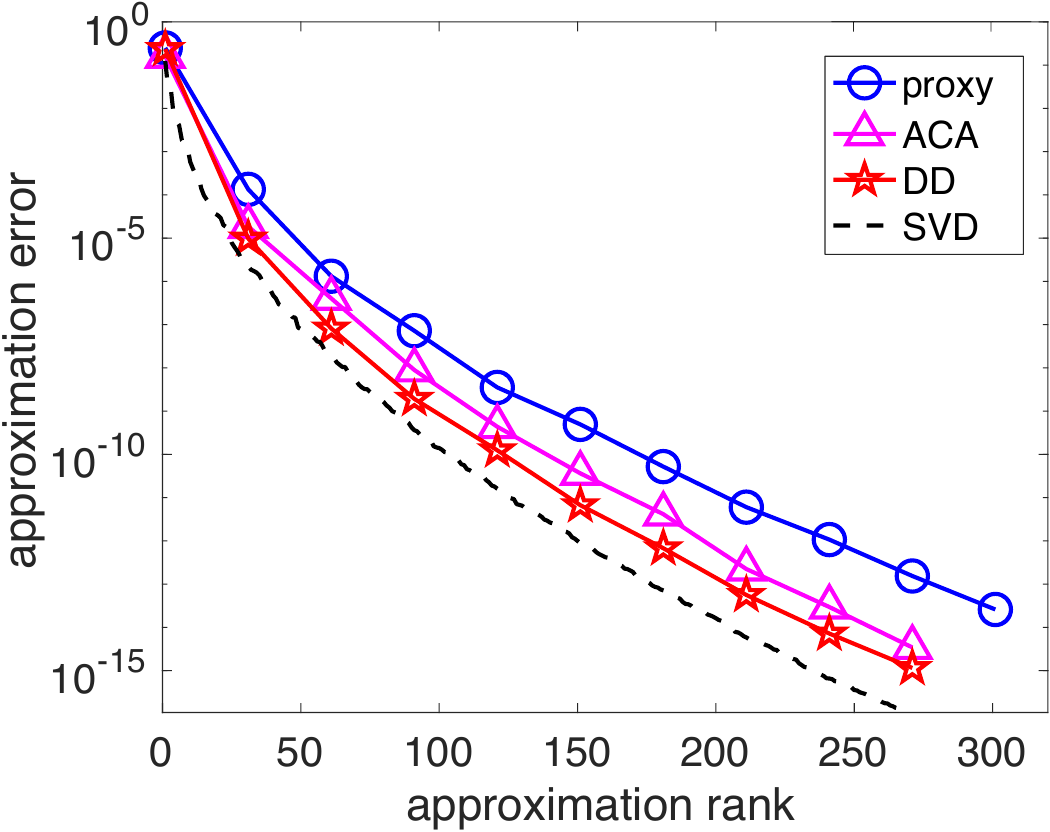}}
    \hspace{0.2in}
    \subfigure[Dataset 2]{\label{fig:3Dset2Error}\includegraphics[scale=.26]{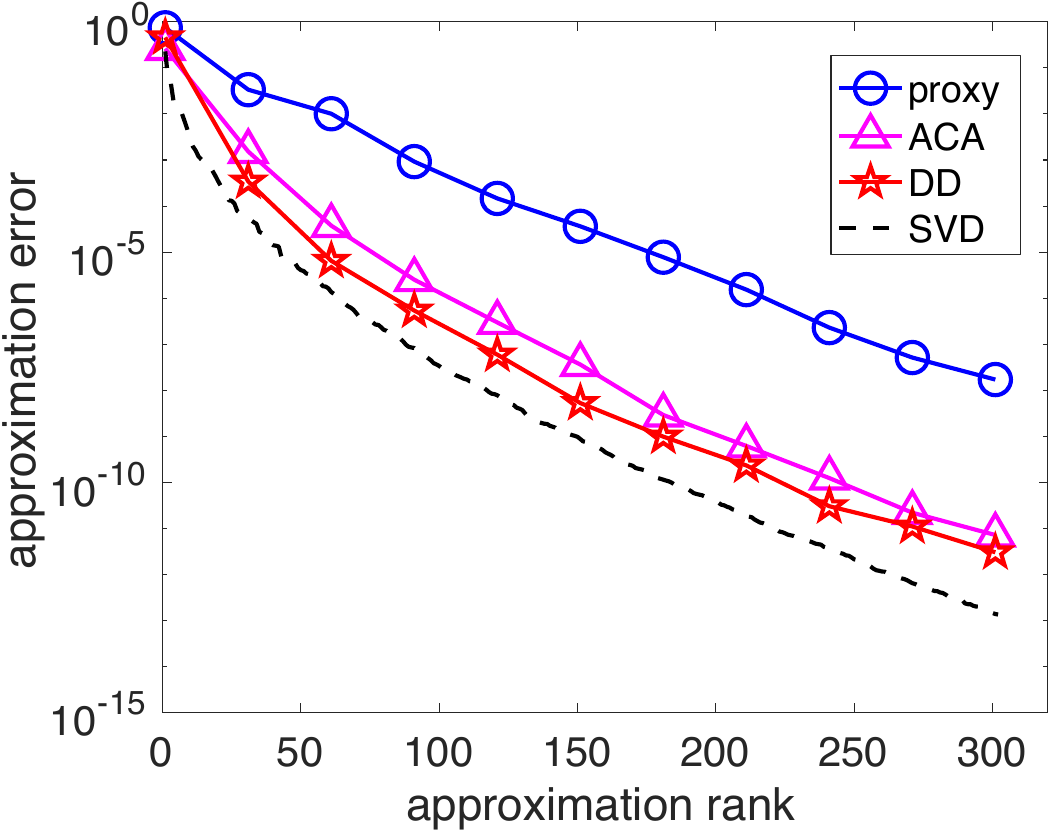}}
    \hspace{0.2in}
    \subfigure[Dataset 3]{\label{fig:3Dset3Error}\includegraphics[scale=.26]{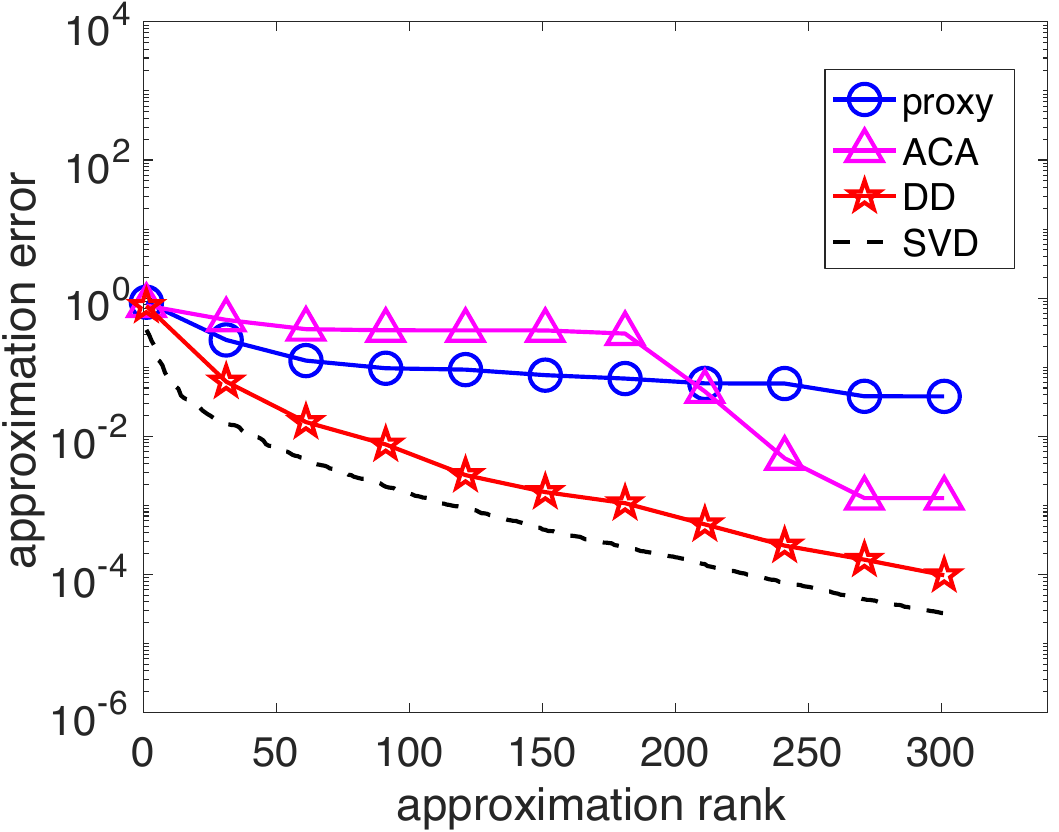}}
    \caption{Accuracy comparison of different methods for constructing low-rank factorizations on the kernel matrices defined by Datasets 1, 2, and 3 shown in Figure \ref{fig:3Dsets} and the kernel function $\kappa(x,y)={1}/{|x-y|}$.}
    \label{fig:3DsetsError}
\end{figure}

We observe that all methods are effective for Dataset 1, with the data-driven method - DD - being the most efficient and most closely tracking the SVD approximation error. 
We remark that \cdf{the large number of random samples in the bounding box} is not effectively used in proxy point method when the data is unstructured.
Hence it is computationally more expensive than DD and ACA in this experiment.

For Dataset 2, we are at the boundary at which hybrid methods
are effective, i.e., those methods that assume a separation of the
bounding boxes for $X$ and $Y$. However, DD and ACA still closely track
the SVD approximation error.

For Dataset 3, the points in $X$ and $Y$ are actually ``intermingled''
(overlapping bounding boxes). ACA might not effectively sample $X$ and
$Y$ in this dataset, especially since the dataset contains disjoint
clusters (points on a half-shell and points in small cubes). However,
DD continues to closely track the SVD approximation error.



\subsection{Data in high-dimensional ambient space}
\label{sub:high}

The data-driven geometric approach can be efficient for data in 
high-dimensional ambient space, whereas many other existing low-rank compression
methods have cost that is exponentially dependent on the dimensionality
of ambient space. To demonstrate the data-driven approach for high-dimensional
data, we use two datasets from the UCI machine learning 
repository:\footnote{\url{https://archive.ics.uci.edu/ml/index.php}}
Covertype ($n=581,012$, $d=54$) and Gas Sensor Array Drift ($n=13,910$, $d=128$).
Each dataset is standardized to have mean zero and variance along each dimension
equal to one.
Instead of using the entire datasets, we choose $X$ and $Y$ to be
two subsets of random samples, selected without replacement,
with 8000 points for $X$ and 10000 points for $Y$.

For both datasets,
the kernel matrix $K_{XY}$ is a $8000$-by-$10000$ matrix associated with
the Gaussian kernel $\kappa(x,y)=\exp(-|x-y|^2 / \sigma^2)$.  Since the
bandwidth parameter $\sigma$ controls the smoothness of the kernel,
we consider $\sigma$ from among three values, denoted as $\sigma_1, \sigma_2, \sigma_3$ and chosen to be 100\%, 50\%, 25\% of the radius
of $X$, respectively.

\textbf{Test 2. Different geometric selection schemes.}
We first examine the effect of different geometric selection schemes (see Section \ref{sec:intro}) used to construct the two-sided low-rank factorization \eqref{eq:2sided}.  
It has been observed in \cite{anchornet} that the approximation error can be extremely large if the subset is not chosen properly.
In this case, to deal with the pseudoinverse, a stable implementation proposed in \cite{yuji20} was used in \cite{anchornet}:
\begin{equation}
\label{eq:nysQR}
    K_{XY}\approx (K_{XS_2}R_{\epsilon}^+)( Q^T K_{S_1Y}),
\end{equation}
where  $K_{S_1S_2}=QR$ is the QR factorization of $K_{S_1S_2}$ and $R_{\epsilon}$ is derived from $R$ by truncating singular values smaller than $\epsilon$ in the SVD.
It is also noted in \cite{anchornet} that the above stabilization is \emph{not} needed if the subset is well-chosen, i.e. spread evenly over the data.
\cdf{For a detailed numerical study on the effect of stabilization, we refer to \cite[Section 5.2]{anchornet}.}
To make a fair comparison of different selection schemes,
the stabilization in \eqref{eq:nysQR} is used in this experiment.
Namely, in Algorithm \ref{alg:fac2}, $K_{S_1 S_2}$ is replaced with its truncated QR factors: $Q$ and $R_{\epsilon}^+$.

For these experiments, we use the Gas Sensor dataset with a Gaussian kernel
with $\sigma=\sigma_1\approx 307.52$.

Figure \ref{fig:d128sig1DD2subset} shows the low-rank approximation error (relative error as in Figure \ref{fig:3DsetsError}) vs.\ approximation rank when different geometric selection schemes are used.  
We compare the following schemes: anchor net (`ANC'), farthest point sampling (`FPS'), uniform random sampling (`Unif'),
as well as mixtures of uniform random sampling and FPS (some points are generated by FPS, the rest by random sampling).
In the mixed scheme, random sampling is used to reduce the cost of FPS and we use the experiment to observe the effect of augmenting FPS samples with random samples.
In these cases, the mixtures are denoted `mixed1', `mixed2', `mixed3',
for 5\%, 10\%, 50\% FPS samples with the remainder of the samples selected by uniform random sampling.
We observe that ANC and FPS perform similarly. These methods have a clear advantage over pure random sampling, suggesting that the data has structure to be exploited that is hidden from pure random sampling. However, random sampling can reduce the cost of a pure FPS method.
The accuracy of ANC and FPS is attributed to the use of evenly spaced points.
The generation of evenly distributed points is studied in discrepancy theory for the unit cube (cf. \cite{QMC1992book,UD2012book}) and recently extended to general geometries using deep neural networks (cf. \cite{dcJCP2022,autm}).
We also note that the approximation error for `DD2-Unif' does not improve after approximation rank about 50, creating the ``flat'' portion in the plot.
In fact, this is due to the stabilization \eqref{eq:nysQR} that prevents the approximation error from ``blowing up''.

\begin{figure}[htbp]
\centering
\includegraphics[scale=.32]{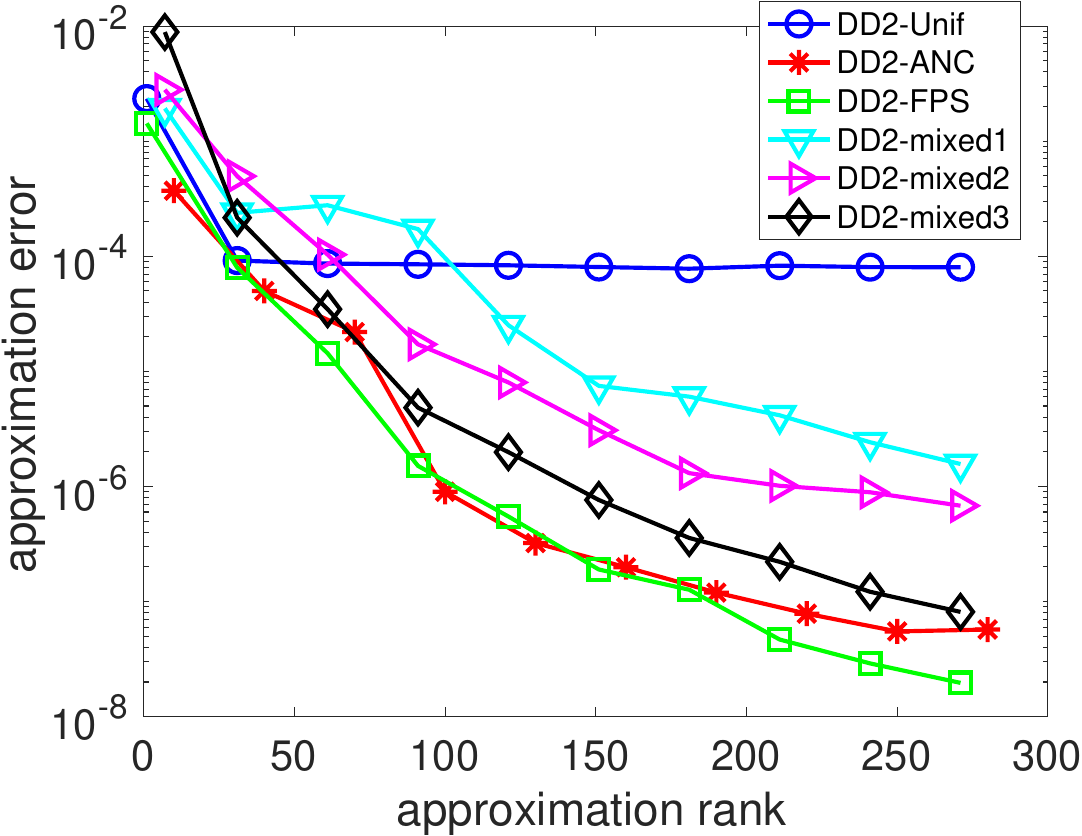}
\caption{Accuracy comparison of different geometric selection schemes
for constructing two-sided data-driven low-rank factorizations on the kernel matrix defined by the Gas Sensor dataset (d = 128) and
a Gaussian kernel with the bandwidth $\sigma_1 \approx 307.5$.
}
\label{fig:d128sig1DD2subset}
\end{figure}

\textbf{Test 3. One-sided vs.\ two-sided low-rank factorization.}
With the same dataset and kernel as immediately above, we compared
the approximation error for one-sided and two-sided low-rank factorization.
Figure \ref{fig:d128sig1DDcompare} shows these results.
The one- and two-sided cases are denoted as `DD1' and `DD2', respectively,
and each is tested with `Unif', `ANC', and `FPS' geometric selection.

We observe that one-sided factorization is generally more accurate.
This is consistent with the theoretical results in Theorem
\ref{thm:KXYerror} and Theorem \ref{thm:KXYUK1}, where the two-sided
approximation error estimate contains the norm of a pseudoinverse
matrix while the one-sided approximation estimate doesn't contain
any matrix norm.  

We notice again the stagnating accuracy for `DD2-Unif' when the approximation rank is larger than 50. 
On the contrary, ``DD1-Unif'' gives effective error reduction as the approximation rank increases.
The difference reveals the fact that the two-sided factorization is \emph{not} as numerically stable as the one-sided factorization.
It should be noted, however, that the one-sided factorization is more expensive to compute than the two-sided factorization.
The one-sided factorization uses geometric selection on only $X$ or only $Y$ rather than both $X$ and $Y$
and applies algebraic compression to a much larger intermediate matrix
than the two-sided factorization.

\begin{figure}[htbp]
\centering
\includegraphics[scale=.32]{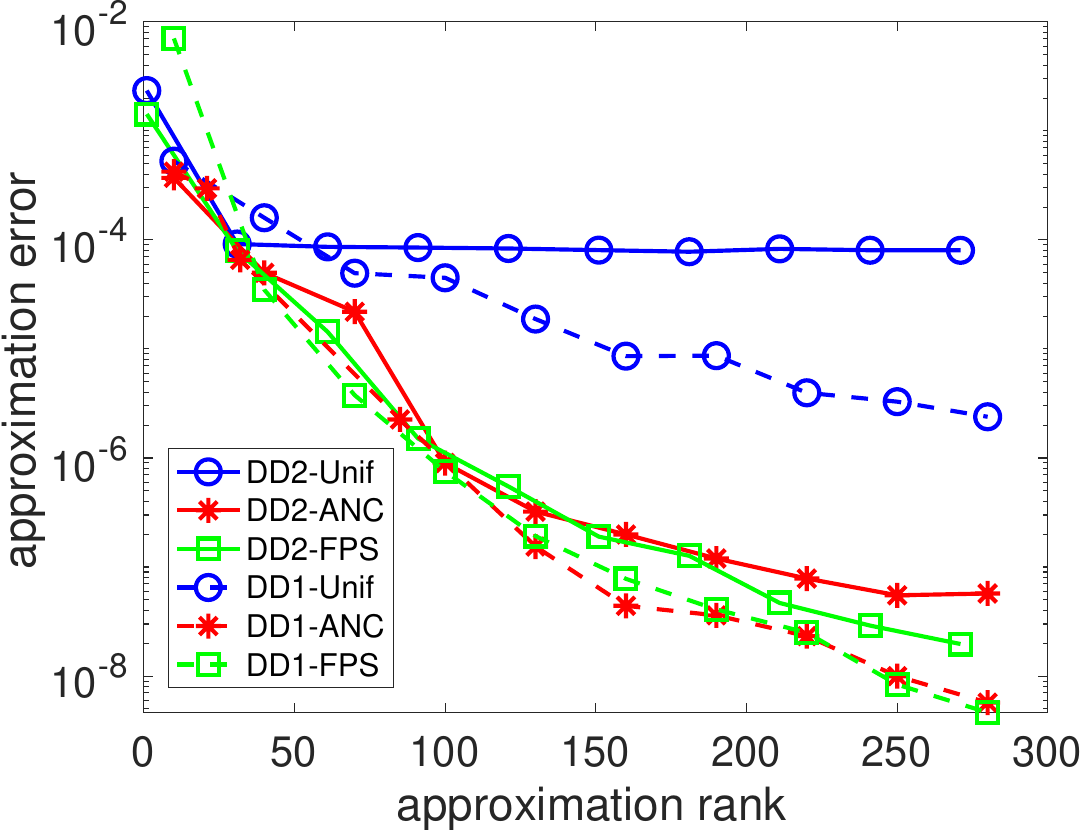}
\caption{Accuracy comparison of one-sided (dashed lines) vs.\ two-sided (solid lines) data-driven factorizations on the kernel matrix defined by the Gas Sensor dataset (d = 128) and
a Gaussian kernel with the bandwidth $\sigma_1 \approx 307.5$.}
\label{fig:d128sig1DDcompare}
\end{figure}

\textbf{Test 4. Comparison with ACA and robustness with respective to kernel parameters.}
There exist a few low-rank compression algorithms that are both efficient in the high-dimensional case and able to handle intermingled data. 
One notable method is the ACA method \cite{bebendorf2003ACA}, which does not require access to the full kernel matrix, has linear complexity, and produces a one-sided factorization.  We thus now compare the data-driven geometric approach
with ACA.  
Specifically, we use the data-driven approach to compute
a one-sided low-rank factorization with ANC and FPS geometric selection schemes, corresponding to ``DD-ANC'' and ``DD-FPS'' in Figures \ref{fig:d54error} to \ref{fig:d128time}.

Figures \ref{fig:d54error} and \ref{fig:d128error} show the approximation errors for the Covertype and Gas Sensor datasets, respectively, each with three values of bandwidth $\sigma$ for the Gaussian kernel.  
For the Covertype dataset, the data-driven methods yield much better accuracy than ACA for all choices of $\sigma$.
The accuracy of ACA stagnates as the approximation rank is increased, suggesting that clusters in the dataset have prevented ACA from selecting rows and columns that help represent the kernel matrix.

For the Gas Sensor dataset, all methods behave similarly for a
large bandwidth $\sigma_1$. For a smaller bandwidth $\sigma_3$,
ACA displays stagnation in accuracy for approximation rank
greater than 100. The smaller bandwidth makes the Gaussian kernel
less smooth, and accentuates the effect of clusters in the data.
These issues in ACA have been explored previously \cite{HCA2005,darve2019}. 
In general, these issues reflect the challenge in approximating
kernel matrices associated with high ambient dimensions but possibly
lower dimensional structures within these dimensions.

Figure \ref{fig:d128time} shows timings for computing the low-rank
approximations.  Although our timings are limited to MATLAB execution,
the results indicate that a geometric approach using a fast
geometric selection scheme (e.g., ANC) can be faster than ACA
for the same approximation error. The results also suggest that
ANC is significantly faster than FPS.


\begin{figure}[htbp]
    \centering
    \subfigure[$\sigma=\sigma_1\approx 74.46$]{\label{fig:d54sig1error}\includegraphics[scale=.26]{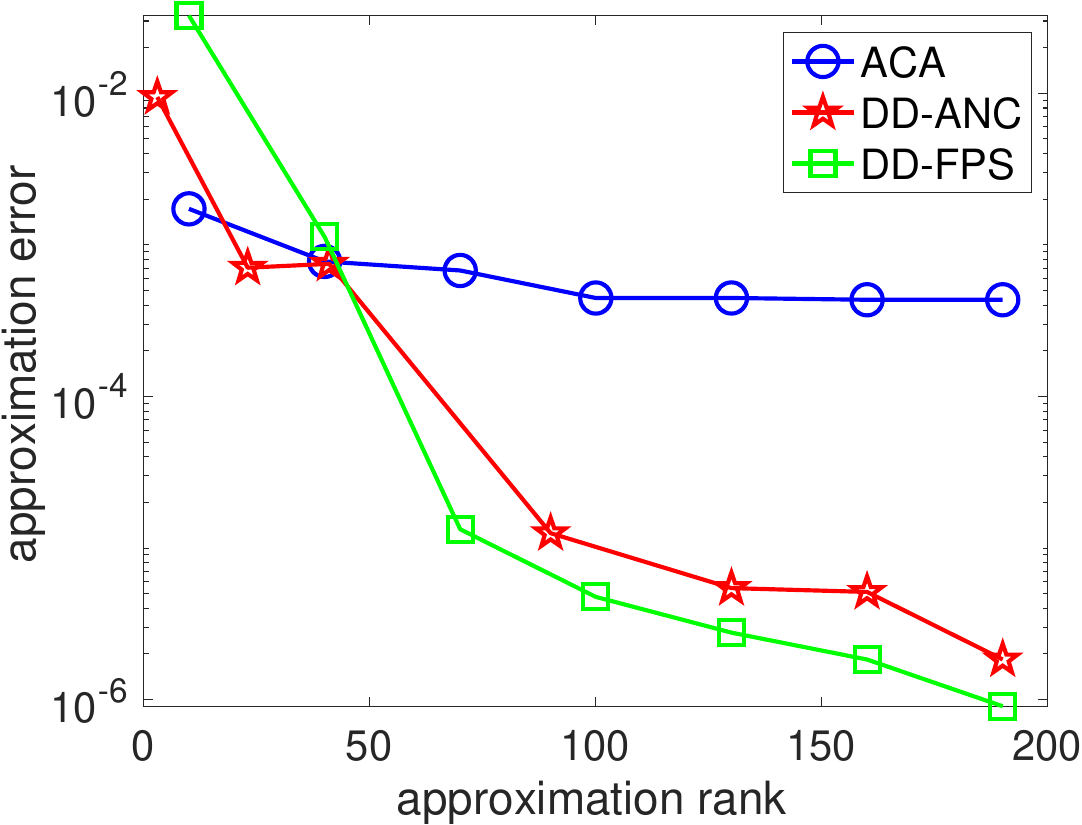}}
    \subfigure[$\sigma=\sigma_2\approx 37.23$]{\label{fig:d54sig2error}\includegraphics[scale=.26]{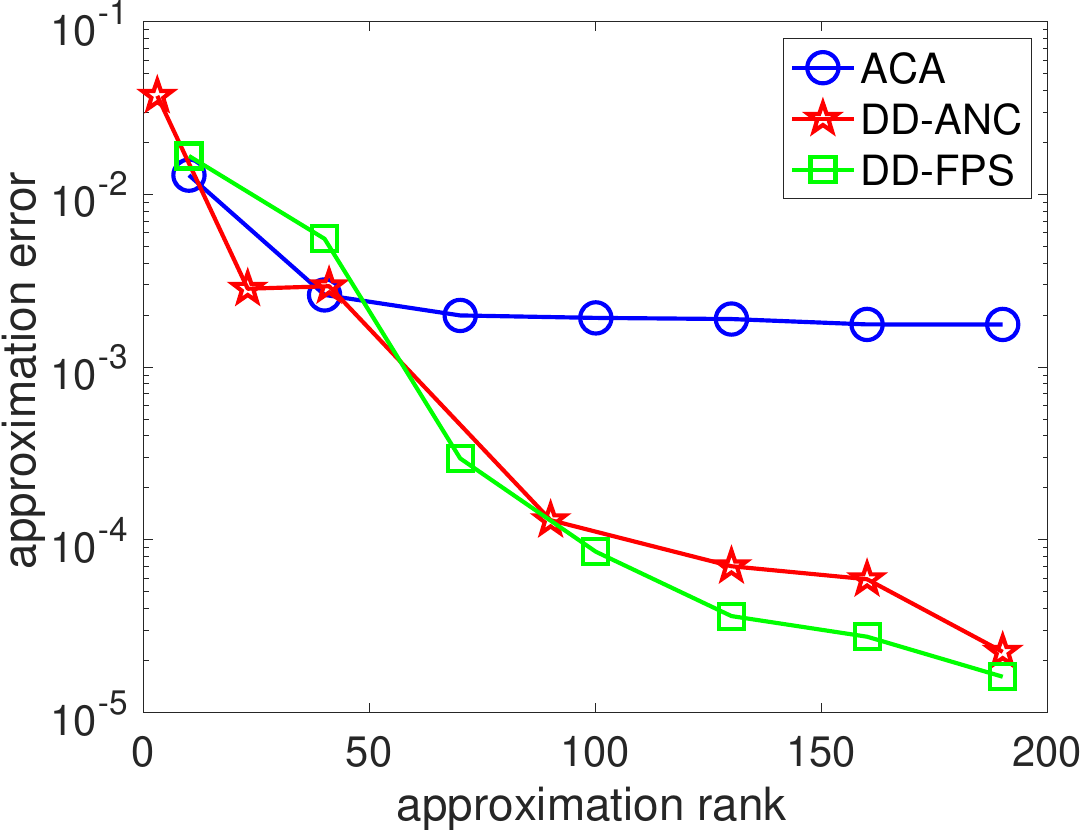}}
    \subfigure[$\sigma=\sigma_3\approx 14.89$]{\label{fig:d54sig3error}\includegraphics[scale=.26]{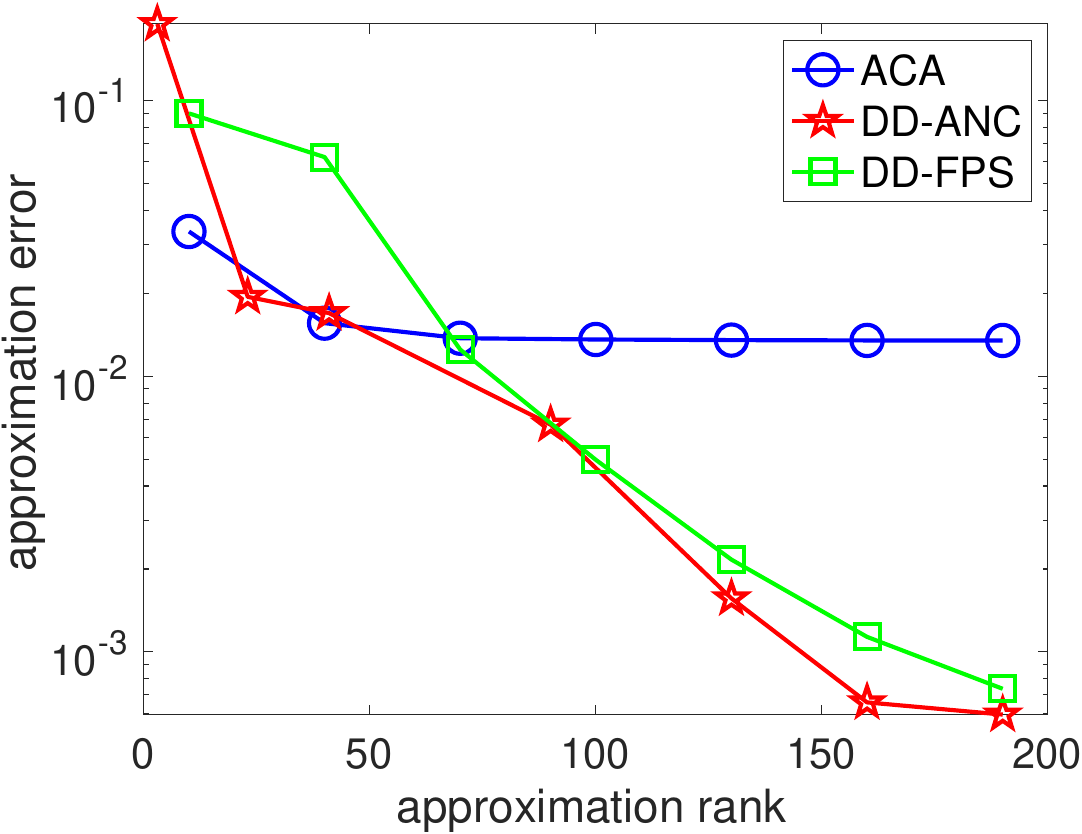}}
    \caption{Accuracy comparison of one-sided data-driven factorizations (DD-ANC and DD-FPS) with ACA on  kernel matrices defined by the Covertype dataset (d=54) and  Gaussian kernel with three different bandwidths $\sigma$.}
    \label{fig:d54error} 
\end{figure}


\begin{figure}[htbp]
    \centering
    \subfigure[$\sigma=\sigma_1\approx 307.52$]{\label{fig:d128sig1error}\includegraphics[scale=.26]{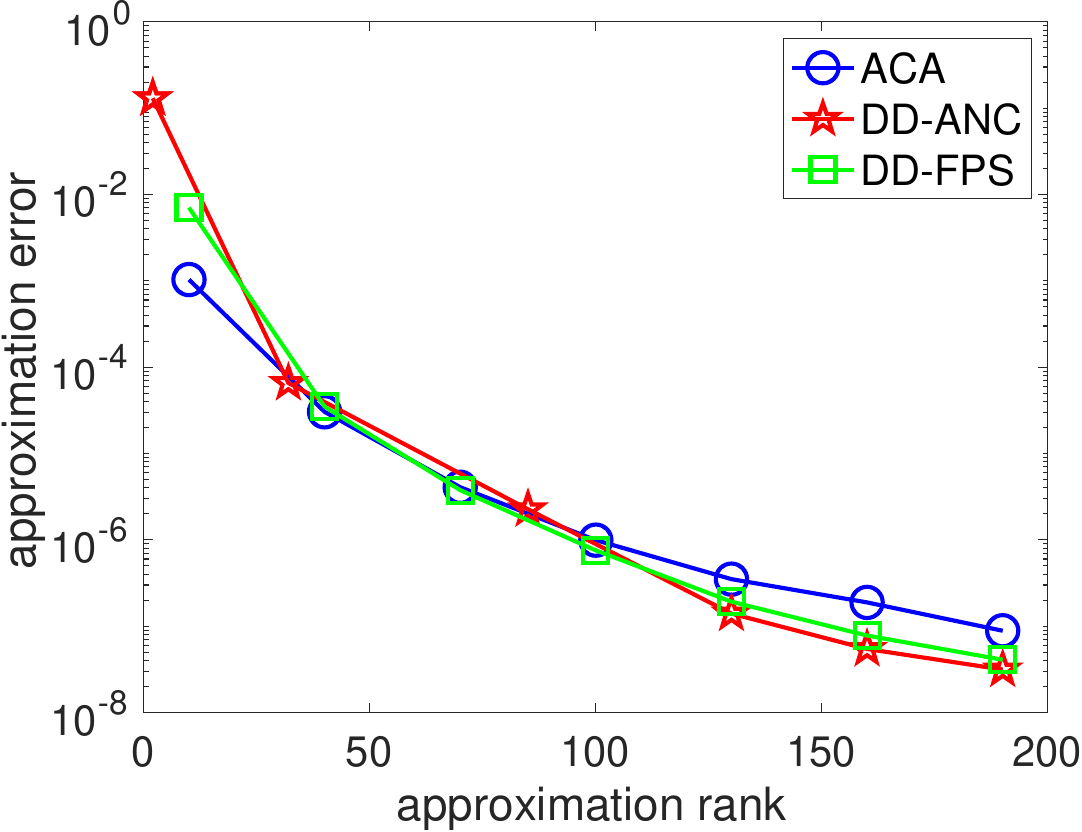}}
    \subfigure[$\sigma=\sigma_2\approx 153.76$]{\label{fig:d128sig2error}\includegraphics[scale=.26]{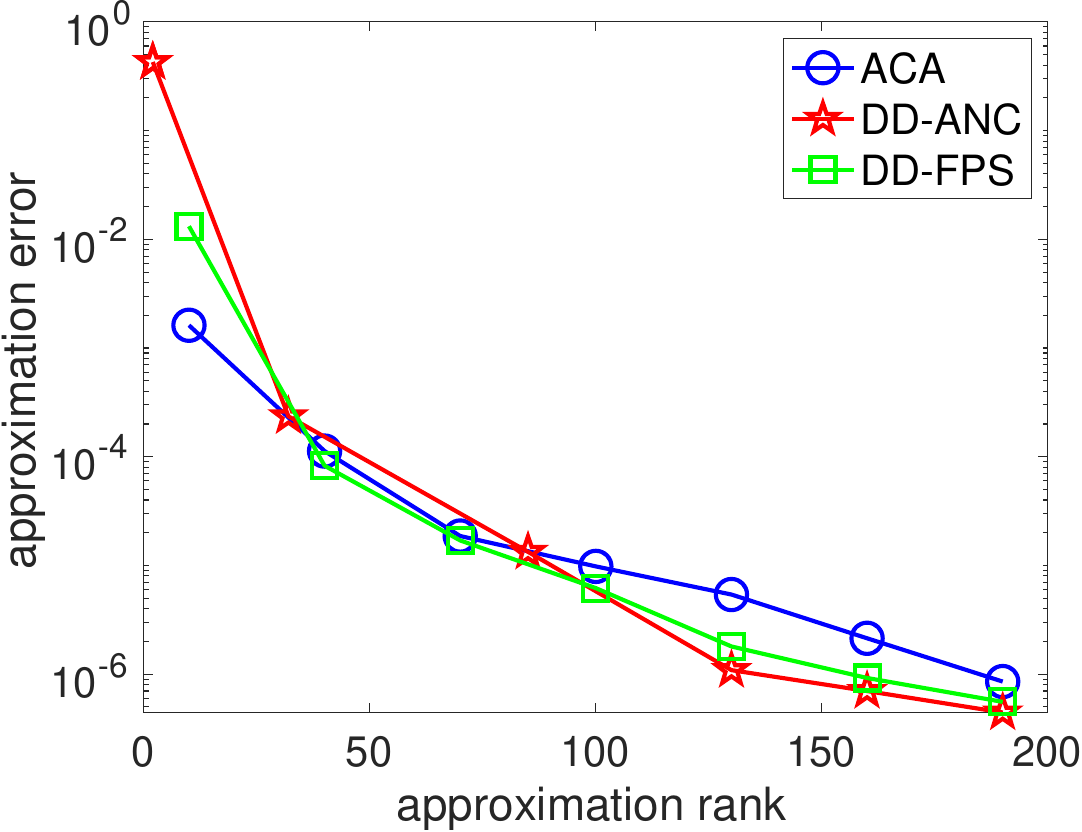}}
    \subfigure[$\sigma=\sigma_3\approx 61.50$]{\label{fig:d128sig3error}\includegraphics[scale=.26]{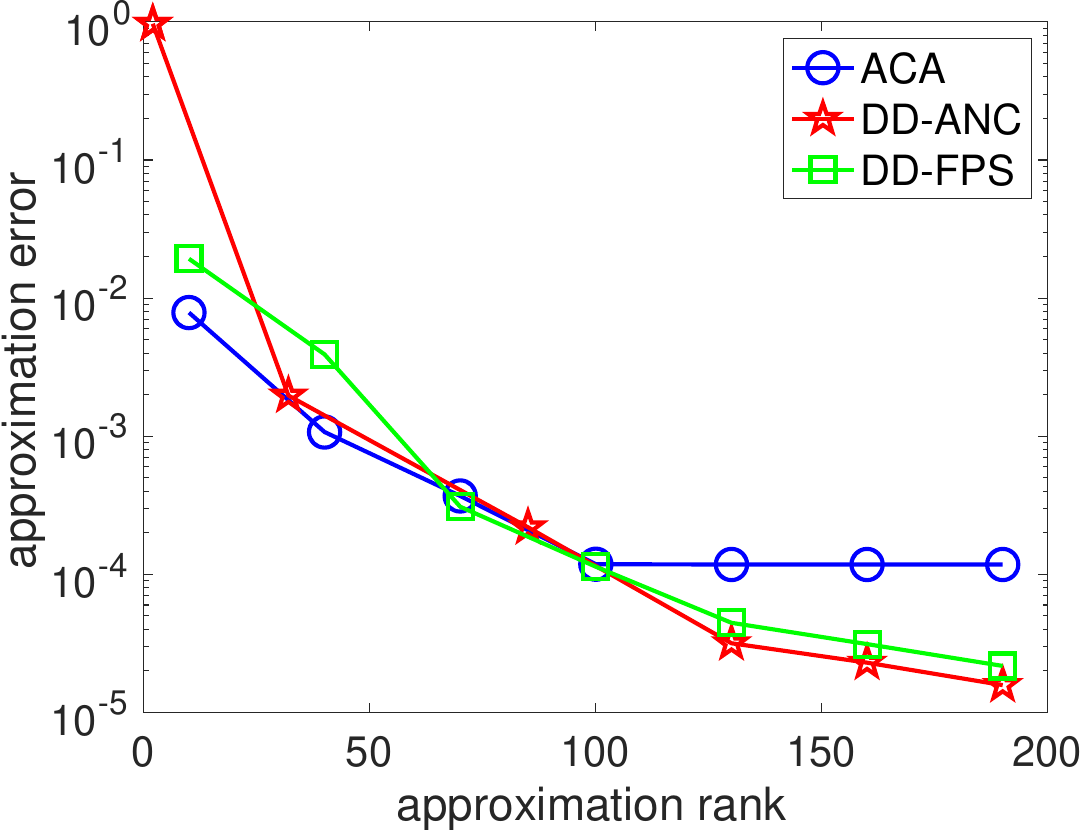}}
    \caption{Accuracy comparison of one-sided data-driven factorizations (DD-ANC and DD-FPS) with ACA on the kernel matrices defined by the Gas Sensor dataset (d=128) and the Gaussian kernel with three different bandwidths $\sigma$.}
    \label{fig:d128error}
\end{figure}

\begin{figure}[htbp]
    \centering
    \subfigure[$\sigma=\sigma_1\approx 307.52$]{\label{fig:d128sig1time}\includegraphics[scale=.26]{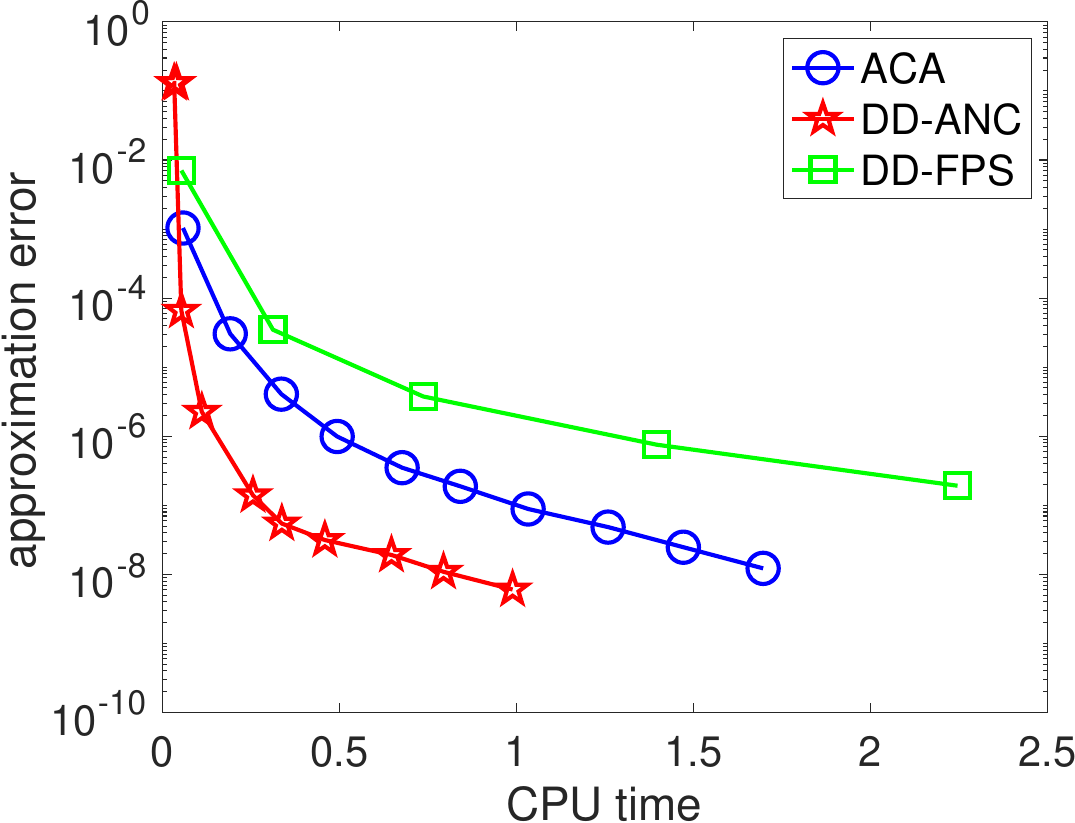}}
    \subfigure[$\sigma=\sigma_2\approx 153.76$]{\label{fig:d128sig2time}\includegraphics[scale=.26]{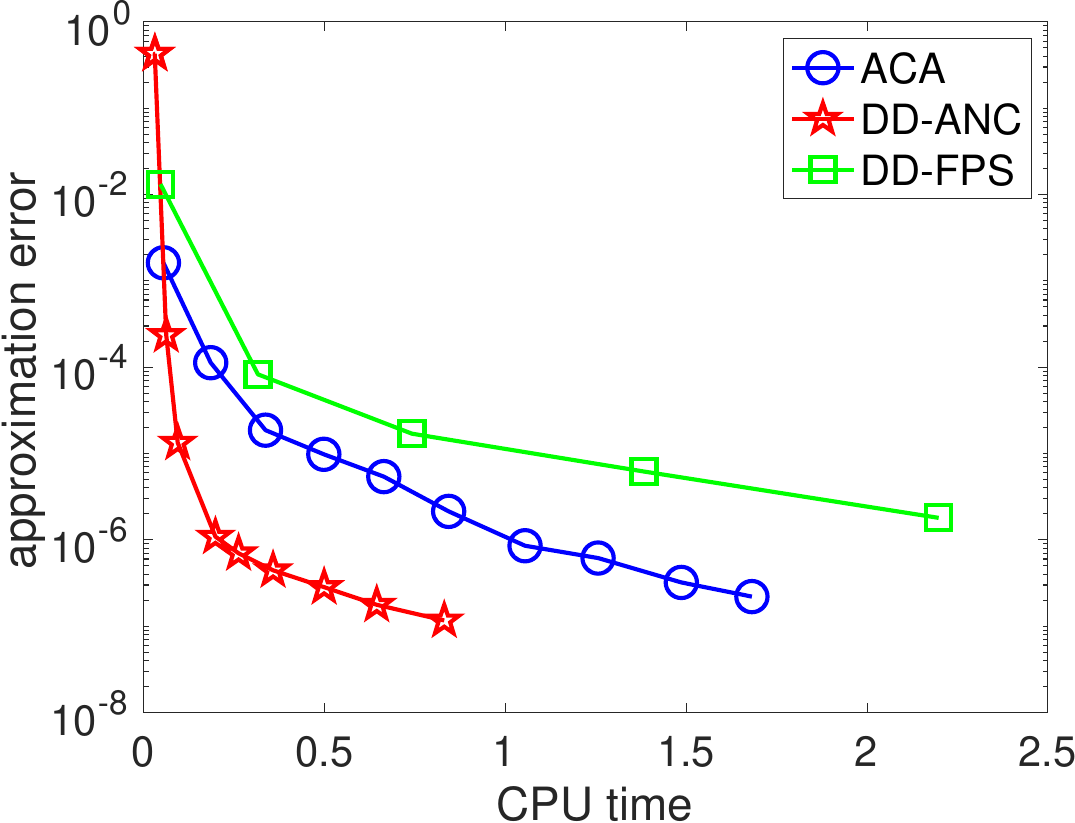}}
    \subfigure[$\sigma=\sigma_3\approx 61.50$]{\label{fig:d128sig3time}\includegraphics[scale=.26]{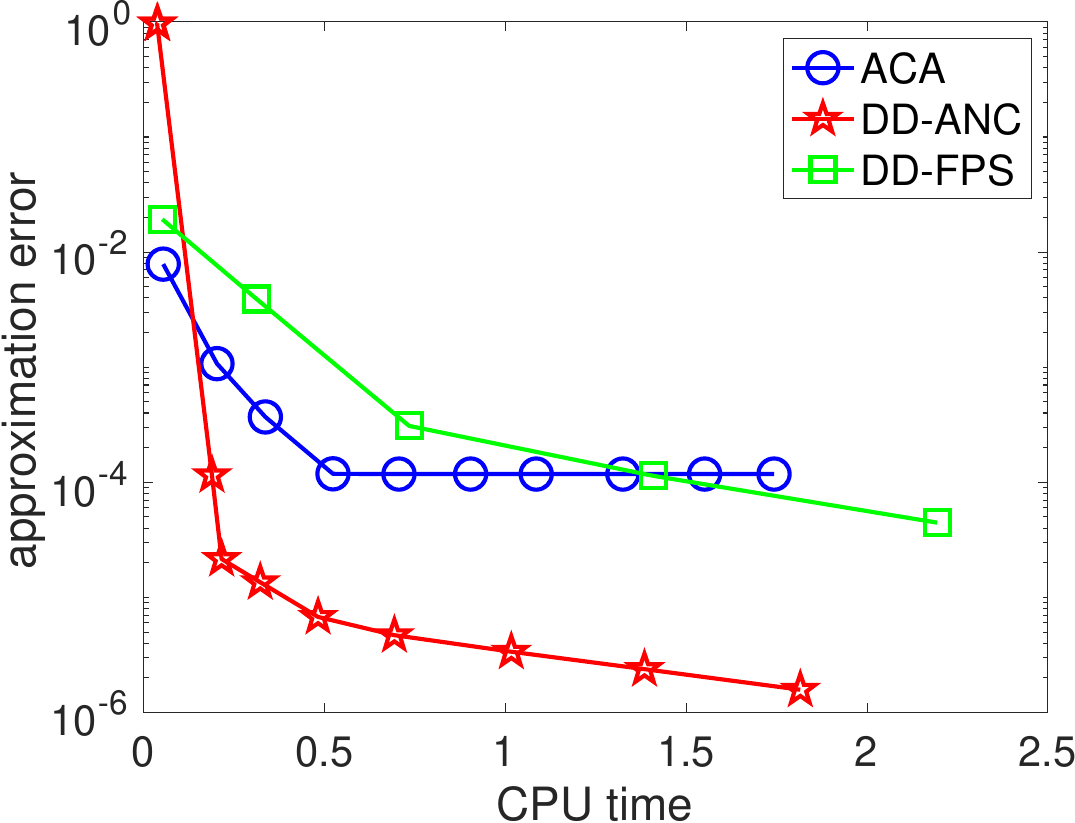}}
    \caption{Time comparison of one-sided data-driven factorizations (DD-ANC and DD-FPS) with ACA on the kernel matrices defined by the Gas Sensor dataset (d=128) and the Gaussian kernel with three different bandwidths $\sigma$. CPU timings (in seconds) are the average of 10 runs.}
    \label{fig:d128time} 
\end{figure}


\cdf{
\textbf{Test 5. Comparison of algebraically generated subsets and geometrically generated subsets.}
ACA \cite{bebendorf2000,bebendorf2003ACA} performs a column-pivoted partial LU decomposition where the pivots are chosen algebraically based on the residual of each rank-1 approximation in the sequential process.
The resulting triangular factorization (LU) is mathematically equivalent to $K_{XS_2}K_{S_1 S_2}^{-1}K_{S_1 Y}$ where $K_{S_1 S_2}$ is a square matrix and the subsets (corresponding to the pivots) $S_1$, $S_2$ are generated by ACA (cf. \cite[Lemma 3]{bebendorf2000}).
Hence we can use ACA to generate subsets for Algorithm \ref{alg:fac2} and Algorithm \ref{alg:fac1}.
Now, given the proposed geometric approach and the algebraic approach via ACA for generating subsets,
a natural question is which approach yields better performance in practice.
In this experiment, we investigate the quality of subsets by comparing the following methods:
(1) ACA; (2) one-sided factorization in Algorithm \ref{alg:fac1} with subset $S_2$ generated by ACA (`DD1-ACA'); (3) one-sided factorization with subset generated by anchor net (`DD1-ANC').
We use the Gas Sensor dataset ($d=128$) and
consider both symmetric positive definite (SPD) matrix and rectangular matrix.
The SPD matrix is the Gaussian kernel matrix $K_{XX}$ with $X$ containing 8000 random samples from the Gas Sensor dataset.
For the rectangular matrix $K_{XY}$, we choose $X$ and $Y$ to contain 8000 and 10000 samples, respectively, as in \textbf{Test 4}.
The bandwidth paragraph is chosen as $\sigma=\sigma_3\approx 61.50$.
The result is shown in Figure \ref{fig:DD1-ACA}.
\begin{figure}[htbp]
    \centering
    \subfigure[SPD Gaussian kernel matrix $K_{XX}$]{\includegraphics[scale=.3]{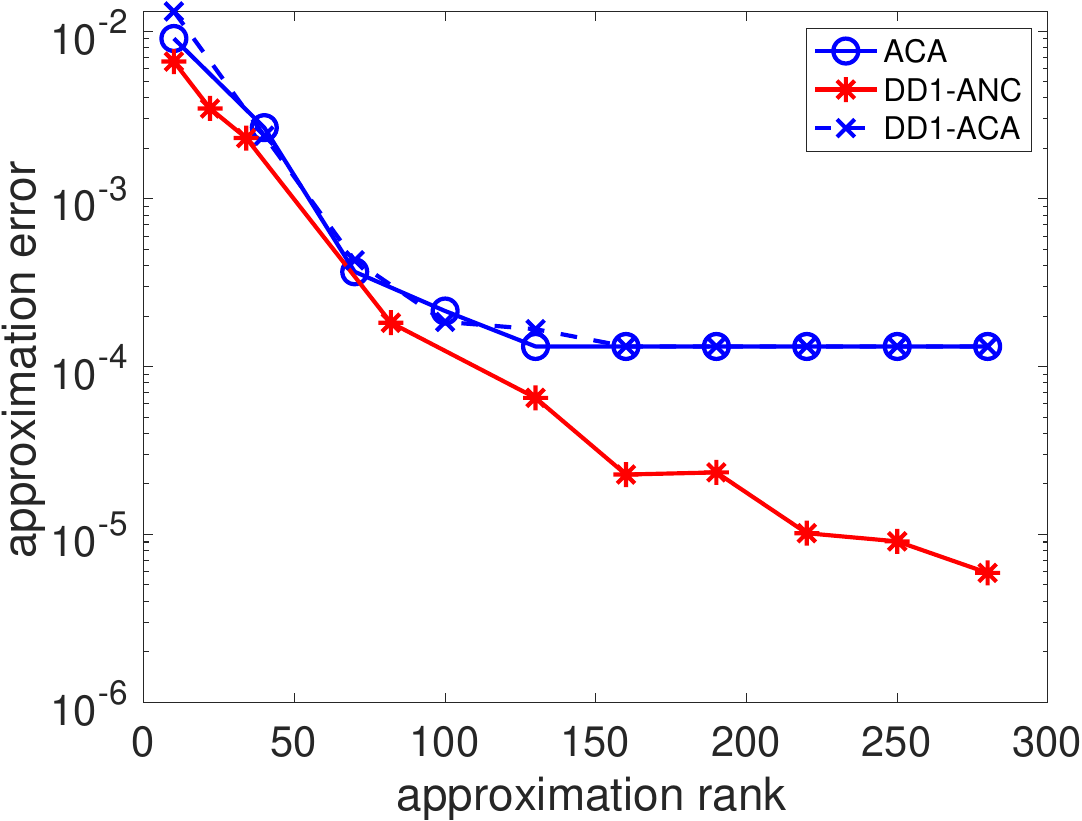}}
    \subfigure[Rectangular Gaussian kernel matrix $K_{XY}$]{\includegraphics[scale=.3]{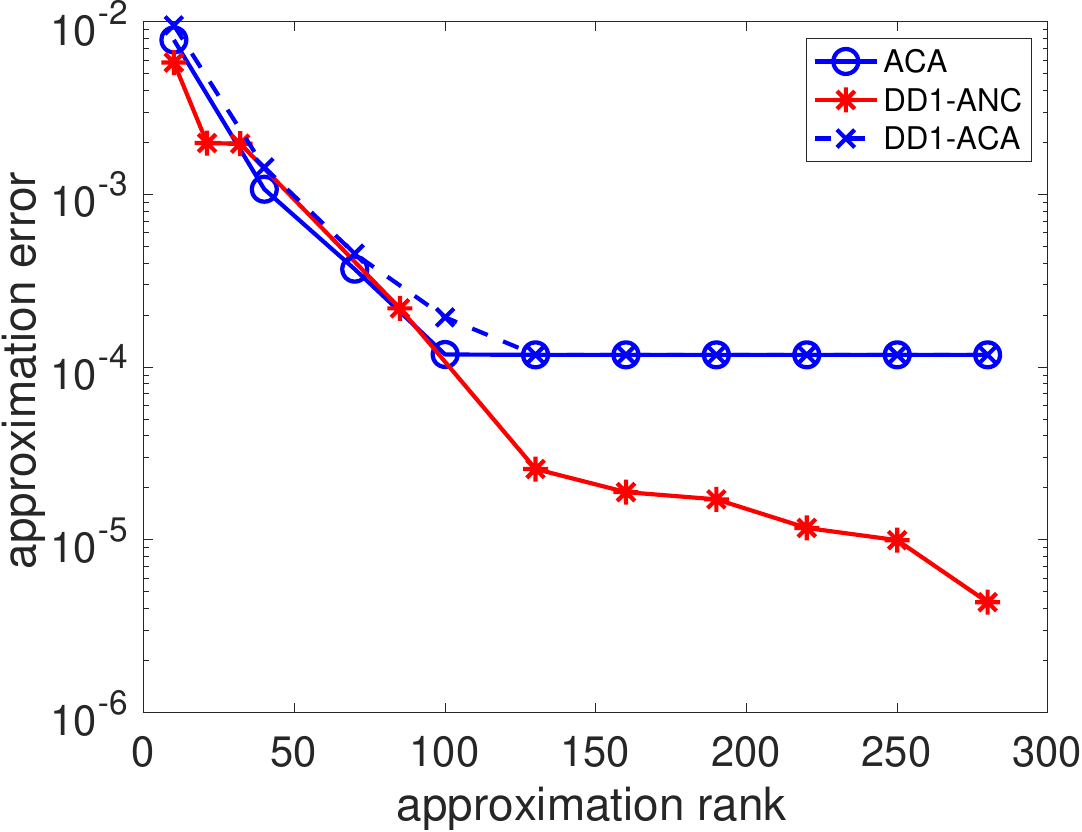}}
    \caption{Accuracy comparison of ACA, one-sided factorizations with ACA-generated subsets  (`DD1-ACA') and anchor net-generated subsets (`DD1-ANC') on the Gaussian kernel matrices with Gas Sensor dataset ($d=128$).}
    \label{fig:DD1-ACA}
\end{figure}

It is clearly seen from Figure \ref{fig:DD1-ACA} that, for this problem, ACA and `DD1-ACA' yield almost the same performance, indicating that the subset generated by ACA (or, algebraically, pivots) does \emph{not} yield an accurate low-rank approximation regardless of whether the underlying matrix is SPD.
The geometric method with anchor net, on the other hand, generates better subsets with more accurate low-rank approximations.
Considering the time efficiency demonstrated in Figure \ref{fig:d128time}, we see that the geometric approach gives overall better performance in terms of accuracy, speed, and robustness.

To provide a straightforward illustration of the subset selected by ACA, we performed an experiment in two dimensions ($X\subset \mathbb{R}^2$) for the Gaussian kernel $\kappa(x,y)=\exp(-|x-y|^2/0.25)$.
The dataset $X$ contains 400 points splitted into three clusters (left to right) with 100, 200, and 100 points respectively.
Thus the kernel matrix $K_{XX}$ is a 400-by-400 SPD matrix.
The dataset $X$ is shown as blue points in Figure \ref{fig:ACA-ANC} (b) or (c).
We compare the points generated by ACA and AnchorNet and show the corresponding low-rank approximation error measured by relative error in matrix 2-norm.
The result is shown in Figure \ref{fig:ACA-ANC}.
We see from Figure \ref{fig:ACA-ANC}(a) that when the rank is smaller than 30, ACA entirely fails to improve the approximation accuracy despite the rank increase.
To understand this from a geometric point of view, we show the scatter plots in (b) and (c) for the case when the rank equals 25. It is clearly seen that the points selected by ACA are ``locked'' in the first two clusters in $X$ with no point selected in the third cluster. 
This results in the stagnation of accuracy. 
Algebraically, this ``locking'' phenomenon is analogous to performing Gaussian elimination \emph{only} on the first two diagonal blocks of a matrix that is composed of three diagonal blocks of similar rank structure.
On the contrary, AnchorNet generates points from all three clusters in a more balanced fashion and achieves consistently better accuracy than ACA according to Figure \ref{fig:ACA-ANC} (a).

\begin{figure}[htbp]
    \centering
    \subfigure[Error-rank plot for approximating $K_{XX}$]{\includegraphics[scale=.3]{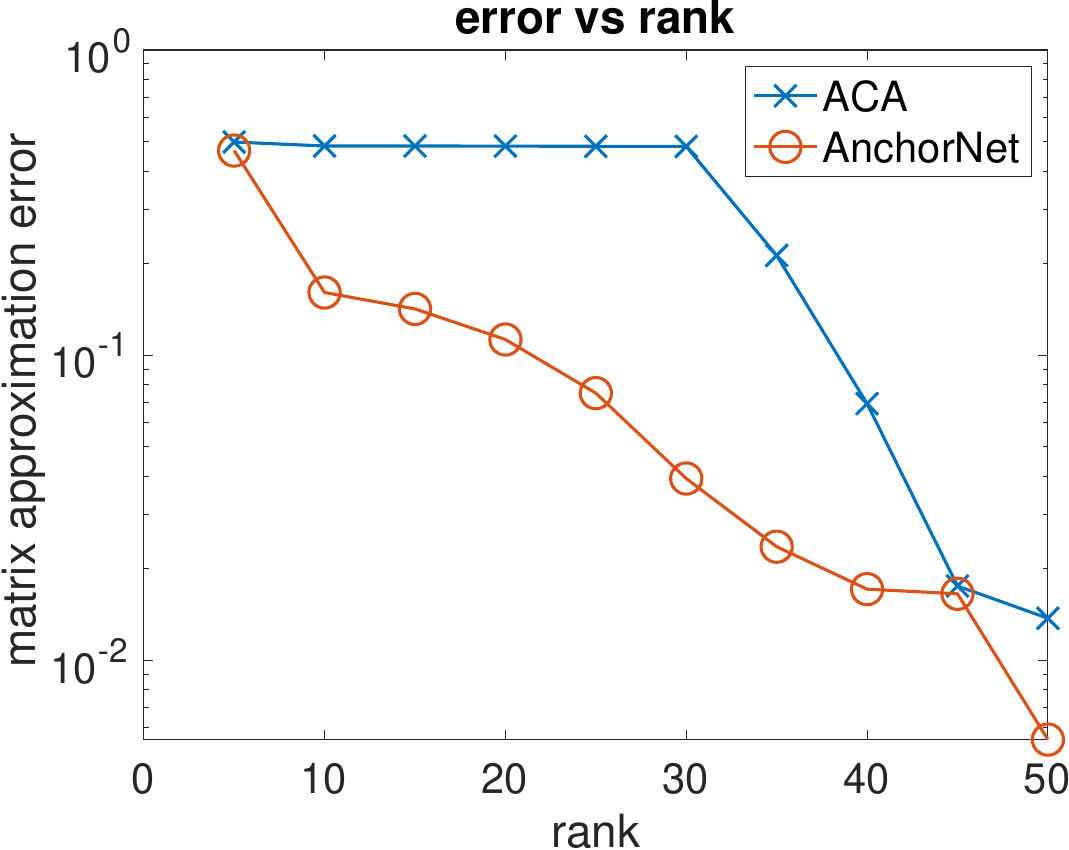}}
\hspace{0.1in}
    \subfigure[$X$ and 25 points selected by ACA (error=0.48)]{\includegraphics[scale=.3]{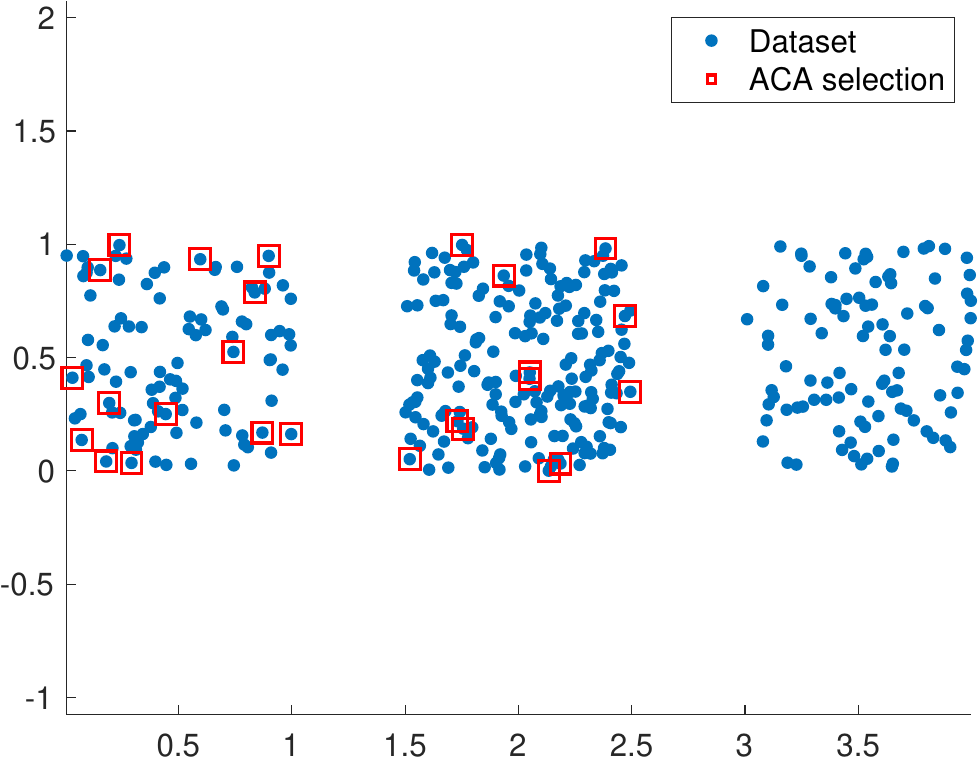}}
\hspace{0.1in}
    \subfigure[$X$ and 25 points selected by AnchorNet (error=0.07)]{\includegraphics[scale=.3]{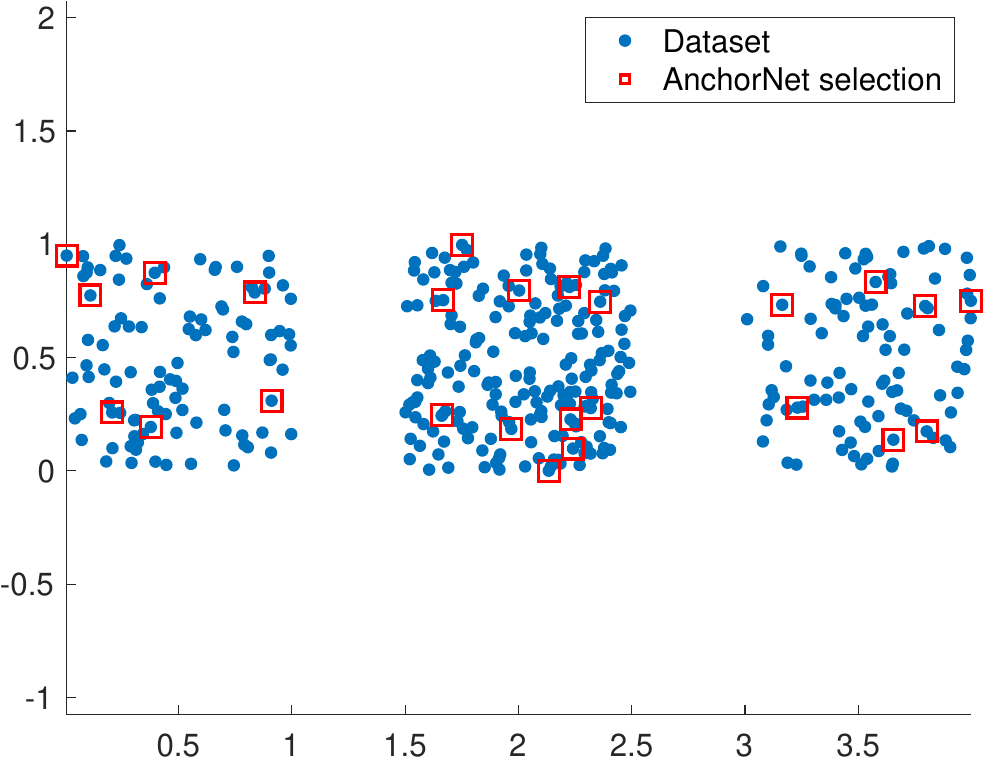}}
    \caption{\cdf{Comparison of ACA-based selection and AnchorNet-based selection.}}
    \label{fig:ACA-ANC}
\end{figure}
}

\subsection{Complexity test} 
\label{sub:Complexity}
In this section, \cdf{we perform experiments to investigate the complexity} of the proposed data-driven methods with respect to the size of the data.

\textbf{Test 6. Linear complexity with respect to data size.}
We consider approximating an $n$-by-$n$ kernel matrix $K_{XY}$ with increasing matrix size $n$ and dimension $d$.
We consider dimensions $d=3,10,50,100$ and generate synthetic data $X, Y$ in $\mathbb{R}^d$.
$X$ and $Y$ are randomly sampled from the uniform distribution over $[0,1]^d$ and $[2,3]^d$, respectively.
The kernel function is chosen as $\log |x-y|$.
We use the one-sided factorization in Algorithm \ref{alg:fac1} based on farthest point sampling.
In Figure \ref{fig:memory-time}, we report the peak memory use and timing for our method as $n$ increases.
The CPU time is computed as an average over ten repeated runs.
For all cases in Figure \ref{fig:memory-time}, the relative low-rank approximation error is around $2\times 10^{-4}$.

It is easily seen from Figure \ref{fig:memory-time} that, for each dimension $d$, the peak memory and timing both increase approximately linearly as a function of $n$, i.e. the number of points in $X$ or $Y$.

\begin{figure}[htbp] 
    \centering 
    \includegraphics[scale=.35]{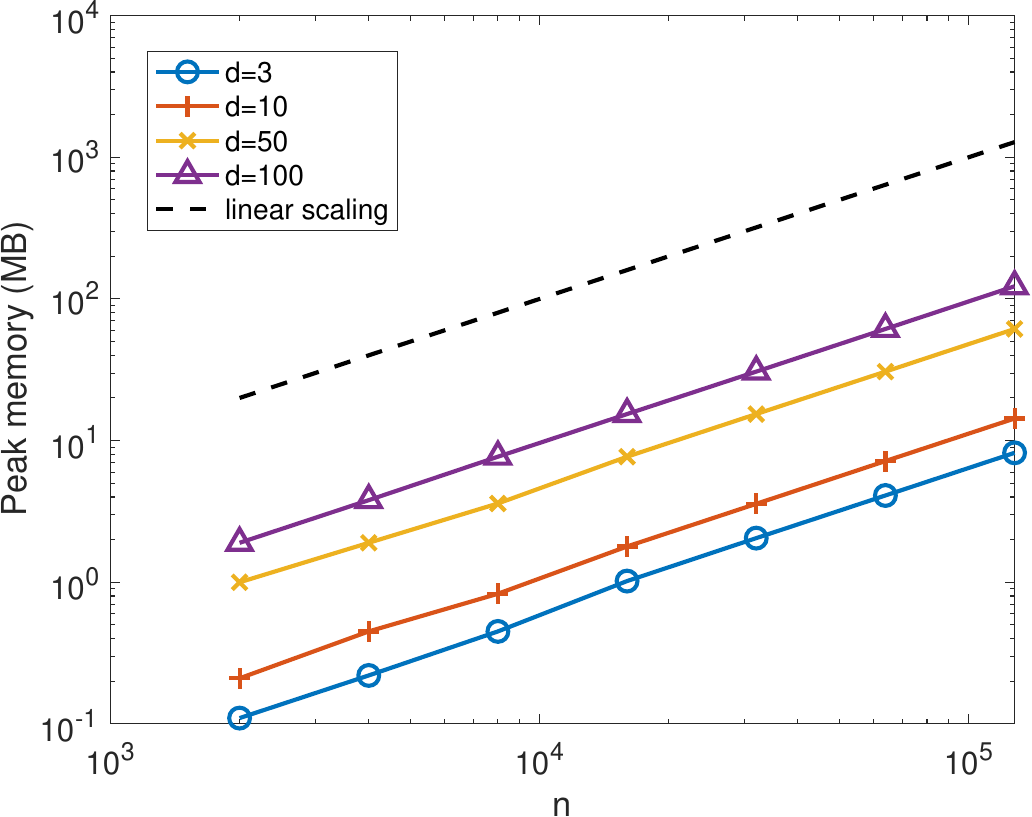} 
    \includegraphics[scale=.35]{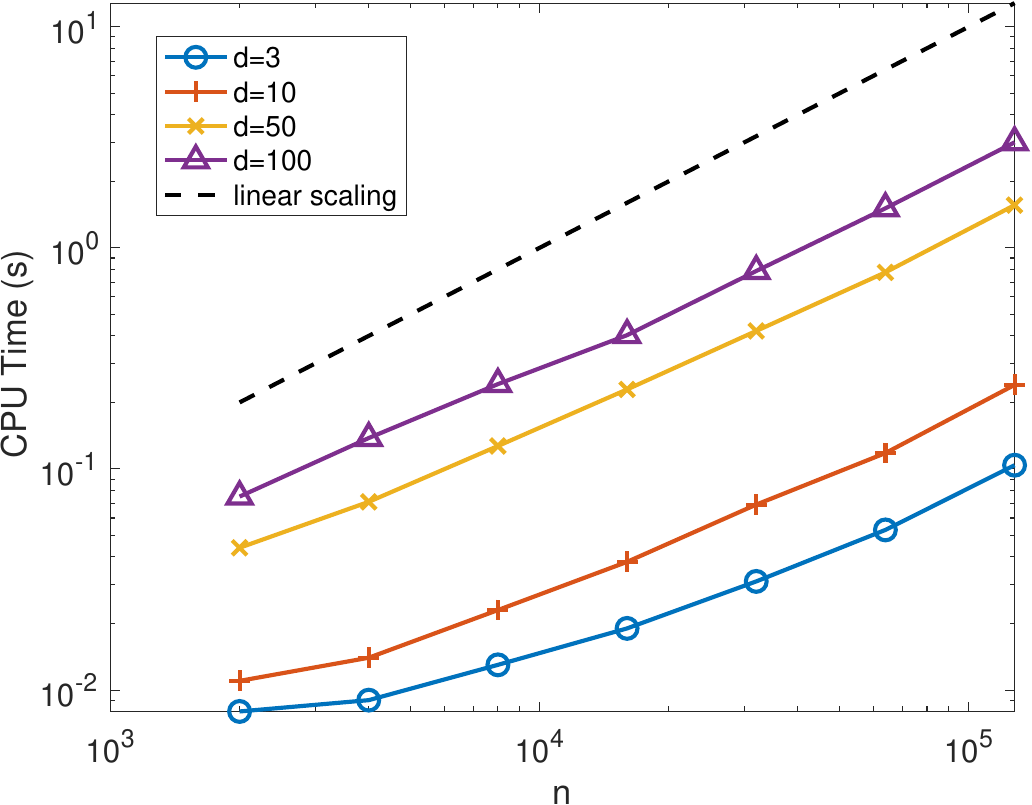} 
    \caption{Linear scaling tests of the data-driven factorization method on the kernel matrices defined by the synthetic data $X, Y$ sampled from the uniform distribution over $[0,1]^d$ and $[2,3]^d$, respectively, and the kernel function $\log |x-y|$. Left: peak memory use. Right: CPU time (average of 10 runs).} 
    \label{fig:memory-time} 
\end{figure}

\subsection{Kernel test} 
\label{sub:Kernel}
To show that the proposed data-driven approach can be applied to various kinds of kernel functions defined in high dimensions, we consider six different kernel functions in Table \ref{tab:kernel} and a dataset in 561 dimensions.
\cdf{The kernel $\kappa_5(x,y)$ in Table \ref{tab:kernel} is non-symmetric.}
We demonstrate the generality and accuracy of the data-driven approach by comparing to ACA. 

\textbf{Test 7. Approximating various types of kernel functions.}
We use a high dimensional dataset from the UCI Machine Learning Repository:
\emph{Smartphone Dataset for Human Activity Recognition in Ambient Assisted Living}\footnote{\url{https://archive.ics.uci.edu/ml/machine-learning-databases/00364/}}.
The training data contains $n=4252$ instances with $d=561$ attributes.
We choose $X$ to be the standardized data,
and define $Y = \frac{2R}{\sqrt{d}}+X$, where $R=\max\limits_{x\in X}\dist{x,0}$.
By construction, $X$ and $Y$ do not overlap and kernel functions in Table \ref{tab:kernel} are all well-defined over $X\times Y$.

The kernel functions $\kappa_1(x,y),\dots,\kappa_6(x,y)$ in the experiment are given in Table \ref{tab:kernel}.

In Table \ref{tab:2},
we report the relative approximation error
$||K-UV||/||K||$  with respect to the approximation rank $r$, which is equal to the number of columns of $U$.
Three methods are compared: ACA, the data-driven compression in Algorithm \ref{alg:fac1} with farthest point sampling (``DD-FPS'') or anchor net method (``DD-ANC'').

We see from Table \ref{tab:2} that DD-ANC achieves the best result for all kernels and ranks tested. 
For the same approximation rank $r$, the accuracy of DD-ANC is noticeably higher than that of DD-FPS and ACA.
DD-FPS outperforms ACA for almost all cases except $\kappa_4(x,y)$ when $r=10,\dots,170$.
Together with Test $4$ in Section \ref{sub:high},
the results show that the proposed fast data-driven approach is not only more robust, but also more accurate for the high dimensional dataset with general kernels.
Compared to existing methods, one advantage of the data-driven method is that, for the same dataset and fixed rank $r$, the geometric selection is performed only \emph{once} and can be used for different kernel functions or kernel function parameters.
This can hardly  be achieved by methods that require kernel function evaluation as the first step of the compression.
For example, for algebraic methods such as ICA (Incomplete Cross Approximation \cite{ICA2000}) and ACA, if the kernel function changes, the pivots need to be computed anew.
In Table \ref{tab:2}, for each $r$, ACA computes pivots six times for six kernels, while DD-FPS and DD-ANC each only select one subset, which is used for all six kernel functions.


\begin{table}
\caption{Kernels used for experiment in Table \ref{tab:2}.  In $\kappa_4$, the constant $c=\frac{0.8}{\max\limits_{x\in X, y\in Y}|x-y|^2}$, and in $\kappa_5$, $x_1$ denotes the first entry of the vector $x\in\mathbb{R}^d$}
\label{tab:kernel}
\begin{center}
\begin{tabular}{cccccc}
\toprule
$\kappa_1(x,y)$ & $\kappa_2(x,y)$ & $\kappa_3(x,y)$ & $\kappa_4(x,y)$ & $\kappa_5(x,y)$ & $\kappa_6(x,y)$ \\
\midrule 
$|x-y|$ & $\log|x-y|$ & $\left(1+\left|\frac{x-y}{R}\right|^2\right)^{-1}$ & $\exp\left(-\frac{1}{1-c|x-y|^2}\right)$ & $\dfrac{x_1}{|x-y|}$ & $x\cdot y+(x\cdot y)^2 + (x\cdot y)^3$\\
\bottomrule 
\end{tabular}
\end{center}
\end{table}

\begin{table}
\caption{Rank-$r$ approximation accuracy of the data-driven factorizations (DD-ANC and DD-FPS) and ACA on
the kernel matrices defined by the Smartphone Dataset ($d=561$) and six kernel functions shown in Table \ref{tab:kernel}.}
\label{tab:2}
\begin{center}
\begin{tabular}{ccccccccc}
\toprule
\midrule
$r$ & & 10 & 50 & 90 & 130 & 170 & 210 & 250 \\
\midrule
\multirow{3}{*}{\centering $\kappa_1(x,y)$} 
& ACA & 2.4E-3 & 4.8E-4 & 1.2E-4 & 8.2E-5 & 3.1E-5 & 1.7E-5 & 8.5E-6 \\
\cmidrule{2-9}
& DD-FPS & 1.3E-3 & 2.4E-4 & 1.4E-4 & 4.2E-5 & 3.5E-5 & 1.4E-5 & 4.7E-6 \\
\cmidrule{2-9}
& DD-ANC & 3.5E-4 & 9.0E-5 & 3.8E-5 & 1.9E-5 & 9.9E-6 & 4.6E-6 & 2.7E-6 \\
\midrule
\multirow{3}{*}{\centering $\kappa_2(x,y)$} 
& ACA & 2.2E-4 & 6.1E-5 & 4.8E-5 & 2.3E-5 & 6.7E-6 & 3.6E-6 & 2.7E-6 \\
\cmidrule{2-9}
& DD-FPS & 1.7E-4 & 4.5E-5 & 2.7E-5 & 7.5E-6 & 5.1E-6 & 2.4E-6 & 1.1E-6 \\
\cmidrule{2-9}
& DD-ANC & 5.9E-5 & 1.6E-5 & 6.4E-6 & 3.6E-5 & 1.9E-6 & 9.0E-7 & 5.5E-7 \\
\midrule
\multirow{3}{*}{\centering $\kappa_3(x,y)$} 
& ACA & 2.5E-3 & 9.3E-4 & 3.2E-4 & 1.8E-4 & 6.3E-5 & 4.3E-5 & 2.9E-5 \\
\cmidrule{2-9}
& DD-FPS & 5.2E-3 & 9.8E-4 & 2.5E-4 & 1.2E-4 & 6.8E-5 & 3.8E-5 & 2.0E-5 \\
\cmidrule{2-9}
& DD-ANC & 8.3E-4 & 1.3E-4 & 6.4E-5 & 3.9E-5 & 1.8E-5 & 9.9E-6 & 6.0E-6 \\
\midrule
\multirow{3}{*}{\centering $\kappa_4(x,y)$} 
& ACA & 1.2E-2 & 8.7E-4 & 3.3E-4 & 1.4E-4 & 7.1E-5 & 4.9E-5 & 4.8E-5 \\
\cmidrule{2-9}
& DD-FPS & 3.6E-2 & 2.5E-3 & 3.9E-4 & 1.7E-4 & 1.0E-4 & 4.6E-5 & 1.8E-5 \\
\cmidrule{2-9}
& DD-ANC & 1.8E-3 & 2.8E-4 & 1.2E-4 & 6.2E-5 & 3.5E-5 & 1.7E-5 & 9.0E-6 \\
\midrule
\multirow{3}{*}{\centering $\kappa_5(x,y)$} 
& ACA & 9.1E-4 & 1.7E-4 & 7.7E-5 & 5.7E-5 & 2.2E-5 & 1.1E-5 & 7.1E-6 \\
\cmidrule{2-9}
& DD-FPS & 4.0E-4 & 1.1E-4 & 4.9E-5 & 2.1E-5 & 1.1E-5 & 4.1E-6 & 2.1E-6 \\
\cmidrule{2-9}
& DD-ANC & 2.9E-4 & 5.7E-5 & 3.2E-5 & 1.3E-5 & 5.4E-6 & 2.0E-6 & 1.2E-6 \\
\midrule
\multirow{3}{*}{\centering $\kappa_6(x,y)$} 
& ACA & 1.3E-1 & 8.9E-3 & 5.1E-3 & 2.8E-3 & 2.2E-3 & 1.8E-3 & 1.7E-3 \\
\cmidrule{2-9}
& DD-FPS & 1.8E-2 & 3.7E-3 & 2.0E-3 & 1.2E-3 & 9.8E-4 & 7.6E-4 & 6.3E-4 \\
\cmidrule{2-9}
& DD-ANC & 1.2E-2 & 2.7E-3 & 1.1E-3 & 7.5E-4 & 6.8E-4 & 5.4E-4 & 4.5E-4 \\
\midrule
\bottomrule
\end{tabular}
\end{center}
\end{table}


\section{Conclusion}
\label{sec:conclusion}

For compressing low-rank kernel matrices where sets of points $X$ and $Y$
are available, it appears appealing to use subsets of $X$ and $Y$ that capture the geometry of $X$ and $Y$.  
This paper presented theoretical
justification and numerical tests that argue for choosing points such
that no original point in $X$ (or $Y$) is very far from a point chosen
for the subset.  If these subsets can be selected in linear time,
then the overall compression algorithm can be performed in linear time,
which is optimal for kernel matrices.  
We demonstrated effective low-rank compression for both low and high dimensional datasets using geometric selection based on farthest point sampling and
the anchor net method, which are both linear scaling.  
It is possible that even more sophisticated linear scaling schemes for selecting subsets can lead to even better approximation accuracy with the same number of selected points, especially in the high-dimensional case.

\bibliography{cdfeng}

\begin{thebibliography}{10}

\bibitem{ainsworth2011confusion}
M.~Ainsworth and T.~Vejchodsk{\'y}.
\newblock Fully computable robust a posteriori error bounds for singularly
  perturbed reaction--diffusion problems.
\newblock {\em Numerische Mathematik}, 119(2):219--243, 2011.

\bibitem{anderson92}
C.~R. Anderson.
\newblock An implementation of the fast multipole method without multipoles.
\newblock {\em SIAM J. Sci. Statist. Comput.}, 13(4):923--947, 1992.

\bibitem{guCUR2015}
David Anderson, Simon Du, Michael Mahoney, Christopher Melgaard, Kunming Wu,
  and Ming Gu.
\newblock Spectral gap error bounds for improving {CUR} matrix decomposition
  and the {N}ystr{\"o}m method.
\newblock In {\em Artificial Intelligence and Statistics}, pages 19--27, 2015.

\bibitem{atkinson1967eig}
K.~E. Atkinson.
\newblock The numerical solution of the eigenvalue problem for compact integral
  operators.
\newblock {\em Transactions of the American Mathematical Society},
  129(3):458--465, 1967.

\bibitem{BarnesHut}
J.~{Barnes} and P.~{Hut}.
\newblock {A hierarchical O(N log N) force-calculation algorithm}.
\newblock {\em Nature}, 324:446--449, December 1986.

\bibitem{bebendorf2020}
M.~Bauer, M.~Bebendorf, and B.~Feist.
\newblock Kernel-independent adaptive construction of h 2-matrix
  approximations.
\newblock {\em Numerische Mathematik}, 150(1):1--32, 2022.

\bibitem{bebendorf2000}
M.~Bebendorf.
\newblock Approximation of boundary element matrices.
\newblock {\em Numer. Math.}, 86(4):565--589, 2000.

\bibitem{bebendorf2003ACA}
M.~Bebendorf and S.~Rjasanow.
\newblock Adaptive low-rank approximation of collocation matrices.
\newblock {\em Computing}, 70(1):1--24, 2003.

\bibitem{bishopbook}
Christopher~M. Bishop.
\newblock {\em Pattern Recognition and Machine Learning}.
\newblock Springer, 2006.

\bibitem{HCA2005}
S.~B{\"o}rm and L.~Grasedyck.
\newblock Hybrid cross approximation of integral operators.
\newblock {\em Numer. Math.}, 101(2):221--249, 2005.

\bibitem{hackintroh2app}
S.~B{\"o}rm, L.~Grasedyck, and W.~Hackbusch.
\newblock {Introduction to hierarchical matrices with applications}.
\newblock {\em Eng. Anal. Bound. Elem.}, 27(5):405--422, 2003.

\bibitem{braess2008}
D.~Braess and J.~Sch{\"o}berl.
\newblock Equilibrated residual error estimator for edge elements.
\newblock {\em Mathematics of Computation}, 77(262):651--672, 2008.

\bibitem{dcJCP2022}
D.~Cai.
\newblock Physics-informed distribution transformers via molecular dynamics and
  deep neural networks.
\newblock {\em Journal of Computational Physics}, 468:111511, 2022.

\bibitem{localL2}
D.~Cai and Z.~Cai.
\newblock A hybrid a posteriori error estimator for conforming finite element
  approximations.
\newblock {\em Computer Methods in Applied Mechanics and Engineering}, 339:320
  -- 340, 2018.

\bibitem{cai2020equi}
D.~Cai, Z.~Cai, and S.~Zhang.
\newblock Robust equilibrated a posteriori error estimator for higher order
  finite element approximations to diffusion problems.
\newblock {\em Numerische Mathematik}, 144(1):1--21, 2020.

\bibitem{smash}
D.~Cai, E.~Chow, L.~Erlandson, Y.~Saad, and Y.~Xi.
\newblock {SMASH}: Structured matrix approximation by separation and hierarchy.
\newblock {\em Numerical Linear Algebra with Applications}, 25(6):e2204, 2018.

\bibitem{autm}
D.~Cai, Y.~Ji, H.~He, Q.~Ye, and Y.~Xi.
\newblock {AUTM Flow: Atomic Unrestricted Time Machine for Monotonic
  Normalizing Flows}.
\newblock In {\em Proceedings of the Thirty-Eighth Conference on Uncertainty in
  Artificial Intelligence}, volume 180, pages 266--274. PMLR, 2022.

\bibitem{anchornet}
D.~Cai, J.~Nagy, and Y.~Xi.
\newblock Fast deterministic approximation of symmetric indefinite kernel
  matrices with high dimensional datasets.
\newblock {\em SIAM Journal on Matrix Analysis and Applications},
  43(2):1003--1028, 2022.

\bibitem{eigCMAM}
D.~Cai and P.~S. Vassilevski.
\newblock Eigenvalue problems for exponential-type kernels.
\newblock {\em Computational Methods in Applied Mathematics}, 20(1):61--78,
  2020.

\bibitem{darve2019}
L{\'e}opold Cambier and Eric Darve.
\newblock Fast low-rank kernel matrix factorization using skeletonized
  interpolation.
\newblock {\em SIAM Journal on Scientific Computing}, 41(3):A1652--A1680, 2019.

\bibitem{ID2005}
Hongwei Cheng, Zydrunas Gimbutas, Per-Gunnar Martinsson, and Vladimir Rokhlin.
\newblock On the compression of low rank matrices.
\newblock {\em SIAM Journal on Scientific Computing}, 26(4):1389--1404, 2005.

\bibitem{edmondgaussian2014}
E.~Chow and Y.~Saad.
\newblock Preconditioned {K}rylov subspace methods for sampling multivariate
  {G}aussian distributions.
\newblock {\em SIAM J. Sci. Comput.}, 36(2):A588--A608, 2014.

\bibitem{kressner2020}
A.~Cortinovis and D.~Kressner.
\newblock Low-rank approximation in the frobenius norm by column and row subset
  selection.
\newblock {\em SIAM Journal on Matrix Analysis and Applications},
  41(4):1651--1673, 2020.

\bibitem{FPS97}
Yuval Eldar, Michael Lindenbaum, Moshe Porat, and Yehoshua~Y Zeevi.
\newblock The farthest point strategy for progressive image sampling.
\newblock {\em IEEE Transactions on Image Processing}, 6(9):1305--1315, 1997.

\bibitem{smashIPDPS}
L.~{Erlandson}, D.~{Cai}, Y.~{Xi}, and E.~{Chow}.
\newblock Accelerating parallel hierarchical matrix-vector products via
  data-driven sampling.
\newblock In {\em 2020 IEEE International Parallel and Distributed Processing
  Symposium (IPDPS)}, pages 749--758, 2020.

\bibitem{Gillman2012}
Adrianna Gillman, Patrick~M. Young, and Per-Gunnar Martinsson.
\newblock A direct solver with o(n) complexity for integral equations on
  one-dimensional domains.
\newblock {\em Frontiers of Mathematics in China}, 7(2):217--247, 2012.

\bibitem{maxvol2001}
Sergei~A Goreinov and Eugene~E Tyrtyshnikov.
\newblock The maximal-volume concept in approximation by low-rank matrices.
\newblock {\em Contemporary Mathematics}, 280:47--52, 2001.

\bibitem{skeleton2011}
Sergei~A Goreinov and Eugene~E Tyrtyshnikov.
\newblock {Quasioptimality of skeleton approximation of a matrix in the
  Chebyshev norm}.
\newblock In {\em Doklady Mathematics}, volume~83, pages 374--375, 2011.

\bibitem{GREENGARD1997280}
L.~Greengard and V.~Rokhlin.
\newblock A fast algorithm for particle simulations.
\newblock {\em J. Comput. Phys.}, 73:325--348, 1987.

\bibitem{rrqr96}
M.~Gu and S.~C. Eisenstat.
\newblock Efficient algorithms for computing a strong rank-revealing {QR}
  factorization.
\newblock {\em SIAM J. Sci. Comput.}, 17(4):848--869, 1996.

\bibitem{hack2015book}
W.~Hackbusch.
\newblock {\em Hierarchical Matrices: Algorithms and Analysis}.
\newblock Springer Series in Computational Mathematics. Springer Berlin
  Heidelberg, 2015.

\bibitem{h2lec}
W.~Hackbusch, B.N. Khoromskij, and S.A. Sauter.
\newblock {On $\mathcal{H}^2$-matrices}.
\newblock In Hans-Joachim Bungartz, Ronald H.~W. Hoppe, and Christoph Zenger,
  editors, {\em Lectures on Applied Mathematics}, pages 9--29. Springer,
  Berlin, 2000.

\bibitem{hack1989}
Wolfgang Hackbusch and Zenon~Paul Nowak.
\newblock On the fast matrix multiplication in the boundary element method by
  panel clustering.
\newblock {\em Numerische Mathematik}, 54(4):463--491, 1989.

\bibitem{huan2019}
H.~He, J.~Henderson, and J.C. Ho.
\newblock Distributed tensor decomposition for large scale health analytics.
\newblock In {\em The World Wide Web Conference}, pages 659--669, 2019.

\bibitem{huan2020fast}
H.~He, Y.~Xi, and J.C. Ho.
\newblock Fast and accurate tensor decomposition without a high performance
  computing machine.
\newblock In {\em 2020 IEEE International Conference on Big Data (Big Data)},
  pages 163--170. IEEE, 2020.

\bibitem{huan2018}
J.~Henderson, H.~He, B.A. Malin, J.C. Denny, A.N. Kho, J.~Ghosh, and J.C. Ho.
\newblock Phenotyping through semi-supervised tensor factorization (psst).
\newblock In {\em AMIA Annual Symposium Proceedings}, volume 2018, page 564.
  American Medical Informatics Association, 2018.

\bibitem{hsiaowendlandbook}
G.~C. Hsiao and W.~L. Wendland.
\newblock {\em Boundary integral equations}.
\newblock Applied Mathematical Sciences. Springer, Berlin, Heidelberg, 2008.

\bibitem{kress2013linear}
R.~Kress.
\newblock {\em Linear Integral Equations}.
\newblock Applied Mathematical Sciences. Springer New York, 2013.

\bibitem{UD2012book}
L.~Kuipers and H.~Niederreiter.
\newblock {\em Uniform distribution of sequences}.
\newblock Courier Corporation, 2012.

\bibitem{seeger2002SGP}
Neil Lawrence, Matthias Seeger, and Ralf Herbrich.
\newblock Fast sparse gaussian process methods: The informative vector machine.
\newblock {\em Advances in neural information processing systems}, 15, 2002.

\bibitem{filldistance}
WR~Madych and SA~Nelson.
\newblock Bounds on multivariate polynomials and exponential error estimates
  for multiquadric interpolation.
\newblock {\em Journal of Approximation Theory}, 70(1):94--114, 1992.

\bibitem{Mahoney697}
Michael~W. Mahoney and Petros Drineas.
\newblock {CUR} matrix decompositions for improved data analysis.
\newblock {\em Proceedings of the National Academy of Sciences},
  106(3):697--702, 2009.

\bibitem{martinsson_tropp_2020}
Per-Gunnar Martinsson and Joel~A. Tropp.
\newblock Randomized numerical linear algebra: Foundations and algorithms.
\newblock {\em Acta Numerica}, 29:403–572, 2020.

\bibitem{yuji20}
Yuji Nakatsukasa.
\newblock Fast and stable randomized low-rank matrix approximation.
\newblock {\em arXiv preprint arXiv:2009.11392}, 2020.

\bibitem{QMC1992book}
H.~Niederreiter.
\newblock {\em Random number generation and quasi-Monte Carlo methods}.
\newblock SIAM, 1992.

\bibitem{FPS06}
Gabriel Peyr{\'e} and Laurent~D Cohen.
\newblock Geodesic remeshing using front propagation.
\newblock {\em International Journal of Computer Vision}, 69(1):145--156, 2006.

\bibitem{rokhlinpotential}
V.~Rokhlin.
\newblock Rapid solution of integral equations of classical potential theory.
\newblock {\em Journal of Computational Physics}, 60(2):187--207, 1985.

\bibitem{FPS11}
Thomas Schl{\"o}mer, Daniel Heck, and Oliver Deussen.
\newblock Farthest-point optimized point sets with maximized minimum distance.
\newblock In {\em Proceedings of the ACM SIGGRAPH Symposium on High Performance
  Graphics}, pages 135--142, 2011.

\bibitem{smola2000SGP}
Alex Smola and Peter Bartlett.
\newblock Sparse greedy gaussian process regression.
\newblock {\em Advances in neural information processing systems}, 13, 2000.

\bibitem{snelson2005SGP}
Edward Snelson and Zoubin Ghahramani.
\newblock Sparse gaussian processes using pseudo-inputs.
\newblock {\em Advances in neural information processing systems}, 18, 2005.

\bibitem{xiaobaiFMM}
X.~Sun and N.P. Pitsianis.
\newblock A matrix version of the fast multipole method.
\newblock {\em SIAM Rev.}, 43(2):289--300, 2001.

\bibitem{Tyrtyshnikov1996}
E.~Tyrtyshnikov.
\newblock Mosaic-skeleton approximations.
\newblock {\em CALCOLO}, 33(1):47--57, Jun 1996.

\bibitem{ICA2000}
E.~Tyrtyshnikov.
\newblock Incomplete cross approximation in the mosaic-skeleton method.
\newblock {\em Computing}, 64(4):367--380, Jun 2000.

\bibitem{vapnikbook}
Vladimir Vapnik.
\newblock {\em The Nature of Statistical Learning Theory}.
\newblock Springer, 2013.

\bibitem{verf1994}
R.~Verf{\"u}rth.
\newblock A posteriori error estimation and adaptive mesh-refinement
  techniques.
\newblock {\em Journal of Computational and Applied Mathematics}, 50(1):67 --
  83, 1994.

\bibitem{verf2005confusion}
R.~Verf{\"u}rth.
\newblock Robust a posteriori error estimates for stationary
  convection-diffusion equations.
\newblock {\em SIAM Journal on Numerical Analysis}, 43(4):1766--1782, 2005.

\bibitem{nys2001}
Christopher~KI Williams and Matthias Seeger.
\newblock {Using the Nystr{\"o}m method to speed up kernel machines}.
\newblock In {\em Advances in Neural Information Processing Systems}, pages
  682--688, 2001.

\bibitem{xing2020interpolative}
X.~Xing and E.~Chow.
\newblock Interpolative decomposition via proxy points for kernel matrices.
\newblock {\em SIAM Journal on Matrix Analysis and Applications},
  41(1):221--243, 2020.

\bibitem{darveEndo}
Zixi Xu, L{\'e}opold Cambier, Fran{\c{c}}ois-Henry Rouet, Pierre L'Eplatennier,
  Yun Huang, Cleve Ashcraft, and Eric Darve.
\newblock Low-rank kernel matrix approximation using skeletonized interpolation
  with endo-or exo-vertices.
\newblock {\em arXiv preprint arXiv:1807.04787}, 2018.

\bibitem{ZZ1987}
O.C. Zienkiewicz and J.Z. Zhu.
\newblock A simple error estimator and adaptive procedure for practical
  engineerng analysis.
\newblock {\em International Journal for Numerical Methods in Engineering},
  24(2):337--357, 1987.

\end{thebibliography}
\bibliographystyle{plain}
\end{document}